\documentclass
[12pt]
{article}

\usepackage{amsmath}
\usepackage{amssymb}
\usepackage{amsthm}

\newcommand{\ga}{\gamma}
\newcommand{\e}{\varepsilon}
\newcommand{\la}{\lambda}

\newtheorem{theorem}{Theorem}[section]
\newtheorem{lemma}[theorem]{Lemma}
\newtheorem{remark}[theorem]{Remark}

\numberwithin{equation}{section}
\usepackage{color}
\usepackage[nottoc,notlot,notlof]{tocbibind}
\begin{document}
\title{
Infinite concentration and oscillation estimates for supercritical semilinear elliptic equations in discs. I
}
\author{Daisuke Naimen
%\thanks{%Muroran Institute of Technology, 27-1, Mizumoto-cho, Muroran-shi, Hokkaido, 0508585, Japan, E-mail address: naimen@muroran-it.ac.jp\\
%E-mail address: naimen@muroran-it.ac.jp,
% Key words: supercritical  semilinear elliptic equations, exponential nonlinearities, blow-up analysis, concentration and oscillation, MSC 2020: 35B44, 35B40,  35J91, 35A21, 35B05, 35B32
%35A21 singularity, 
%35B05 oscillation, 
%35B09 positive sol, 
%35B32 bifurcation, 
%%35B33 critical exp, 
%%35B38 criti point, 
%35B40 asymptotic behav, 
%35B44 blow-up, 
%%35J08 green func, 
%%35J15 second order elliptic, 
%%35J25 BVP of 2nd order elliptic, 
%35J61% Semi elliptic, 
%%35J91  Semi elliptic with Laplav
%}
}
%\date{}
\date{\small{Muroran Institute of Technology, 27-1, Mizumoto-cho, Muroran-shi, Hokkaido, 0508585, Japan}}
\maketitle
\begin{abstract} 
In our series of papers, we establish infinite concentration and 
oscillation estimates for supercritical semilinear elliptic equations in discs. Especially, 
we extend the previous result by the author (N.  arXiv:2404.01634) 
to the general supercritical  case. Our growth condition is related 
to the one introduced by Dupaigne-Farina (J. Eur. Math. Soc. 12: 855--882, 2010)  and admits 
two types of supercritical nonlinearities, the Trudinger-Moser type growth 
$e^{u^p}$ with $p>2$ and the multiple exponential one 
$\exp{(\cdots(\exp{(u^m)})\cdots)}$ with $m>0$. In this first part, we 
carry out the analysis of infinite concentration phenomena on any blow-up 
solutions. As a result, we classify all the infinite concentration 
behaviors into two types which in particular shows a new 
behavior for the multiple exponential case.  More precisely, we detect 
an infinite sequence of concentrating parts on any blow-up solutions 
via the scaling and pointwise techniques. The precise description of 
the limit profile, energy, and position of each concentration is 
given  via the Liouville equation with two types of 
the energy recurrence formulas.  This leads us to observe 
two types of  supercritical behaviors. The behavior for the latter 
growth is new and can be understood as the limit 
case of the former one.  Our concentration estimates lead to 
the analysis of  infinite oscillation phenomena, including infinite oscillations of 
bifurcation diagrams, discussed in the second part. 
\end{abstract}
\tableofcontents
\section{Introduction}\label{sec:intr}
We consider the next problem,
\begin{equation}\label{p0}
\begin{cases}
-\Delta u=\la f(x,u), \ u>0&\text{ in }\Omega,\\
u=0&\text{ on }\partial \Omega,
\end{cases}
\end{equation}
where $\Omega\subset \mathbb{R}^2$ is a bounded domain with smooth boundary $\partial \Omega$, $\la>0$ a parameter, $f:\bar{\Omega}\times [0,\infty)\to \mathbb{R}$ a nonnegative continuously differentiable function. Moreover we impose a generalized exponential growth condition at infinity on $f$. This will be mentioned more precisely in Subsection \ref{sub:gr} below. Typical examples are functions with the Trudinger-Moser type growth $e^{u^p}$ with   $p>1$ and also the multiple exponential one $\exp{(\cdots(\exp{(u^m)})\cdots)}$ with   $m>0$.  In our main argument, we assume $\Omega$ is a disc and investigate the infinite concentration and oscillation behavior of any blow-up solutions of \eqref{p0} in the supercritical case. 

For several decades, \eqref{p0} with exponential growth focuses various attentions. We first notice that if $f(x,t)=e^t$, \eqref{p0} is knows as the Gelfand problem \cite{G}. The complete classification of the structure of the solutions set is given by Joseph-Lundgren \cite{JL} for the case $\Omega$ is an $N$-dimensional ball with $N\ge1$. They show that the shape of the bifurcation diagram is strongly affected by the dimension $N$. Indeed, one sees that in the case $1\le N\le2$, the diagram, which emanates from the trivial solution $(\la,u)=(0,0)$, has a unique turning point and goes to the point $(0,\infty)$ at infinity. On the other hand, in the case $3\le N\le9$, one observes that it oscillates  around the axis $\la=\la^*$ infinitely many times for some value $\la^*>0$. Note that this ensures the existence of infinitely many solutions $u$ for $\la=\la^*$  and many solutions near $\la=\la^*$. Finally, in the case $N\ge 10$, one finds that it has no turning point and goes to the point $(\la^*,\infty)$ at infinity for a number $\la^*>0$.  More recently, several extensions to general supercritical cases are given in the case $N\ge3$. See \cite{Mi2}, \cite{KW}, and \cite{MN} for results on balls and  \cite{FZ} on convex domains. We remark that the supercritical problem in dimension two was not treated in these works. This is one of our main aims in this series.

Another interesting case is $f(x,t)=h(x,t)e^{t^2}$ with a suitable perturbation term $h(x,t)$. In this case, \eqref{p0} becomes the critical variational problem due to the Trudinger-Moser inequality by \cite{T} and \cite{M}. In view of this, one of the most important discussions lies in the analysis of the non-compact  phenomena appearing as concentration behaviors of solutions and Palais-Smale sequences. The existence of solutions are obtained, for example, in \cite{A} via  the variational method together with the concentration-compactness principle by Lions \cite{Lions}. The existence of multiple concentrating solutions is studied in \cite{DMR}. On the other hand, the classification of the concentration behavior of solutions is done by several authors \cite{AS}, \cite{AD}, \cite{D}, and \cite{DT}. In particular, in \cite{D},  Druet proves that any energy-bounded sequence of blow-up solutions exhibits the finite energy quantization phenomenon. More precisely, he shows that it is decomposed by at most finitely many concentrating parts which, after an appropriate scaling, converge to a regular solution of the Liouville equation \cite{L} 
\begin{equation}\label{eq:lv}
-\Delta U=e^U\text{ in }\mathbb{R}^2\ \text{ with }\int_{\mathbb{R}^2}e^Udx<\infty.
\end{equation}
This observation with the classification result by \cite{CL} leads to the explicit quantification of the blow-up levels. The concentration analysis in this direction plays important roles in the variational analysis of  critical problems. See \cite{DMMT}.  %with the aid of the classification result by \cite{CL}.  % where one observes the energy quantization phenomena on concentrating solutions via the scaling argument.  

In addition,  we refer the reader to \cite{FMR}, \cite{OS},  and \cite{DM} for the existence and asymptotic behavior of solutions for the general situation $0<p\le 2$. 

Now, let us proceed to the discussion on the supercritical case which is the main subject of our series. Let $f(x,t)=h(x,t)e^{t^p}$ and $p>2$ where $h(x,t)$ is again a suitable perturbation term.  We find a few earlier works and some very recent results in discs.  The basic existence result and asymptotic formulas  are obtained by Atkinson-Peletier in \cite{AtPe} where, for the supercritical case, the exact behavior was left open. See Theorem 3 there and  Example 1 under it. After that, McLeod-McLeod \cite{McMc} investigate the asymptotic behavior of the parameter $\la$ for blow-up solutions. In addition, an interesting  note is that based on heuristic and qualitative observations, they suggest that large solutions exhibit  ``the bouncing process" which closely relates to our main observations  in this series. (We shall clarify the observations in these earlier works  from the microscopic and more general points of view.) 

After decades, very recently,  new developments are established by some authors \cite{N1}, \cite{N2}, \cite{FIRT}, \cite{K}.  The author, in \cite{N1} and \cite{N2}, accomplishes the concentration and oscillation analysis of blow-up solutions via the scaling and pointwise techniques developed  by Druet \cite{D}. In his result, one observes  that any blow-up solutions exhibit remarkably different behavior from those in the critical and subcritical cases.  In fact, he detects an infinite sequence of concentrating parts on  any blow-up solutions.  Moreover, he proves that the profiles and energies of the first and subsequent concentration parts are described by the regular and singular  solutions of the Liouville equation \eqref{eq:lv}
%\[
%-\Delta U=e^U\text{ in }\mathbb{R}^2\setminus\{0\}\ \text{ with }\int_{\mathbb{R}^2}e^Udx<\infty
%\]
respectively. We remark that those limit solutions are explicitly determined with the aid of the energy recurrence formulas \eqref{b1} and \eqref{b2} below. Then,  based on this precise analysis, he proves  that the infinite sequence of concentrating parts breaks the uniform boundedness of the ``energy" which is usually assumed or proved in the subcritical and critical cases. Moreover, he shows that the infinite concentration  causes infinite oscillations  of blow-up solutions  around singular solutions. Then, noting this fact with the idea by \cite{Mi}, he  arrives at a proof of infinite oscillations of bifurcation  diagrams which lead to  the existence of infinitely many solutions of \eqref{p0}. 

On the other hand, \cite{FIRT} and \cite{Ku} investigate problems with the generalized exponential growth which permits not only the Trudinger-Moser type  but also multiple exponential one.  In \cite{FIRT}, the authors prove the existence of singular solutions of \eqref{p0} with the precise asymptotic formulas. In \cite{Ku}, the author %investigates the asymptotic behavior of the parameter $\la$ and also the convergence to singular solutions of blow-up solutions. Moreover, he also
 shows the uniform boundedness of  finite Morse index solutions which leads to a proof of that the bifurcation diagram has infinitely many turning points under his analytic setting. We will see that the interaction between these two works and our present series gives several interesting  remarks.    

Now, a natural question is that whether one can accomplish the author's concentration and oscillation analysis in \cite{N1} and \cite{N2} for the generalized exponential growth, especially for the multiple exponential case. Before answering this, let us recall a suggestive observation in Section 6 in \cite{N2}. The author examines the blow-up behavior of solutions of  \eqref{p0} with $f(x,t)=h(t)e^{t^p}$ in the limit  $p\to \infty$. Interestingly, he observes that the infinite concentration and oscillation structure remains even for this drastic limit. Moreover, he finds another   system of the recurrence formulas \eqref{b3} and \eqref{b4} below as the limit formulas of \eqref{b1} and \eqref{b2}. This leads us to expect that a new type concentration and oscillation behavior,  controlled by \eqref{b3} and \eqref{b4}, may appear for problems with much higher growth  at infinity than the Trudinger-Moser type.

Noting these facts and observations, the main aim of this series is to extend the concentration and oscillation estimates in \cite{N1} and \cite{N2} to the generalized exponential case. To this end, we employ  a setting related to the generalized H\"older conjugate exponent introduced by Dupaigne-Farina \cite{DF}. More precisely, we set $f(x,t)=h(x)e^{g(t)}$ and impose
\[
\lim_{t\to \infty}\frac{g'(t)^2}{g(t)g''(t)}=q\ \text{ and }\ \lim_{t\to \infty}\frac{tg'(t)}{g(t)}=p
\] 
for some values $q\in[1,\infty)$ and $p\in (1,\infty]$. More precise conditions are given in Subsection \ref{sub:gr} below. As mentioned there, the cases $q>1$ and $q=1$ correspond to the Trudinger-Moser type growth $e^{g(t)} \sim e^{t^p}$ and  multiple exponential one  $e^{g(t)} \sim\exp{(\cdots(\exp{(t^m)}))}$ with $m>0$ respectively. Moreover, one observes that the relation $1/p+1/q=1$ holds and thus, we can interpret $p$ and $q$ as the growtn rate  of $g$ at infinity and its conjugate respectively. One clearly sees that the case $q=1$ is new which  can not be treated under the framework of \cite{N1} and \cite{N2}.  Moreover, even for the case $q>1$, our condition admits wider classes of nonlinearities, for example, $g(t)=t^p(\log{t})^l$ with any given number $l\in\mathbb{R}$. Under this generalized setting, we shall accomplish the   extension of the author's previous argument in \cite{N1} and \cite{N2} and find the new behavior for the case $q=1$.

Consequently, in this first part, we succeed in detecting an infinite sequence of concentrating parts on any blow-up solutions under our generalized setting. Furthermore, we give the precise characterization of the limit profile, energy, and position of each concentrating part via the Liouville equation with the two types of the energy recurrence formulas mentioned in Subsection \ref{sub:bf} below. Here we classify all the inifinite concentration behaviors into two types  depending on the choice of the growth condition $q\in(1,2)$ and $q=1$. In fact, we prove that in the former case, the infinite concentration behavior is controlled by \eqref{b1} and \eqref{b2} and thus, we may say it essentially coincides with the one  observed in the previous work \cite{N1}. On the other hand,  in the later case, we show that the behavior is described by the limit formulas \eqref{b3} and \eqref{b4}. Hence we may conclude that blow-up solutions exhibit new behavior in this case. Here, we can observe the remarkable quantitative differences. Indeed, we find that  every concentrating part appears at a much higher (the almost highest) region than that in the former case. In particular, every two successive concentrating parts get much closer to each other  and the ratio of their heights  is almost equal to one in contrast to the behavior in the former case. We also notice that our characterization via the two types of the energy recurrence formulas suggests that the latter behavior can be regarded as the limit behavior of the former one which implies a continuous relation between the two supercritical phenomena  in the cases $q\in(1,2)$ and $q=1$.  See our main results Theorems \ref{th:sup1} and \ref{th:sup2} below.

Moreover, in the second part \cite{N3}, we apply our concentration estimates to the study  of infinite oscillation phenomena. We there first give a precise description of the asymptotic shapes of the graphs of blow-up solutions near the origin. This shows that the infinite sequence of concentrating parts produces an infinite sequence of  bumps on the graphs of blow-up solutions around the origin (which is consistent with the bouncing process in \cite{McMc} noted above). Moreover, thanks to our explicit concentration estimates, we can give the precise estimates for the heights of the tops and the bottoms of the bumps. This enables us to show that any blow-up solutions oscillate infinitely many times around singular solutions with a suitable asymptotic behavior near the origin. Then, with the idea from \cite{Mi}, we observe that this infinite oscillation around the origin causes several infinite boundary oscillations. This leads to a proof of infinite oscillations of bifurcation diagrams with the existence of infinitely many solutions for certain classes of nonlinearities including the two types mentioned above. In this way, we complete the desired generalization of the results in \cite{N1} and \cite{N2}. In addition, the above discussion provides an interesting global picture of our concentration and oscillation phenomena. It shows that the infinite boundary oscillations generate an infinite number of bumps one after another from the boundary and then, all of them climb up the graphs of blow-up solutions from the bottom to the top and finally, due to the supercritical growth, they grow up as the infinite sequence of bubbling parts around the origin which we detect via the scaling procedure in this paper.   See the discussion in the second part for the detail.

A novelty of the present paper is that we classify all the infinite concentration behaviors of blow-up solutions with the precise characterization under the general supercritical setting. This completes the generalization of the result in \cite{N1} and shows a new phenomenon for the multiple exponential case $q=1$.  This provides new knowledge on the study of blow-up and concentration phenomena on \eqref{p0} extensively carried out  in the literature.

Another novelty is that the result in the present paper can be applied to the study of infinite oscillation phenomena in the second part \cite{N3} which generalizes the author's previous result in \cite{N2}. Here, concerning our results on infinite oscillations of bifurcation diagrams, our approach  can show not only the divergence of the number of turning points, proved in \cite{Ku}, but also the infinite oscillation around an axis $\la=\la^*$ for some suitable value $\la^*>0$ which ensures the existence of infinitely many solutions $u$ of \eqref{p0} for $\la=\la^*$. Hence, an advantage of our approach is that it leads to the Joseph-Lundgren type multiplicity result, mentioned in the second paragraph above, for general supercritical problems in dimension two. Moreover, another advantage is that it clarifies  the direct connection between the infinite bubbling phenomena on blow-up solutions and  infinite oscillation phenomena on  bifurcation diagrams. 

%In this way, our series provides clarifications of the drastic supercritical phenomena on \eqref{p0}. 
%We believe that results in our series provides fruitful contributions to the study of blow-up, concentration, and bifurcation phenomena on \eqref{p0}.  % from the view points of . %in the literature.
 In the following, let us start more precise discussions. 

\subsection{Setting} 
In the following, let $D$ be the unit disc centered at the origin and assume $\Omega=D$ in \eqref{p0}. Moreover, we suppose the next basic condition on $f$.
\begin{enumerate}
\item[(H0)]  $f(x,t)$ has the form $f(x,t)=h(|x|)f(t)$ with continuously differentiable functions $h:[0,1]\to (0,\infty)$ and $f:[0,\infty)\to [0,\infty)$. Moreover, there exists a value $t_0\ge0$ such that $f(t)>0$ for all $t\ge t_0$ and  $f\in C^2([t_0,\infty))$. Finally, we put $g(t)=\log{f(t)}$ for all $t\ge t_0$. 
\end{enumerate}
We always assume (H0) throughout this paper without further comments.  Here, note that if  $h\equiv 1$, from the result by \cite{GNN}, any solution $u$ in $C^2(\bar{D})$ of \eqref{p0} is radially symmetric.   %Hence our setting is reasonable. 
Hence, it is reasonable to consider any radially symmetric solutions of \eqref{p0}. Then we may regard each solution $u=u(|x|)$ for all $x\in \bar{D}$ and reduce \eqref{p0} to the following ordinary differential equation for $u=u(r)$,  
\begin{equation}\label{p}
\begin{cases}
- u''-\frac1r u'=\la h f(u), \ u>0&\text{ in }(0,1),\\
u'(0)=0=u(1).
\end{cases}
\end{equation}
In the following, we study solutions $u\in C^2([0,1])$ of \eqref{p}. The existence of blow-up solutions of \eqref{p} is ensured for a wide class of nonlinearities $f$ including supercritical case by the shooting argument. We refer the reader to  Theorem 3 in \cite{AtPe} and also Lemma 2.1 in \cite{AKG}. %For the  existence of such solutions, we refer the readers to Lemma 2.1 in \cite{AKG} with Theorem 2.1 in \cite{NT}.
 Moreover, by \eqref{p} with (H0), we easily check that every solution $u$ with $u(0)> t_0$ is strictly decreasing. Indeed, for any $r\in(0,1]$, multiplying the equation by $r$ and integrating over $[0,r]$, we get
\[
ru'(r)=-\int_0^r \la h(s)f(u(s))ds< 0.
\] 
In particular, we have that $\max_{r\in [0,1]}u(r)=u(0)$.
\subsection{Generalized exponential growth}\label{sub:gr} 
Let us introduce a generalized exponential growth condition on $f(t)=e^{g(t)}$ at infinity. To this end, we set two functions $Q(t)$ and $P(t)$ for all  $t\ge t_0$ by 
\[
Q(t)=\frac{g'(t)^2}{g(t)g''(t)}\text{\ \ and\ \ }\ P(t)=\frac{t g'(t)}{g(t)}.
\]
As we will see later, the former one can be regarded as the generalized H\"{o}lder conjugate exponent of the growth rate of $g$. This quantity is first introduced by Dupaigne-Farina \cite{DF}. Related quantities are used in recent works \cite{FIRT} and \cite{Ku} in discs. %The relationship among those conditions are noted in Remark 2.2 in \cite{FIRT}. 
 For our concentration and oscillation analysis, we assume the next condition (H1) on the limits of $Q$ and $P$. In the following, we set $\exp_1{(t)}=\exp{(t)}$ and for any $k\in \mathbb{N}$ with $k\ge2$,  $\exp_k(t)=\exp_{k-1}(\exp{(t)})$ for all $t\ge0$ by induction. 
\begin{enumerate}
\item[(H1)] We have the next (i) and (ii). 
\begin{enumerate}
\item[(i)] $g'(t)>0$ %$tg'(t)\to \infty$ as $t\to \infty$,
and $g''(t)>0$  for all $t\ge t_0$ and there exists a pair  $(q,p)\in \{1\}\times (0,\infty]\cup (1,\infty)\times (0,\infty)$ of values such that  
\begin{equation}\label{g1}
\lim_{t\to \infty}\frac{g'(t)^2}{g(t)g''(t)}=q
\end{equation}
and %the limit of $P(t)$ at infinity exists or diverges to infinity, say, 
\begin{equation}\label{g2}
\lim_{n\to \infty}\frac{t g'(t)}{g(t)}=p.
\end{equation}
\item[(ii)] In the case $q=1$, $tg'(t)/g(t)$ is nondecreasing for all $t\ge t_0$ and there exist a number $k\in \mathbb{N}$  and a function $\hat{g}\in C^2([t_0,\infty))$ such that $f(t)=\exp_k(\hat{g}(t))$ and $\hat{g}'(t)/\hat{g}(t)$ is nonincreasing 
for all $t\ge t_0$. 
\end{enumerate}
\end{enumerate}
The former condition (i) is the essential assumption.  As we will see in the examples below and Lemma \ref{lem:pq}, it ensures that $p$ is the growth rate of $g$ at infinity and $q$ the H\"{o}lder conjugate exponent of $p$, that is, $p>1$ and 
\[
\frac1p+\frac1q=1
\]
where we regarded $1/\infty=0$. %In particular, we have that $g(t)\to \infty$ and $g'(t)\to \infty$ as $t\to \infty$ as proved in Lemma \ref{lem:monog}.
 Typical examples for the case $1<q<\infty$ are given by the Trudinger-Moser type nonlinearities $f$  such like
\[
%g(t)= t^p(\log{r})^l\pm t^q(\log{r})^{l'}+m\log{t},
%f(r,t)=h(r)t^me^{t^p (\log{t})^l+c t^{\bar{p}}(\log{t})^{l'}},\ \ \ f(r,t)=h(r)t^me^{t^p (\log{t})^l+c t^{\bar{p}}(\log{t})^{l'}}
f(t)=t^me^{t^p+c t^{\bar{p}}},\ \ \ e^{t^p (\log{t})^l}
\]
for all large $t>0$ with any given constants $p>1$, $0<\bar{p}<p$, and $c,l,m\in \mathbb{R}$. We remark that the former one satisfies the condition (H1) in \cite{N1} while the latter one is a new case which can not be treated under the setting there.

 On the other hand,  $q=1$ yields that $g(t)$ grows much more rapidly than the power. In fact, a typical  example is  a multiple exponential function such as    
\[
f(t)=\exp_k{(t^m(\log{t})^l)}
\]
for all large $t>0$ with  any given values $k\in \mathbb{N}$ with $k\ge2$, $m>0$, and $l\in \mathbb{R}$. See Lemma \ref{lem:ex1}  and its remark below for useful sufficient conditions to check  (H1) with $q=1$. This is clearly a new case which can not be covered by the assumptions  in \cite{N1}. %Moreover, we will observe a new phenomenon in this case. 

From the facts and examples above, we notice that $q>2$, $q=2$, and $q\in[1,2)$ imply that $f$ has the subcritical, critical, and supercritical growths in dimension two respectively. We mainly interested in the case $q\in[1,2)$ throughout this series of papers.

 In addition, we remark on the condition (ii), which seems to be technical, for the case $q=1$. The monotonicity of $tg'(t)/g(t)$ is used for only Lemma \ref{lem:g4} below. The latter condition is applied for only Lemma \ref{lem:g2} (and its preliminaries Lemmas \ref{lem:g200} and \ref{lem:g20}). Its main consequence  is Lemma \ref{lem:g1} which will be applied when we deduce the limit equation of blow-up solutions. Noting this fact, we may employ the conclusion of Lemma \ref{lem:g2} as our (possibly weaker) condition instead of the latter one in (ii). 
We lastly  introduce the next standard condition.
\begin{enumerate}
\item[(H2)] We have that $\displaystyle\inf_{t>0}\frac{f(t)}{t}>0$. 
\end{enumerate} 
This ensures  an upper bound for the parameter $\la>0$. See Lemmas \ref{lem:kap} below.
\subsection{Energy recurrence formulas}\label{sub:bf}
Let us next introduce some sequences $(a_k)$, $(\delta_k)$, and $(\eta_k)$ of values which explicitly quantify the concentration and oscillation phenomena in the supercritical case. The former sequence describes the limit profiles and energies of all the concentrating parts while the latter two determine their heights. Those are defined via the next two systems of recurrence formulas. In this subsection, let $q\in[1,2)$ and $p\in (2,\infty]$ be any quantities such that $1/q+1/p=1$ where $1/\infty=0$.

First, we set $a_1=2$, $\delta_1=1$, and $\eta_1=1$ for any $q\in[1,2)$. Then in the case $q\in(1,2)$, we define the sequences $(a_k)\subset(0,2]$ and $(\delta_k)\subset (0,1]$ by the next formulas, 
\begin{equation}\label{b1}
\frac{2p}{2+a_k}\left(1-\frac{\delta_{k+1}}{\delta_k}\right)-1+\left(\frac{\delta_{k+1}}{\delta_k}\right)^p=0 \text{\ \  with \ \  }\delta_{k+1}<\delta_k
\end{equation}
and 
\begin{equation}\label{b2}
a_{k+1}=2-\left(\frac{\delta_{k+1}}{\delta_k}\right)^{p-1}(2+a_k)
\end{equation}
for all $k\in \mathbb{N}$ by induction. Moreover, we set $\eta_k=\delta_k^p$ for all $k\in \mathbb{N}$. On the other hand, if $q=1$, we determine the sequences $(a_k)\subset (0,2]$ and $(\eta_k)\subset (0,1]$ by the relations,  
\begin{equation}\label{b3}
\frac{2}{2+a_k}\log{\frac{\eta_{k}}{\eta_{k+1}}}-1+\frac{\eta_{k+1}}{\eta_{k}}=0 \text{\ \  with \ \ }\eta_{k+1}<\eta_k
\end{equation}
and 
\begin{equation}\label{b4}
a_{k+1}=2-\frac{\eta_{k+1}}{\eta_{k}}(2+a_k).
\end{equation}
In addition, we put $\delta_k=1$ for all $k\in \mathbb{N}$. As in Lemma \ref{lem:b1} below, for any $q\in[1,2)$, these sequences are well-defined. Moreover,  $(a_k)$ and $(\delta_k)$ for $q>1$ and $(a_k)$ and  $(\eta_k)$ for $q=1$ are strictly decreasing sequences which converge to zero as $k\to \infty$. Furthermore, we have that  $\sum_{k=1}^\infty a_k=\infty$. See Lemma \ref{lem:b2} below.  Lastly, for the convenience, for each $q\in[1,2)$,  we define a sequence $(\tilde{\eta}_k)$ of numbers so that  $\tilde{\eta}_k=\eta_k^{1/q}$  for all $k\in \mathbb{N}$.  

We remark that the former formulas \eqref{b1} and \eqref{b2} are first observed in \cite{McMc} as a result of the study of ``the bouncing process"  and later derived with a different way in \cite{N1} as a consequence of infinite concentration estimates. See the argument in Section 2 in \cite{McMc} and  (1.14) and (1.15) in \cite{N1}. On the other hand, as discussed in Section 6 in \cite{N2}, \eqref{b3} and \eqref{b4} are nothing but the limit formulas of \eqref{b1} and \eqref{b2} as $q\to1^+$ or equivalently, $p\to \infty$. More precisely, if we write the dependence of $a_k$, $\delta_k$, and $\eta_k$ on $q\in [1,2)$ as $a_k(q)$, $\delta_k(q)$, and $\eta_k(q)$ respectively, we have the continuous relations  
\[
a_k(1)=\lim_{q\to 1^+}a_k(q),\ \delta_k(1)=\lim_{q\to 1^+}\delta_k(q),\ \eta_k(1)=\lim_{q\to 1^+}\eta_k(q)
\]
and 
\[
\begin{split}
\lim_{q\to 1^+}&\left\{\frac{2p}{2+a_k(q)}\left(1-\frac{\delta_{k+1}(q)}{\delta_k(q)}\right)-1+\left(\frac{\delta_{k+1}(q)}{\delta_k(q)}\right)^p\right\}\\
&\ \ \ \ \ \ \ \ \ \ \ \ \ \ \ \ \ \ \ \ \ \ \ \ \ \ \ \ =\frac{2}{2+a_k(1)}\log{\frac{\eta_{k}(1)}{\eta_{k+1}(1)}}-1+\frac{\eta_{k+1}(1)}{\eta_{k}(1)}
\end{split}
\]
for all $k\in \mathbb{N}$. See Proposition 6.2 and its proof (especially Lemma 6.8) in \cite{N2}. 
\subsection{Sequence of limit profiles and scaling structure}\label{sub:lipro}
Next, we define the infinite sequence $(z_k)$ of limit profiles %of concentration parts of blow-up solutions
 by using the sequence $(a_k)$ above. We first set the standard limit profile $z_0$ by 
\begin{equation}\label{def:z0}
z_0(r)=\log{\frac{64}{(8+r^2)^2}}
\end{equation}
for all $r\ge 0$. It satisfies
\begin{equation}\label{eq:z}
\begin{cases}
-z_0''-\frac1r z_0'=e^{z_0}\text{ in }(0,\infty),\\
z_0(0)=0,\ z_0'(0)=0,
\end{cases}
\end{equation}
and 
\[
\int_0^\infty e^{z_0}rdr=4.
\]
We next assume $q\in[1,2)$ and define, for each $k\in \mathbb{N}$,
\begin{equation}\label{def:zk}
z_k(r)=\log{\frac{2a_k^2 b_k}{r^{2-a_k}(1+b_kr^{a_k})^2}}
\end{equation}
for all $r\ge0$ if $k=1$ and all $r>0$ if $k>1$ where $b_k=(\sqrt{2}/a_k)^{a_k}$. Then for every $k\in \mathbb{N}$, $z_k$ verifies
\begin{equation}\label{eq:zk}
\begin{cases}
-z_k''-\frac1r z_k'=e^{z_k}\text{ in }(0,\infty),\\
z_k(a_k/\sqrt{2})=0,\ (a_k/\sqrt{2})z_k'(a_k/\sqrt{2})=2,
\end{cases}
\end{equation}
and 
\begin{equation}\label{eq:ken}
\int_0^\infty e^{z_k}rdr=2a_k.
\end{equation}
Notice that $z_k$ has a singularity at the origin if and only if $k>1$. 

Our main strategy is to connect blow-up solutions  of \eqref{p} with the sequence $(z_k)$ of  these ``bubbles" via the scaling procedure. This allows us to obtain the precise quantification of the blow-up behavior by using the explicit information of $(z_k)$. Especially, the scaling procedure enables us to connect the quantity
\[
\int_0^1hf'(u)rdr
\]       
with the mass \eqref{eq:ken}. We notice that uniform estimates for this ``energy" associated to the scaling property often play important roles in the analysis of \eqref{p}. In fact, for the critical case $f(t)\sim e^{t^2}$, it essentially coincides with  the $H_0^1$ energy $\int_0^1uf(u)rdr(=\int_0^1u'(r)^2rdr)$ and for the exponential case $f(t)\sim e^t$, it becomes the mass $\int_0^1e^urdr$ of \eqref{p}.  Moreover, we also observe their roles in the recent analysis of the supercritical problems. See Lemmas 2.2 in \cite{FZ} for higher dimensions and Lemma 4.2 in \cite{Ku} for discs. In our main theorems below, we shall clarify the connection between the local energy and each bubbling part and finally give a constructive proof of the divergence of the global energy.
\subsection{Main results: Concentration estimates}\label{sub:mr}
Now let us show our main theorems. For our aim, we consider any sequence of blow-up solutions of \eqref{p}. Let  $(\la_n,\mu_n,u_n)\in (0,\infty)\times (0,\infty)\times C^2([0,1])$ satisfy  
\begin{equation}\label{pn}
\begin{cases}
- u_n''-\frac1r u_n'=\la_n hf(u_n), \ u_n>0&\text{ in }(0,1),\\
u_n(0)=\mu_n,\ u_n'(0)=0=u_n(1),
\end{cases}
\end{equation}
for all $n\in \mathbb{N}$. After this, we call $\{(\la_n,\mu_n,u_n)\}$ a sequence of solutions of \eqref{pn}. In this first part, we shall give  infinite concentration estimates for such sequences. 
%\subsubsection{Concentration estimates}\label{ce}
We first deduce the limit profile of the first concentration part which holds for any  $q\in [1,\infty)$.  
\begin{theorem}\label{th:0} Assume (H1) and $\{(\la_n,\mu_n,u_n)\}$ is a sequence of solutions of \eqref{pn} such that $\mu_n\to \infty$ as $n\to \infty$. Set sequences $(\ga_{0,n})$ of values and $(z_{0,n})$ of functions so that for each large $n\in\mathbb{N}$,
\[
\la_n h(0)f'(\mu_n)\ga_{0,n}^2=1
\]
and 
\[
z_{0,n}(r)=g'(\mu_n)(u_n(\ga_{0,n}r)-\mu_n)
\]
for all $r\in [0,1/\ga_{0,n}]$. Then we get that $\ga_{0,n}\to0$ as $n\to \infty$ and there exists a sequence $(\rho_{0,n})\subset (0,1)$ of values such that $u_n(\rho_{0,n})/\mu_n\to1$, $\rho_{0,n}\to0$, $\rho_{0,n}/\ga_{0,n}\to \infty$, $\|z_{0,n}- z_0\|_{C^2([0,\rho_{0,n}/\ga_{0,n}])}\to0$ as $n\to \infty$, and
\[
\lim_{n\to \infty}g'(\mu_n)\int_0^{\rho_{0,n}}\la_n hf(u_n)rdr=4=\lim_{n\to \infty}\int_0^{\rho_{0,n}}\la_n hf'(u_n)rdr
\]
up to a subsequence. 
\end{theorem}
In this theorem, we prove that the first concentration is described  by the standard bubble $z_0$ for any $q\ge1$. This is a generalization of  Theorem 1.1 in \cite{N1}. Similarly to that paper, our question is what happens on the outside of the first concentration interval $[0,\rho_{0,n}]$. The answer for the subcritical case $q>2$ is that no additional concentration occurs. More precisely, we can show
\begin{equation}\label{eq:sub1}
\lim_{n\to \infty}g'(\mu_n)\int_0^1\la_n hf(u_n)rdr=4=\lim_{n\to \infty}\int_0^1\la_n hf'(u_n)rdr.
\end{equation}
Moreover, this leads to the convergence to the Green function in the following sense, 
\begin{equation}\label{eq:sub3}
\lim_{n\to \infty}g'(\mu_n)u_n(r)=4\log{\frac1r}\text{ in }C^2_{\text{loc}}((0,1]).
\end{equation}
In addition, we arrive at the next asymptotic formula,
\begin{equation}\label{eq:sub2}
\lim_{n\to \infty}\frac{\log{\frac1{\la_n}}}{g(\mu_n)}=\frac{2-p}{2},
\end{equation}
in particular,  $\la_n\to0$ as $n\to \infty$ by the fact that $1/p+1/q=1$. We postpone the proof of these facts in Appendix \ref{sec:sub} below. On the other hand, the situation in the critical case $q=2$ is very delicate. The similar behavior to that in the subcritical case is observed for $f(t)=te^{t^2+\alpha t^{\beta}}$ with $\alpha\ge0$ and $\beta\in(0,2)$. See Theorem 2 in \cite{MM1} for the case $\alpha=0$ where the authors accomplish the proof with the precise analysis of the asymptotic expansion of the scaled function. The result for the case $\alpha>0$ is obtained in \cite{N0}. See the behaviors (i) and (ii) with $k=0$ in Theorem 1.1 there. On the other hand, interestingly, different behaviors appear if $\alpha<0$. See the nonexistence result Theorem 1.2 in \cite{ASY} for the case $\beta\in(0,1]$. Furthermore, it is proved by Theorem 0.3 in \cite{MT} that, in addition to the first concentration, a residual mass appears if $\beta=1$.  These results suggest that the growth of the perturbation delicately affects the blow-up behavior in the critical case. 
%In this paper, we proceed to our main discussion for the supercritical case $q\in[1,2)$.

Now, let us show our main results on the supercritical case $q\in[1,2)$ where we observe the drastically different behavior from those in the subcritical and critical cases noted above. Indeed,  we detect an infinite sequence of concentrating parts as follows. %Moreover, we classify the two types of infinite concentrating behaviors depending on the choice of $q\in[ 1,2)$ as follows.
\begin{theorem}\label{th:sup1} Assume as in Theorem \ref{th:0} with $q\in [1,2)$. Then, extracting a subsequence if necessary, for each $k\in \mathbb{N}$, there exists a sequence $(r_{k,n})\subset (0,1)$ of values such that %putting $u_n(r_{k,n})=\mu_{k,n}$ for all $n\in \mathbb{N}$, we have that
 $u_n(r_{k,n})/\mu_n\to \delta_k$
and 
\begin{equation}\label{eq:sup0}
\la_nr_{k,n}^2h(r_{k,n})f'(u_n(r_{k,n}))\to \frac{a_k^2}2
\end{equation}
as $n\to \infty$ and if we put sequences $(\ga_{k,n})$ of values and $(z_{k,n})$ of functions so that for each large $n\in \mathbb{N}$,
\[
\la_nh(r_{k,n})f'(u_n(r_{k,n}))\ga_{k,n}^2=1
\]
and 
\[
z_{k,n}(r)=g'(u_n(r_{k,n}))(u_n(\ga_{k,n}r)-u_n(r_{k,n}))
\]
for all $r\in [0,1/\ga_{k,n}]$, then  we have that $\ga_{k,n}\to0$ and there exist sequences $(\bar{\rho}_{k,n}),(\rho_{k,n})\subset [0,1)$, where we chose $\bar{\rho}_{1,n}=0$ for all $n\in \mathbb{N}$ if $k=1$, such that $\rho_{k,n}\to0$, $\bar{\rho}_{k,n}/\ga_{k,n}\to0$, $\rho_{k,n}/\ga_{k,n}\to \infty$,  $u_n(\bar{\rho}_{k,n})/\mu_n\to \delta_k$, $u_n(\rho_{k,n})/\mu_n\to \delta_k$, 
and
\[
\|z_{k,n}-z_k\|_{C^2([\bar{\rho}_{k,n}/\ga_{k,n},\rho_{k,n}/\ga_{k,n}])}\to0
\]
as $n\to \infty$ and further, if $q=1$, we have more precisely that 
\begin{equation}\label{eq:sup00}
%\begin{cases}
\displaystyle u_n(r_n)=\mu_n-\left(\log{\frac1{\eta_k}}+o(1)\right)\frac{g(\mu_n)}{g'(\mu_n)}
\end{equation}
for the sequences $(r_n)=(r_{k,n}),(\rho_{k,n}),$ and $(\bar{\rho}_{k,n})$   as $n\to \infty$. Moreover, for all $k\in \mathbb{N}$,  we get 
\begin{equation}\label{eq:sup01}
\begin{split}
\lim_{n\to \infty}g'(u_n(r_{k,n}))&\int_{\bar{\rho}_{k,n}}^{\rho_{k,n}}\la_n hf(u_n)rdr\\
&= 2a_k=\lim_{n\to \infty}\int_{\bar{\rho}_{k,n}}^{\rho_{k,n}}\la_n hf'(u_n)rdr
\end{split}
\end{equation}
and 
\begin{equation}\label{eq:sup02}
\lim_{n\to \infty}g'(\mu_n)\int_{\rho_{k,n}}^{\bar{\rho}_{k+1,n}}\la_n hf(u_n)rdr= 0.
\end{equation}
\end{theorem}
This theorem proves that for any $k\in \mathbb{N}$, there exists a ``center" $(r_{k,n})\subset (0,1)$ of the $k$-th concentration  such that the blow-up solutions $(u_n)$,  after scaling around $(r_{k,n})$, converges  to the solutions $z_k$ of the Liouville equation with the energy $2a_k$. Since $k$ is arbitrary, we observe an infinite sequence of concentrating parts with the precise characterization. Notice that the behavior in the case $q\in(1,2)$ is the same type of that in Theorem 1.4 in \cite{N1}. On the other hand, we observe a new behavior in the case $q=1$. In fact, we see that if $q\in(1,2)$,  $u_n(r_{k,n})/\mu_n\to \delta_k\in(0,\delta_{k-1})$ as $n\to \infty$ for all $k\not=1$ while  if $q=1$,  $u_n(r_{k,n})/\mu_n\to \delta_k=1$ for all  $k\in \mathbb{N}$. This means that in the latter case, all the concentrating points appear at much higher (the almost highest) points than those in the former case. A more precise description  of the heights of the concentration parts is given in \eqref{eq:sup00} with the decreasing sequence $(\eta_k)$ defined by the limit recurrence formulas \eqref{b3} and \eqref{b4}.

 As a consequence of the previous theorem, we obtain the following asymptotic formulas.  
\begin{theorem}\label{th:sup2} Assume as in the previous theorem. Then, for any sequence $(r_n)\subset (0,1)$ such that 
\[
\begin{cases}
\displaystyle\frac{u_n(r_n)}{\mu_n}\to 0&\text{ if }q>1,\\
\displaystyle(\mu_n-u_n(r_n))\frac{g'(\mu_n)}{g(\mu_n)}\to \infty&\text{ if }q=1,
\end{cases}
\]
%$u_n(r_n)/\mu_n\to 0$ if $q>1$ and $(\mu_n-u_n(r_n))g'(\mu_n)/g(\mu_n)\to \infty$ if $q=1$
as $n\to \infty$, we have that 
\begin{equation}\label{eq:sup1}
\lim_{n\to \infty}\int_0^{r_n}\la_n hf'(u_n)rdr=\infty.
\end{equation}
Moreover, we get that
\begin{equation}\label{eq:sup4}
g'(\mu_n)u_n(r)\to \infty
\end{equation}
for all $r\in[0,1)$ 
%as $n\to \infty$ and
%\begin{equation}\label{eq:sup20}
%\lim_{n\to \infty}\frac{\log{\frac1{\la_n}}}{\mu_ng'(\mu_n)}=0.
%\end{equation}
%Finally,
and  additionally assuming (H2) if $q=1$, we have that 
\begin{equation}\label{eq:sup2}
\lim_{n\to \infty}\frac{\log{\frac1{\la_n}}}{g(\mu_n)}=0
\end{equation}
and 
\begin{equation}\label{eq:sup3}
\lim_{n\to \infty}\frac{\log{\frac1{r_{k,n}}}}{g(\mu_n)}=\frac{\eta_k}{2}
%\begin{cases}
%\displaystyle\frac{\delta_k^p}{2}\text{ if }q>1,\vspace{0.3cm}\\
%\displaystyle\frac{\eta_k}{2}\text{ if }q=1,
%\end{cases}
\end{equation}
for all $k\in \mathbb{N}$ and all $q\in[1,2)$.
\end{theorem}
Comparing \eqref{eq:sup1}, \eqref{eq:sup4}, and \eqref{eq:sup2} with \eqref{eq:sub1}, \eqref{eq:sub3}, and \eqref{eq:sub2} respectively, we observe that striking differences appear because of the presence of the infinite sequence of concentrating parts.  First, \eqref{eq:sup1} shows that it breaks the uniform boundedness of the ``energy" which is usually proved or assumed in the previous works on critical and subcritical cases. Moreover, \eqref{eq:sup4} and \eqref{eq:sup2} imply that the convergence to the Green function in the sense \eqref{eq:sub3} or the asymptotic formula \eqref{eq:sub2} do no longer hold for the supercritical case $p>2$.  In view of this,  we may expect that different phenomena may occur in the pointwise limit of the blow-up solutions $(u_n)$ and the convergence of the parameters $(\la_n)$. Actually, Kumagai \cite{Ku} recently proves that any blow-up solutions converge to a singular solution up to a subsequence by (B) of Theorem 1.1 there. For the reader convenience, we give Lemma \ref{lem:ku} which  shows that our (H1) implies his essential assumption. Then, combining our theorem with his observation, we see that any  blow-up solutions behave as the infinite sequence of concentrating parts around the origin  while they behave as the limit singular solution in the region away from the origin. This  behavior  is considerably different from those in the critical and subcritical cases found in the previous works. In the second part \cite{N3}, we  carry out the deeper analysis of the interaction between concentrating parts  and singular solutions (which are not necessarily the limit ones).

Finally we remark on  \eqref{eq:sup3} which gives the precise asymptotic information of the position of the center of each concentration. This with Lemmas \ref{lem:g3} and \ref{lem:g4} below implies
\[
g(u_n(r_{k,n}))=(2+o(1))\log{\frac1{r_{k,n}}}
\]
as $n\to \infty$. The asymptotic formulas of this kind will be the key tools to show the infinite oscillations of blow-up solutions around singular solutions. See Theorem 2.1 and its consequences in the second part \cite{N3}. In this way, our  concentration analysis proceeds to the study of infinite oscillation phenomena in the second part.  The rest of this first part  is devoted to the proof of the theorems above.   
\subsection{Strategy and organization}
For the proof, we follow the argument in the previous paper \cite{N1} based on the scaling and pointwise techniques developed in \cite{D} and the radial analysis in \cite{MM1} together with the careful use of the Green type identities in Lemma \ref{lem:id} below.  In order to accomplish our proof under the general setting, the first key step is to deduce all the essential properties of our nonlinearirties from (H1). This is successfully done in  Subsection \ref{sub:geg2} below. 

The second important step is to detect the subsequent concentrating parts. To this end, following the idea in \cite{D}, the key tools are  ``the scaling function" $\phi_n$ and ``the energy (or gradient) function" $\psi_n$ defined in Subsection \ref{sub:D} below. Actually, detecting a new concentration part corresponds to finding a new sequence $(r_n)\subset (0,1)$ of blowing-up points such that $\phi_n(r_n)\to c$ as $n\to \infty$ for some value $c\not=0$. Once we find such a sequence, we can detect a new bubble appearing around $(r_n)$ via the scaling procedure. In order to accomplish these arguments, it is important to understand the delicate balance  between $\phi_n$ and $\psi_n$. The essential properties of those functions are  summarized in Subsection \ref{sub:D}  which are  extensions of those found in Appendix. Radial analysis in \cite{D}.  

Now, let us explain the outline of the procedure to detect the infinite sequence of concentrating parts. First we introduce the standard scaled functions and find the first concentration interval in Lemmas \ref{lem:c1} and \ref{lem:c2}.  Then using the suitable pointwise estimate in the latter lemma, we extend the interval as far as no additional concentration occurs. We remark that, in the subcritical case, this step  essentially ends the proof because it allows the extension up to the boundary, see Lemma \ref{lem:c4}, while in the supercritical case it does not. However, it enables us to find the maximal interval where no additional energy appears. See Lemmas \ref{lem:c10} and \ref{lem:c11}. In particular, in the latter lemma, we prove that a certain amount of additional energy appears in the outside of that interval. Here, the Green type identity  and the energy recurrence formulas play a successful role to determine the exact ``height", described by the sequences $(\delta_k)$ and $(\eta_k)$, of the region where the additional energy appears. % In other words, we observe that %may interpret that
Then we may expect that the appearance of the nontrivial energy implies the appearance of the next concentration. Actually, utilizing Lemma \ref{lem:D2} below,  we can arrive in the next concentration region in Lemma  \ref{lem:c12}.  We then find the new center of concentration and detect the singular bubble via the scaling argument in Lemma \ref{lem:c13}. We here remark that the center of concentration is characterized as the local maximum point of the scaling function $\phi_n$.  Finally, noting the energy identity \eqref{eq:enid} at the critical points of  $\phi_n$ with \eqref{b2} and \eqref{b4}, we determine the exact profile and energy value by using  the sequence $(a_k)$. See Lemma \ref{lem:c15}. %By the proof, we see that \eqref{b2} and \eqref{b4} are nothing but the identity \eqref{eq:enid} at the center of concentration.
 This completes the description of the new concentrating part via the Liouville equation with the energy recurrence formulas. Then we repeat the same argument and finally find that this  procedure can be repeated infinitely many times.  In this way, we detect an infinite sequence of concentrating parts with the precise characterization. 
  
The organization of this paper is the following. In Section \ref{sec:pre}, we deduce  all the essential properties of nonlinearities  from (H1) and collect all the key tools for our proof. Next in Section \ref{sec:fc}, we study the first concentration part  and prove  Theorem \ref{th:0}. Next, in Section \ref{sec:ic}, we give our main argument to detect the subsequent concentrating parts. Finally, in Section \ref{sec:prf}, we accomplish the proof of  Theorems \ref{th:sup1} and \ref{th:sup2}. In addition, Appendix \ref{sec:sub} is devoted the proof of the blow-up estimates in the subcritical case. %Finally, we complete our oscillation estimates and prove Theorems \ref{th:osc} and Theorem \ref{th:o3}. 

In the proofs below, we often choose subsequence without any change of suffixes for simplicity. %We use same character $c,C$, and $(r_n)$ to denote several constants and sequences when there are no confusions. 
\section{Preliminaries}\label{sec:pre}
In this section, we deduce the essential properties of our nonlinearity based on (H1) and introduce some key tools.
\subsection{Properties of generalized exponential growth}\label{sub:geg2}
We first summarize the essential properties of our generalized exponential growth. %In this subsection, we always assume (H1).
 We begin with some basic consequences on the monotonicity and asymptotic behavior of $g$ and $g'$. 
\begin{lemma}\label{lem:monog} Assume (H1). Then we have that $g(t)$ and $g'(t)$ are strictly increasing  for all $t\ge t_0$ and that $g(t)\to \infty$ and $g'(t)\to \infty$ as $t\to \infty$. %and  if $q>1$, $g'(t)/g(t)$ is strictly increasing
 Moreover, we get
\begin{equation}\label{eq:lgg}
\lim_{t\to \infty}\frac{\log{g'(t)}}{g(t)}=0.
\end{equation}
\end{lemma}
\begin{proof} The monotonicity of $g(t)$ and $g'(t)$ clearly follows from (i) of (H1). Moreover, since it also ensures $tg'(t)\to \infty$ as $t\to \infty$, we have that 
\[
g'(t)\ge \frac1t
\]
for all large $t>0$. This proves $g(t)\to \infty$ as $t\to \infty$ after integration. Then by \eqref{g1}, we have that 
\[
(\log{g(t)})'\le (q+1)(\log{g'(t)})'
\] 
for all large $t>0$. This implies $g'(t)\to \infty$ as $t\to \infty$ by integration. Finally, using the de l'H\^{o}pital rule and \eqref{g1}, we get 
\[
\lim_{t\to \infty}\frac{\log{g'(t)}}{g(t)}=\lim_{t\to \infty}\frac{g(t)g''(t)}{g'(t)^2}\frac1{g(t)}=0.
\]
We finish the proof.   
\end{proof}
The next lemma ensures that $q$ is the H\"{o}lder conjugate of $p$. 
\begin{lemma}\label{lem:pq} We suppose (H1). Then, we have that $1/p+1/q=1$ where we regarded $1/\infty=0$.
\end{lemma}
\begin{proof} Noting $g(t),tg'(t)\to \infty$ as $t\to \infty$ from the previous lemma, we get by the de l'H\^{o}pital rule that
\[
\lim_{t\to \infty}\frac{tg'(t)}{g(t)}=\lim_{t\to \infty}\frac{g'(t)+tg''(t)}{g'(t)}=\lim_{t\to \infty}\left(1+\frac{tg'(t)}{g(t)}\cdot\frac{g(t)g''(t)}{g'(t)^2}\right).
\]
Then, \eqref{g1} and \eqref{g2} prove the desired conclusion. We complete the proof. 
\end{proof}
Next we shall deduce some lemmas which will be used for our scaling argument. The next two lemmas are preliminaries, based on (ii) of (H1), for Lemma \ref{lem:g2} below. 
\begin{lemma}\label{lem:g200} Suppose (H1). If $q>1$ ($q=1$), then $g'(t)/g(t)$ ($\hat{g}'(t)/\hat{g}(t)$ respectively)  is nonincreasing for all large $t\ge t_0$. 
\end{lemma}
\begin{proof} Assume $q>1$. Then we have that
\[
\left(\frac{g'(t)}{g(t)}\right)'=\left(\frac{g'(t)}{g(t)}\right)^2\left(\frac{g(t)g''(t)}{g'(t)^2}-1\right)
\]
for all $t\ge t_0$. Then \eqref{g1} proves the assertion. In the case $q=1$, the conclusion is in (ii) of (H1). We finish the proof. 
\end{proof}
\begin{lemma}\label{lem:g20} Assume (H1) with $q=1$ and put $\hat{g}_0(t)=\hat{g}(t)$ and $\hat{g}_i(t)=\exp{(\hat{g}_{i-1}(t))}
$ for all $t\ge t_0$ and $i=1,\cdots,k$ by induction. Then we have that
\[
\limsup_{t\to \infty}\frac{\hat{g}_i(t)\hat{g}_i''(t)}{\hat{g}_i'(t)^2}\le1
\]
for all $i=0,1,\cdots,k$.
\end{lemma}
\begin{proof}
We argue by induction. From (ii), we get for all large $t>0$ that
\[
0\ge \left(\frac{\hat{g}'(t)}{\hat{g}(t)}\right)'=\left(\frac{\hat{g}'(t)}{\hat{g}(t)}\right)^2\left(\frac{\hat{g}(t)\hat{g}''(t)}{\hat{g}'(t)^2}-1\right)
\]
 This gives the conclusion for $i=0$.  Then,  we assume that for some $i=l\in \{0,1,\cdots,k-1\}$, the assertion  holds true. Then noting $\hat{g}_l(t)\to \infty$ as $t\to \infty$ by Lemma \ref{lem:monog} and 
\[
\frac{\hat{g}_{l+1}(t)\hat{g}_{l+1}''(t)}{\hat{g}_{l+1}'(t)^2}=1+\frac{\hat{g}_l(t)\hat{g}_l''(t)}{\hat{g}_l'(t)^2}\frac{1}{\hat{g}_l(t)}
\]
for all $t\ge t_0$, we prove the conclusion for $i=l+1$. This finishes the proof. 
\end{proof}
Then we prove the next key lemma.  
\begin{lemma}\label{lem:g2} Assume (H1). Then, for any value $M>0$, we get that 
\[
\lim_{t\to \infty}\sup_{-M\le y\le M}\left|\frac{g'\left(t+\frac{y}{g'(t)}\right)}{g'(t)}-1\right|=0.
\]
\end{lemma}
\begin{proof} Fix any value $M>0$ and  suppose $y\in [-M,M]$. Then, for any large $t\ge t_0$, we set $\alpha=t+y/g'(t)$. First we note that there exist values $\theta_1\in(0,1)$ such that, putting $\alpha_1=t+\theta_1 y/g'(t)$, we have that 
\[
\log{g'\left(\alpha\right)}=\log{g'\left(t\right)}+(\alpha-t)\frac{g''\left(\alpha_1\right)}{g'(\alpha_1)}.
\]
In particular, we get that
\[
\log{\frac{g'\left(\alpha\right)}{g'(t)}}%=(\alpha-t)\frac{g''\left(\alpha_2\right)}{g'(t)}
=O\left((\alpha-t)\frac{g'\left(\alpha_1\right)}{g(\alpha_1)}\right)
\]
by \eqref{g1}. Therefore for the desired conclusions, it suffices to show 
\[
\lim_{t\to \infty}\frac{g'\left(\beta\right)}{g'(t)g(\beta)}=0
\]
uniformly for all $t-M/g'(t)\le \beta\le t+M/g'(t)$. Then, noting the monotonicity of $g'$, we easily ensure the desired convergence uniformly for $t-M/g'(t)\le \beta\le t$. Hence, our goal is to show 
\begin{equation}\label{eq:g21}
\lim_{t\to \infty}\sup_{t\le \beta\le t+M/g'(t)}\left(\frac{1}{g'(t)}\frac{g'(\beta)}{g(\beta)}\right)=0.
\end{equation}
Here if $q>1$, or $q=1$ and (ii) of (H1) holds true with $k=1$, then we get 
\[
\sup_{t\le \beta\le t+M/g'(t)}\left(\frac{1}{g'(t)}\frac{g'(\beta)}{g(\beta)}\right)\le \frac{1}{g(t)}
\]
for all large $t\ge t_0$ by Lemma \ref{lem:g200} which gives \eqref{eq:g21}. Hence, we assume $q=1$ and (ii)  of  (H1) holds with $k>1$. Then, for all large $t\ge t_0$, set a value $\beta_1 \in [t,t+M/g'(t)]$ so that 
\[
\sup_{t\le \beta\le t+M/g'(t)}\left(\frac{1}{g'(t)}\frac{g'(\beta)}{g(\beta)}\right)=\frac{1}{g'(t)}\frac{g'(\beta_1)}{g(\beta_1)}.
\]
Using the definition in Lemma \ref{lem:g20}, we may write $g(t)=\hat{g}_{k-1}(t)=e^{\hat{g}_{k-2}(t)}$ for all $t\ge t_0$. Then  there exists a value $\beta_2\in [t,\beta_1]$ such that
\begin{equation}\label{eq:g23}
\begin{split}
\log{\left(\frac1{g'(t)}\frac{g'(\beta_1)}{g(\beta_1)}\right)}&=\log{\frac1{g'(t)}}+\log{(\hat{g}_{k-2}'(\beta_1))}\\
&=\log{\frac1{g'(t)}}+\log{(\hat{g}_{k-2}'(t))}+(\beta_1-t)\frac{\hat{g}_{k-2}''(\beta_2)}{\hat{g}_{k-2}'(\beta_2)}\\
&\le \log{\frac1{g(t)}}+(1+o(1))\left(\beta_1-t\right)\frac{\hat{g}_{k-2}'(\beta_2)}{\hat{g}_{k-2}(\beta_2)}
\end{split}
\end{equation}
by Lemma \ref{lem:g20}. Consequently, once we prove  
\begin{equation}\label{eq:g24}
\lim_{t\to \infty}\sup_{t\le \beta\le t+M/g'(t)}\left(\frac{1}{g'(t)}\frac{\hat{g}_{k-2}'(\beta)}{\hat{g}_{k-2}(\beta)}\right)=0,
\end{equation}
\eqref{eq:g23} shows \eqref{eq:g21}. Here, similarly to the argument before, if (ii) of (H1) holds true with $k=2$, then  $\hat{g}_{k-2}=\hat{g}$ and thus, Lemma \ref{lem:g200} gives \eqref{eq:g24}. Therefore we  suppose (ii) of (H1) holds with $k>2$. Then for each large $t\ge t_0$, we set a value $\beta_3\in[t,t+M/g'(t)]$ so that 
\[
\sup_{t\le \beta\le t+M/g'(t)}\left(\frac{1}{g'(t)}\frac{\hat{g}_{k-2}'(\beta)}{\hat{g}_{k-2}(\beta)}\right)=\frac{1}{g'(t)}\frac{\hat{g}_{k-2}'(\beta_3)}{\hat{g}_{k-2}(\beta_3)},
\]
write $\hat{g}_{k-2}(t)=e^{\hat{g}_{k-3}(t)}$, and  repeat the computation similarly to \eqref{eq:g23} for 
\[
\log{\left(\frac{1}{g'(t)}\frac{\hat{g}_{k-2}'(\beta_3)}{\hat{g}_{k-2}(\beta_3)}\right)}.
\] 
In this way, repeating the same argument some more times if necessary, we arrive at  our final goal  to show that
\[
\lim_{t\to \infty}\sup_{t\le \beta\le t+M/g'(t)}\frac{1}{g'(t)}\frac{\hat{g}'(\beta)}{\hat{g}(\beta)}= 0
\] 
which is readily confirmed with Lemma \ref{lem:g200}.  We finish the proof.
\end{proof}
The main consequence of the previous lemma is the next one which will used to determine  the limit equation. 
\begin{lemma}\label{lem:g1} %Let $(\xi_n)$ be a sequence of positive values such that $\xi_n\to \infty$ as $n\to \infty$. Then
 Supoose (H1). Then for any value $M>0$, we obtain that
\[
\lim_{t\to \infty}\sup_{-M\le y\le M} \left|g\left(t+\frac{y}{g'(t)}\right)-g(t)-y\right|=0.
\]
In particular, we have that
\[
\lim_{t\to \infty}\sup_{-M\le y\le M} \left|\frac{f\left(t+\frac{y}{g'(t)}\right)}{f(t)}-e^{y}\right|=0=\lim_{t\to \infty}\sup_{-M\le y\le M} \left|\frac{f'\left(t+\frac{y}{g'(t)}\right)}{f'(t)}-e^{y}\right|.
\]
\end{lemma}
\begin{proof}Fix any value $M>0$. For each large $t\ge t_0$,  take a value $y_t\in[-M,M]$ so that 
\[
\begin{split}
\sup_{-M\le y\le M}& \left|g\left(t+\frac{y}{g'(t)}\right)-g(t)-y\right|=\left|g\left(t+\frac{y_t}{g'(t)}\right)-g(t)-y_t\right|.
\end{split}
\]
Then there exists a value  $\alpha_t\in [t-M/g'(t),t+M/g'(t)]$ such that
\begin{equation}\label{eq:a1}
\begin{split}
\left|g\left(t+\frac{y_t}{g'(t)}\right)-g(t)-y_t\right|&= \frac{g''\left(\alpha_t\right)y_t^2}{2g'(t)^2}=O\left(\frac{1}{g(\alpha_t)}\left(\frac{g'\left(\alpha_t\right)}{g'(t)}\right)^2\right) %\\
%&=\frac{y_t^2}{2g(\alpha_t)}\left(\frac{g'\left(\alpha_t\right)}{g'(t)}\right)^2\frac{g(\alpha_t)g''\left(\alpha_t\right)}{g'(\alpha_t)^2}.
\end{split}
\end{equation}
from \eqref{g1}. Hence,  Lemmas \ref{lem:monog} and \ref{lem:g2} show  the first conclusion. Noting this together with Lemma \ref{lem:g2}, we readily confirm the  latter ones.  This completes the proof.
\end{proof}
%\begin{lemma}[memo]  (H1) implies that for any value $\e>0$, it holds that
%\[
%\lim_{t\to \infty}\frac{\log{g'(t)}}{g(t)^\e}=0.
%\]
%\end{lemma}
We next derive some lemmas applied to obtain the energy recurrence formulas. The next lemma is a preliminary.    
\begin{lemma}\label{lem:g00} Suppose (H1). Then for any $\e>0$, there exists a number $t_\e\ge t_0$ such that 
\[
\left(\frac{g(t)}{g(s)}\right)^{\frac1{q+\e}}\le \frac{g'(t)}{g'(s)}\le \left(\frac{g(t)}{g(s)}\right)^{\frac1{q-\e}}.
\]
for any $t>s\ge t_\e$.
\end{lemma}
\begin{proof}
From \eqref{g1},  for any value $\e>0$, there exists a constant $t_\e\ge t_0$ such that  
\[
 (q-\e)(\log{g'(t)})'\le (\log{g(t)})'\le (q+\e)(\log{g'(t)})'
\]
for all $t\ge t_\e$. Then, for every $t>s \ge t_\e$, integrating over $[s,t]$ gives 
\[
(q-\e)\log{\frac{g'(t)}{g'(s)}} \le \log{\frac{g(t)}{g(s)}}\le (q+\e)\log{\frac{g'(t)}{g'(s)}}.
\]
This proves the desired estimate. We finish the proof.
\end{proof} 

The next lemmas determine the exact ratio of growths at two blowing-up points. We first give the lemma for $q>1$. 
\begin{lemma}\label{lem:g3} Assume (H1) with $q>1$ and take any sequences $(s_n),(t_n)\subset (0,\infty)$ of numbers and a value $x\in [0,1]$ such that $s_n<t_n$ for all $n\in \mathbb{N}$ and  $s_n\to \infty$ and  $s_n/t_n\to x$ as $n\to \infty$. Then we have that
\[
\lim_{n\to \infty}\frac{g(s_n)}{g(t_n)}=x^p=\lim_{n\to \infty}\left(\frac{g'(s_n)}{g'(t_n)}\right)^q.
\]
%Moreover, there exists a constant $\delta \in (0,1]$ such that $s_n/t_n\to \delta $ as $n\to \infty$. Finally, we get that if $q>1$, then $\eta=\delta^p$ and if $q=1$, then $\delta=1$. 
\end{lemma}
\begin{proof} First, by \eqref{g2}, for any value $\e>0$, there exists a constant $t_\e'\ge t_0$ such that 
\[
(p-\e)\frac1t\le (\log{g(t)})'\le (p+\e)\frac1t
\]
for all $t\ge t_\e'$. It follows that 
\[
\left(\frac{t_n}{s_n}\right)^{p-\e}\le\frac{g(t_n)}{g(s_n)}\le \left(\frac{t_n}{s_n}\right)^{p+\e}
\]
for all large $n\in \mathbb{N}$. This proves that 
\[
\lim_{n\to \infty}\frac{g(s_n)}{g(t_n)}=x^p
\]
since $\e>0$ is arbitrary. Then, using Lemma \ref{lem:g00}, we obtain the latter conclusion. 
%at
%\[
%\lim_{n\to \infty}\left(\frac{g'(t_n)}{g'(s_n)}\right)^q= \lim_{n\to \infty}\frac{g(t_n)}{g(s_n)}.
%\]
%These two formulas prove the lemma. 
We finish the proof. 
\end{proof}
Next, we give the lemma for $q=1$. Again we apply (ii) of (H1).
\begin{lemma}\label{lem:g4} Suppose (H1) with $q=1$. Then, for any sequences  $(t_n),(x_n)\subset (0,\infty)$ of values and number  $x_0\in [0,\infty)$ such that $t_n\to \infty$ and $x_n\to x_0$ as $n\to \infty$,  we have that
\[
\lim_{n\to \infty}\frac{g\left( t_n-x_n\frac{g(t_n)}{g'(t_n)}\right)}{g(t_n)}=e^{-x_0}=\lim_{n\to \infty}\frac{g'\left( t_n-x_n\frac{g(t_n)}{g'(t_n)}\right)}{g'(t_n)}.
\]
\end{lemma}
\begin{proof} There exists a sequence $(\alpha_n)\subset(t_n-x_ng(t_n)/g'(t_n),t_n)$ of constants  such that 
\begin{equation}\label{eq:g41}
\begin{split}
&\log{\frac{g\left( t_n-x_n\frac{g(t_n)}{g'(t_n)}\right)}{g(t_n)}}+x_0\\
&\ \ \ \ \ =\log{g\left( t_n-x_n\frac{g(t_n)}{g'(t_n)}\right)}-\log{g(t_n)}+x_n+o(1)\\%\left(\frac{g'\left( t-\theta x\frac{g(t)}{g'(t)}\right)}{g\left( t-\theta x\frac{g(t)}{g'(t)}\right)}\frac{g(t)}{g'(t)}-1\right)x.
&\ \ \ \ \ =O\left(\left(\frac{g\left( \alpha_n\right)g''\left(\alpha_n\right)}{g'\left(\alpha_n\right)^2}-1\right)\left(\frac{\alpha_ng'\left( \alpha_n\right)}{g\left( \alpha_n\right)}\frac{g(t_n)}{t_ng'(t_n)}\right)^2\right)+o(1)
\end{split}
\end{equation}
by the second order expansion with the fact that $\alpha_n/t_n\to1$ as $n\to \infty$ from \eqref{g2} and Lemma \ref{lem:pq}. Consequently, \eqref{g1} and the monotonicity of $tg'(t)/g(t)$ by (ii) of (H1) conclude that%Here, notice that \eqref{g2} implies that 
\[
\lim_{n\to \infty}\left(\log{\frac{g\left( t_n-x_n\frac{g(t_n)}{g'(t_n)}\right)}{g(t_n)}}+x_0\right)=0.
\]
This proves the former conclusion. Then, Lemma \ref{lem:g00} shows the latter one. We complete the proof.
\end{proof}
The next one is the well-known consequence of (H2). % for Theorem \ref{th:o3}.     
\begin{lemma}\label{lem:kap} (H2) implies that
\[
\sup\{\la>0\ |\ \text{\eqref{p0} admits a classical solution $u$}\}<\infty.
\]
\end{lemma}
\begin{proof} The proof is done similarly  to Lemma 2.4 in \cite{N1}. We show the proof for the reader's convenience. Assume $u$ is a solution of \eqref{p0}. Set $c=\inf_{t>0}(f(t)/t)$. Let $\Phi_1>0$ be the first Dirichlet eigenfunction of $-\Delta$ on $D$ with the eigenvalue $\Lambda_1>0$. Then integrating  by parts with the equation in \eqref{p0} and (H0), we get
\[
\Lambda_1\int_D u\Phi_1dx =\la \int_D f(x,u)\Phi_1dx \ge c\la \min_{x\in \bar{D}}h(|x|)\int_D u\Phi_1dx.
\]
This proves $\la\le \Lambda_1/(c\min_{x\in \bar{D}}h(|x|))$. We finish the proof. 
\end{proof}
 We end this subsection by giving additional remarks on our conditions. 
We give a useful sufficient condition for (H1) with $q=1$.
\begin{lemma}\label{lem:ex1} %Set $f(t)=\exp_k{(t^m(\log{t})^l)}$ for $t\ge t_0$ with any given constants $k\ge2$, $m>0$, and $l\in \mathbb{R}$. Then it satisfies (H1).  
Assume $k\in \mathbb{N}$ with $k\ge2$. Set $f(t)=\exp_k{(g_0(t))}$ for all large $t\ge t_0$ with a function $g_0\in C^2([t_0,\infty))$ such that  
\begin{equation}\label{eq:ex1}
\lim_{t\to \infty}tg_0'(t)= \infty,\ \ \ \lim_{t\to \infty}\frac{g_0''(t)}{g_0'(t)^2}=0,  % and $g_0(t)g_0''(t)/g'(t)^2\le1 $ for all large $t>0$.
\end{equation}
\begin{equation}\label{eq:ex2}
\text{$tg_0'(t)$ is nondecreasing, and\ $\frac{g_0'(t)}{g_0(t)}$ is nonincreasing}
\end{equation}
for all $t\ge t_0$. Then $f(t)$ satisfies (H1) with $q=1$. 
\end{lemma}
\begin{proof}
First setting $\hat{g}(t)=g_0(t)$ for all  $t\ge t_0$, we confirm that the latter condition in (ii) of (H1) is satisfied. Next we set $g(t)=\exp_{k-1}(g_0(t))$ for all $t\ge t_0$. Then we claim that $g$ satisfies the rest of the assertions in (H1). To show this,  we put $g_i(t)=\exp_i(g_0(t))$ for all $t\ge t_0$ and $i=1,\cdots,k-1$. We first observe that
\[
\lim_{t\to \infty}\frac{g_1(t)g_1''(t)}{g_1'(t)^2}=\lim_{t\to \infty}\left(1+\frac{g_0''(t)}{g_0'(t)^2}\right)=1
\] 
and 
\[
\lim_{t\to \infty}\frac{tg_1'(t)}{g_1(t)}=tg_0'(t)=\infty
\]
by \eqref{eq:ex1}. %In particular, the former formula shows that $g_1''(t)>0$ for all large $t>0$. Moreover,
Then the latter formula  implies that 
\[
(\log{g_1(t)})'\ge (\log{t})'
\]
for all large $t>0$. This yields that $g_1(t)\to \infty$ and thus, by the latter formula again, $tg_1'(t)\to \infty$ as $t\to \infty$. Then it follows that %Moreover, using the former formula, we see 
\[
\lim_{t\to \infty}\frac{g_1''(t)}{g_1'(t)^2}=\lim_{t\to \infty}\left(\frac{g_1(t)g_1''(t)}{g_1'(t)^2}\frac1{g_1(t)}\right)=0.
\]
In addition, since $tg_1'(t)/g_1(t)=tg_0'(t)$ and \eqref{eq:ex1} yield $g_1(t)$ is increasing for all large $t\ge t_0$, we have with \eqref{eq:ex2} that $tg_1'(t)/g_1(t)$ and $tg_1'(t)$ are nondecreasing for all large $t\ge t_0$.  Then similarly we can show that for each $i=1,\cdots,k-1$, 
\[
\lim_{t\to \infty}\frac{g_i(t)g_i''(t)}{g_i'(t)^2}=1,\ \ \lim_{t\to \infty}\frac{tg_i'(t)}{g_i(t)}=\infty,\ \ \lim_{t\to \infty}\frac{g_i''(t)}{g_i'(t)^2}=0,\ \text{ and }\ \lim_{t\to \infty}tg_i'(t)= \infty
\]
and $tg_i'(t)$ and $tg_i'(t)/g_i(t)$ are nondecreasing for all large $t\ge t_0$ by induction. In particular, we show the claim.  We finish the proof. 
\end{proof}
\begin{remark} 
Notice that since
\[
\left(\frac{g_0'(t)}{g_0(t)}\right)'=\left(\frac{g_0'(t)}{g_0(t)}\right)^2\left(\frac{g_0(t)g_0''(t)}{g_0'(t)^2}-1\right),
\]
the latter condition in \eqref{eq:ex2} is equivalent to the one 
\[
\frac{g_0(t)g_0''(t)}{g_0'(t)^2}\le 1
\]
for all $t\ge t_0$.  
\end{remark}
The previous lemma is useful to check that given examples satisfy (H1) with $q=1$. For example, if $f(t)=\exp_k{(t^m(\log{t})^l)}$ for all large $t>0$ with any given values $k\in \mathbb{N}$ with $k\ge2$, $m>0$, and $\l\in \mathbb{R}$, then putting $g_0(t)=t^m(\log{t})^l$, we easily confirm that $g_0$ satisfies all the assumption in this lemma and thus, $f$ satisfies (H1) with $q=1$.

We finally  remark that (H1) implies the basic condition ($f_3$) in \cite{Ku}. Put $F(t)=\int_0^t f(s)ds$ for all $t\ge0$. 
\begin{lemma}\label{lem:ku} Assume (H1). Then it holds that 
\[
\limsup_{t\to \infty}\left(\frac{F(t)\log{F(t)}}{f(t)}\right)'\le \frac1p
\]
where we regarded $1/\infty=0$.
\end{lemma}
\begin{proof} First from the de l'H\^{o}spital rule, \eqref{eq:lgg}, and \eqref{g1}, we have
\begin{equation}\label{eq:ku1}
\lim_{t\to \infty}\frac{\log{F(t)}}{g(t)}=\lim_{t\to \infty}\frac{f(t)/g'(t)}{F(t)}=1.
\end{equation}
Next, for all $t>\tau\ge t_0$, we have that
\[
\begin{split}
F(t)&= F(\tau)+\int_{\tau}^t\frac{1}{g'(s)}\left(f(s)\right)'ds\\
&=F(\tau)-\frac{f(\tau)}{g'(\tau)}+\frac{f(t)}{g'(t)}+\int_{\tau}^t \frac{g(s)g''(s)}{g'(s)^2}\frac{f(s)}{g(s)}ds.%\\
%&\le c_3+\frac{f(t)}{g'(t)}+\e F(t).
\end{split}
\]
Hence, by \eqref{g1}, for any value $\e\in (0,1/q)$, there exist constants $c_\e,t_\e>0$ such that 
\begin{equation}\label{eq:ku2}
F(t)\ge \frac{f(t)}{g'(t)}+\left(\frac{1}{q}-\e\right)\tilde{F}_\e(t)-c_\e
\end{equation}
for all $t\ge t_\e$ where $\tilde{F}_\e(t)=\int_{t_\e}^t (f(s)/g(s))ds$. Using \eqref{eq:ku1} and \eqref{eq:ku2}, we compute
\[
\begin{split}
\left(\frac{F(t)\log{F(t)}}{f(t)}\right)'&=1-\log{F(t)}\left(\frac{f'(t)F(t)}{f(t)^2}-1\right)\\
&\le 1-\left(\frac1 q-\e+o(1)\right)\frac{g(t)g'(t)\tilde{F}_\e(t)}{f(t)}
\end{split}
\]
as $t\to \infty$. Here we calculate by Lemma \ref{lem:monog}, the de l'H\^{o}spital rule, and \eqref{g1} that
\[
\lim_{t\to \infty}\frac{g(t)g'(t)\tilde{F}_\e(t)}{f(t)}=\lim_{t\to \infty}\frac{(\tilde{F}_\e(t))'}{\left(\frac{f(t)}{g(t)g'(t)}\right)'}=1.
\]
Consequently, we deduce 
\[
\limsup_{t\to \infty}\left(\frac{F(t)\log{F(t)}}{f(t)}\right)'\le \frac1p+\e.
\]
Since $\e>0$ is  arbitrary, we get the desired formula. We finish the proof. 
\end{proof}
In the supercritical case $q<2$, Lemma \ref{lem:pq} shows $p>2$ and thus, this lemma yields that our (H1) implies ($f_3$) in \cite{Ku}. 
\subsection{Identities}
We next gives some key identities. In this subsection, we always assume $\{(\la_n,\mu_n,u_n)\}$ is a sequence of solutions of \eqref{pn}. The first one is standard.
\begin{lemma} 
We have that
\begin{equation}\label{id0}
-ru_n'(r)=\int_0^r \la_nh(s)f(u_n(s))sds
\end{equation}
for all $r\in[0,1]$ and all $n\in \mathbb{N}$. 
\end{lemma}
\begin{proof} Multiplying the equation in \eqref{pn} by $r\in[0,1]$, we get 
\[
(-ru_n'(r))'=\la_nh(r)f(u_n(r))r
\]
for all $n\in \mathbb{N}$. Then integrating over $[0,r]$, we get \eqref{id0}. This finishes the proof. 
\end{proof}
Next ones are key identities often used for our concentration analysis.  
\begin{lemma}\label{lem:id} We get that for all $0<r\le 1$,
\begin{equation}\label{id1}
u_n(0)-u_n(r)=\int_0^r \la_nh(s)f(u_n(s))s\log{\frac rs}ds
\end{equation}
and for all $0<r<s\le1$,
\begin{equation}\label{id2}
\begin{split}
u_n(r)-u_n(s)&=\log{\frac sr}\int_0^r \la_nh(t)f(u_n(t))tdt\\
&\ \ \ \ \ \ \ \ \ \ +\int_r^s \la_nh(t)f(u_n(t))t\log{\frac st}dt
\end{split}
\end{equation}
for all $n\in \mathbb{N}$.
\end{lemma}
\begin{proof} Multiplying the equation in \eqref{pn} by $r\log{r}$, we see
\[
(-ru_n'(r))'\log{r}=\la_nf(u_n(r))r\log{r}
\]
for all $n\in \mathbb{N}$. Then integrating by parts and using \eqref{id0}, we get the desired formulas. For some more details, we refer the reader to Lemma 2.5 in \cite{N1}. We complete the proof. 
\end{proof}
We next give the Pohozaev identity and its consequences. We set $F(t)=\int_0^tf(s)ds$ as  in Lemma \ref{lem:ku}.  
\begin{lemma}\label{po}We have that
\[
(ru_n'(r))^2=4\int_0^r \la_n \left(1+\frac{sh'(s)}{2h(s)}\right)h(s)F(u_n(s))sds-2\la_n h(r)F(u_n(r))r^2
\] 
for all $n\in \mathbb{N}$.
\end{lemma} 
\begin{proof} Multiplying the equation in \eqref{pn} by $r^2 u'$, we have
\[
-\frac12\{(ru_n'(r))^2\}'=\la h(r)(F(u_n(r)))'r^2
\]
for all $r\in[0,1]$. Then, integrating by parts, we prove the desired formula. We finish the proof. 
\end{proof}
Using this, we get the following. 
\begin{lemma}\label{lem:gl}
Assume $\mu_n\to \infty$ as $n\to \infty$. Then $(u_n)$ is locally uniformly bounded in $(0,1]$. In particular, for any sequence $(r_n)\subset (0,1)$ such that $u_n(r_n)\to\infty$ as $n\to \infty$,  we have that $r_n\to0$ as $n\to \infty$. 
\end{lemma}
\begin{proof} %We argue as in the proof of {\color{red}Lemma 2.9 in \cite{N1}}.
 Noting Lemma \ref{lem:monog} and the de l'H\^{o}spital rule, we get
\[
\lim_{t\to \infty}\frac{F(t)}{f(t)}=\lim_{t\to \infty}\frac{1}{g'(t)}=0.
\]
 Hence for any $\e>0$, there exists a constant $t_\e>0$ such that 
\[
F(t)\le \e f(t)
\]
for all $t\ge t_\e$. It follows from Lemma \ref{po} with \eqref{id0} that 
\begin{equation}\label{eq:gl2}
\begin{split}
&\left(\int_0^1\la_n hf(u_n)rdr\right)^2\\
&\ \ \ \ \ \ \ \le C\e\int_0^1\la_n hf(u_n)rdr+C\la_n \max_{0\le r\le 1}h(r)\max_{0\le t\le t_\e}F(t)
\end{split}
\end{equation}
where $C= 4\max_{0\le r\le 1}|1+rh'(r)/h(r)|$. 
Now, we argue by contradiction. If the conclusion fails, then there exist a sequence $(r_n)\subset (0,1)$ and a value $r_0\in(0,1]$ such that $r_n\to r_0$ and $u_n(r_n)\to \infty$ as $n\to \infty$ up to a subsequence. Then,  noting the monotonicity of $u_n(r)$ with respect to $r$ and the properties of $f(t)$ proved in  Lemma \ref{lem:monog}, we get that 
\[
\lim_{n\to \infty}\int_0^1\ hf(u_n)rdr=\infty.
\]
Hence \eqref{id0} and \eqref{eq:gl2}  imply that there exists a value $C>0$ such that 
\[
-ru_n'(r)\le  \int_0^1\ \la_n hf(u_n)rdr\le C
\] 
for all $r\in[0,1]$ and all $n\in \mathbb{N}$. As a consequence, multiplying this inequality by $1/r$ and integrating over $[r_n,1]$, we get
\[
u_n(r_n)\le C\log{\frac{1}{r_0}}+o(1)
\] 
as $n\to \infty$. This is a contradiction. We finish the proof. 
\end{proof}
\subsection{Scaling and energy functions}\label{sub:D}
Now, we introduce some key tools to detect concentrating parts. In the following, let $\{(\la_n,\mu_n,u_n)\}$ be any sequence of solutions of \eqref{pn}. Then we define ``the scaling function"
\[
\phi_n(r):=\la_nr^2h(r) f'(u_n(r))
\]
and  ``the gradient function"
\[
\psi_n(r):=-rg'(u_n(r))u_n'(r)
\]
for all $r\in [0,1]$ and all $n\in \mathbb{N}$ such that $u_n(r)\ge t_0$. Note that from \eqref{id0}, we can interpret $\psi_n$ as ``the energy function" in the sense of the formula
\[
\psi_n(r)=g'(u_n(r))\int_0^r\la_nhf(u_n)sds
\]
for all $r\in [0,1]$ and all $n\in \mathbb{N}$ as above. The following argument is an extension of the radial analysis in Appendix, especially Lemma 3,  in \cite{D}.  (We also remark that the study of similar quantities plays important role in related works on supercritical problems with different approaches, see the quantity $H$ in Section 2 and the estimates in  Lemma 3.1 in \cite{McMc} and  (3.3) in Lemma 3.2 in \cite{Ku}.)

We first give the next lemma. 
\begin{lemma}\label{lem:D1}
We have that
\[
\begin{split}
&\phi_n'(r)=\frac{\phi_n(r)}{r}\left[2-\psi_n(r)\left\{1+\frac{g''(u_n(r))}{g'(u_n(r))^2}-\frac{rh'(r)}{h(r)\psi_n(r)}\right\}\right]
\end{split}
\]
for all $r\in[0,1]$ and all $n\in \mathbb{N}$ such that $u_n(r)\ge t_0$. Moreover,  assuming (H1) and $(r_n)\subset (0,1)$ is a sequence of values such that $u_n(r_n)\to \infty$ as $n\to \infty$, we get
\[
\begin{split}
&\phi_n'(r_n) =\frac{\phi_n(r_n)}{r_n}\left\{2-\psi_n(r_n)\left(1+o(1)\right)\right\}
\end{split}
\]
as $n\to \infty$ where we additionally supposed that $\liminf_{n\to \infty}\psi_n(r_n)>0$ if $h$ is not constant. In particular, if $\phi_n'(r_n)=0$ for all large $n\in \mathbb{N}$, then we obtain that 
\begin{equation}\label{eq:enid}
\lim_{n\to \infty}\psi_n(r_n)=2
\end{equation}
\end{lemma}
\begin{proof} The former  conclusion is a consequence of a direct calculation. Then the latter one follows from the former one with  \eqref{g1} and the facts that $g(u_n(r_n))\to \infty$ and $r_n\to0$ as $n\to \infty$ by Lemmas \ref{lem:monog} and \ref{lem:gl} respectively. We complete the proof.   
\end{proof}
\begin{lemma}\label{lem:D2} Assume (H1). Let $(r_n)$ and $(s_n)$ be sequences of values such that $0<r_n<s_n<1$ and there exists a   value $c_0>0$ such that $g'(u_n(s_n))/g'(\mu_n)\ge c_0$ for all $n\in \mathbb{N}$. Moreover suppose that $\sup_{r\in [r_n,s_n]}\phi_n(r)\to 0$ and there exists a constant $c_1\in (0,2)$ such that $\psi_n(r_n)\to c_1$ as $n\to \infty$. Then we get that 
\[
\lim_{n\to \infty}g'(\mu_n)\int_{r_n}^{s_n}\la_n hf(u_n)rdr=0.
\]
\end{lemma}
\begin{proof} First, since $g'(t)\to \infty$ as $t\to \infty$ by Lemma \ref{lem:monog}, we have that $u_n(s_n)\to \infty$ as $n\to \infty$ by the assumption and thus,  $s_n\to0$ as $n\to \infty$ by Lemma \ref{lem:gl}. Then, it follows  from the monotonicity of $g'$ by Lemma \ref{lem:monog} that 
\[
\inf_{r\in[r_n,s_n]}\psi_n(r)\ge \frac{g'(u_n(s_n))}{g'(\mu_n)}\psi_n(r_n)\ge \frac{c_0 c_1}2>0
\]
for all large $n\in \mathbb{N}$. Now, if the conclusion fails, there exists a sequence $(t_n)\subset (r_n,s_n)$ such that 
\[
\lim_{n\to \infty}g'(\mu_n)\int_{r_n}^{t_n}\la_n hf(u_n)rdr\in (0,2-c_1)
\]
up to a subsequence. Then using the assumption, the facts proved above, and Lemma \ref{lem:D1}, we find a sequence $(\bar{t}_n)\subset (r_n,t_n)$ such that
\[
\begin{split}
o(1)&=\phi_n(t_n)-\phi_n(r_n) =\int_{r_n}^{t_n}\frac{\phi_n(r)}{r}\left\{2-\psi_n(r)\left(1+o(1)\right)\right\}dr\\
&=\left\{2-\psi_n(\bar{t}_n)\left(1+o(1)\right)\right\}\int_{r_n}^{t_n}\frac{\phi_n(r)}{r}dr
\end{split}
\]
for all $n\in \mathbb{N}$.  Here noting the monotonicity of $g'$ again and our assumption, we estimate
\[
\begin{split}
\psi_n(\bar{t}_n)&= \frac{g'(u_n(\bar{t}_n))}{g'(u_n(r_n))}\psi_n(r_n)+\frac{g'(u_n(\bar{t}_n))}{g'(\mu_n)}g'(\mu_n)\int_{r_n}^{\bar{t}_n}\la_n hf(u_n)rdr\\
&\le c_1+g'(\mu_n)\int_{r_n}^{\bar{t}_n}\la_n hf(u_n)rdr+o(1)
\end{split}
\]
for all large $n\in \mathbb{N}$. Using this for the previous formula, we calculate
\[
\begin{split}
o(1)\ge
  c_0\left(2-c_1-g'(\mu_n)\int_{r_n}^{t_n}\la_n hf(u_n)rdr+o(1)\right)g'(\mu_n)\int_{r_n}^{t_n}\la_n hf(u_n)rdr
\end{split}
\]
for all large $n\in \mathbb{N}$. From out choice of $(t_n)$, we arrive at  a contradiction.  This finishes the proof. 
\end{proof}
\subsection{Notes on energy recurrence formulas} We finally summarize some basic properties of the energy recurrence formulas. In this subsection, we let $q\in[1,2)$ and $p\in (2,\infty]$ be quantities such that $1/p+1/q=1$ where $1/\infty=0$ and take the sequences $(a_k)$, $(\delta_k)$, $(\eta_k)$, and $(\tilde{\eta}_k)$ of values defined in Section \ref{sub:bf}. The next lemma is  proved by Lemmas 2.12 in \cite{N1} and 6.6 in \cite{N2}.
\begin{lemma}\label{lem:b1} Assume $q\in(1,2)$ and set $a_1=2$ and $\delta_1=1$. Then, for all $k\in \mathbb{N}$ with $k\ge2$, the numbers $a_k\in (0,2)$ and  $\delta_k\in (0,\delta_{k-1})$ are well-defined by \eqref{b1} and \eqref{b2}. Suppose $q=1$ and put $a_1=2$ and $\eta_1=1$. Then for every $k\in \mathbb{N}$ with $k\ge2$, the values $a_k\in (0,2)$ and $\eta_k\in  (0,\eta_{k-1})$ are well-defined by \eqref{b3} and \eqref{b4}.
\end{lemma}
The next one gives the basic properties of the sequences.  
\begin{lemma}\label{lem:b2} If $q\in(1,2)$, then the sequences $(a_k)$ and $(\delta_k)$ are strictly decreasing and  converge to zero as $k\to \infty$ and $\sum_{k=1}^\infty a_k=\infty$. If $q=1$, then the sequences $(a_k)$ and $(\eta_k)$ are strictly decreasing and converge to zero as $k\to \infty$ and $\sum_{k=1}^\infty a_k=\infty$. 
\end{lemma}
\begin{proof} The proof for the case $q>1$ is given by Lemmas 4.16 and 4.18 in \cite{N1}. The case $q=1$ is proved by Lemma 6.9 and 6.10  in  \cite{N2} except for the convergence  of $(\eta_k)$. To prove this, first note that $(\hat{c}_k)_{k\in \mathbb{N}\cup\{0\}}$ in \cite{N2} corresponds to $(\eta_{k})_{k\in \mathbb{N}}$ in the present paper. Then, as in the proof of Lemmas 6.10 in \cite{N2}, setting $\eta_{k+1}/\eta_k=1-\e_k$, we have that $\e_k\to0$ as $k\to \infty$ and there exist values $c>0$ and $k_0\in \mathbb{N}$ such that  $\e_k\ge c/k$ for all $k\ge k_0$. Then taking $k_0$ larger if necessary, we get that for all $k>k_0$,
\[
\begin{split}
\log{\frac1{\eta_k}}&%=\log{\left(\frac{\eta_{k-1}}{\eta_k}\frac{\eta_{k-2}}{\eta_{k-1}}\cdots \frac{\eta_{k_0}}{\eta_{k_0+1}} \frac1\eta_{k_0}\right)}
=\log{\frac1{\eta_{k_0}}}+\sum_{i=k_0}^{k-1}\log{\frac{\eta_{i}}{\eta_{i+1}}}%\\=\log{\frac1{\eta_{k_0}}}+\sum_{i=k_0}^{k-1}\log{\frac1{1-\e_i}}\\%\ge \log{\frac1{\eta_{k_0}}}+\sum_{i=k_0+1}^{k}\log{\frac1{1-\e_i}}\\&
\ge  \log{\frac1{\eta_{k_0}}}+c \sum_{i=k_0}^{k-1}\frac1i\to \infty
\end{split}
\]  
as $k\to \infty$. This proves the desired conclusion. We finish the proof. 
\end{proof}
We also use the following relation.  
\begin{lemma}\label{lem:b3} Assume $q\in [1,2)$. Then we have that
\[
\displaystyle \tilde{\eta}_k\sum_{i=1}^k \frac{2a_i}{\tilde{\eta}_i}=2+a_k
\]
for all $k\in \mathbb{N}$.
\end{lemma} 
\begin{proof} If $q>1$,  we obtain the proof by Lemma 2.13 in \cite{N1}. For the case $q=1$, we argue by induction. It is clear that the formula holds for $k=1$ since $a_1=2$. Assume that  the formula is true for some $k\ge1$. Then using the assumption and \eqref{b4}, we get
\[
\eta_{k+1}\sum_{i=1}^{k+1} \frac{2a_i}{\eta_i}=2a_{k+1}+\frac{\eta_{k+1}}{\eta_k}(2+a_k)=2+a_{k+1}.
\] 
This  finishes the proof.
\end{proof}

\section{First concentration}\label{sec:fc}
Let us start our concentration analysis. We begin with the detection of the first bubble around the maximum point. In this section, we always assume (H1) with $q\ge1$ and  $\{(\la_n,\mu_n,u_n)\}$ is a sequence of solutions of \eqref{pn} such that $\mu_n\to \infty$ as $n\to \infty$ without further comments. 
\subsection{Detection of standard bubble}
In this subsection, we detect the first concentrating part and prove Theorem \ref{th:0}. The next lemma gives the basic observation. 
\begin{lemma}\label{lem:c1}  Set sequences $(\ga_{0,n})$ of positive values and $(z_{0,n})$ of functions so that 
\[
\la_n h(0)f'(\mu_n)\ga_{0,n}^2=1
\]
and 
\[
z_{0,n}(r)=g'(\mu_n)(u_n(\ga_{0,n}r)-\mu_n)
\]
for all $r\in [0,1/\ga_{0,n}]$ and all large $n\in \mathbb{N}$. Then we have that $\ga_{0,n}\to 0$ and $z_{0,n}\to z_0$ in $C^2_{\text{loc}}([0,\infty))$ as $n\to \infty$ up to a subsequence where $z_0$ is defined by \eqref{def:z0}. Moreover, we get
\[
\lim_{R\to \infty}\lim_{n\to \infty}g'(\mu_n)\int_0^{R \ga_{0,n}}\la_n hf(u_n)rdr=4.
\]
\end{lemma}
\begin{proof} First we claim $\ga_{0,n}\to0$ as $n\to \infty$. If not, it follows from \eqref{pn} that there exists a value $C_0>0$ such that  
\[
g'(\mu_n)(-ru_n'(r))'=\frac{1}{\ga_{0,n}^2}\frac{h(r)}{h(0)}\frac{f(u_n(r))}{f(\mu_n)}r\le C_0 r
\]
for all $r\in[0,1]$ and all $n\in \mathbb{N}$ up to a subsequence. Integrating the inequality, we get 
\[
g'(\mu_n)(-u_n'(r))\le \frac{C_0r}{2}
\]
for all $r\in[0,1]$ and all  $n\in \mathbb{N}$. Then, integrating over $[0,1]$ gives 
\[
g'(\mu_n)\mu_n\le \frac{C_0}{2}
\]
for all $n\in \mathbb{N}$ which contradicts  the fact $g'(\mu_n)\to \infty$ as $n\to \infty$ ensured in Lemma \ref{lem:monog}. This proves the claim. Next, from \eqref{pn} and the monotonicity of $g$ noted in Lemma \ref{lem:monog}, we get a value $C_1>0$ such that
\begin{equation}\label{eq:z0n}
\begin{cases}
-z_{0,n}''-\frac1r z_{0,n}'=\frac{h(\ga_{0,n}\cdot)}{h(0)}\frac{f\left(\mu_n+\frac{z_{0,n}}{g'(\mu_n)}\right)}{f(\mu_n)}\le C_1\text{ in }(0,1/\ga_{0,n}),\\
z_{0,n}(0)=0=z_{0,n}'(0),
\end{cases}
\end{equation}
for all large $n\in \mathbb{N}$. We claim that $(z_{0,n})$ is bounded in $C^2_{\text{loc}}([0,\infty))$. To prove this, take any number $r>0$. Multiplying the inequality in \eqref{eq:z0n}  by $r$ and integrating over $[0,r]$, we get
\[
0\le -z_{0,n}'(r)\le \frac {C_1r}2\text{ \ and \ }0\le -z_{0,n}(r)\le \frac{C_1r^2}4
\] 
for all large $n\in \mathbb{N}$. This with the equation in \eqref{eq:z0n} proves the claim. Hence, by the Ascoli-Arzel\`{a} theorem and again  the equation with Lemma \ref{lem:g1}, there exists a function $z$ such that $z_{0,n}\to z$ in $C_{\text{loc}}^1([0,\infty))\cap C^2_{\text{loc}}((0,\infty))$ as $n\to \infty$ and %. It follows from the equation in \eqref{eq:z0n} and Lemma \ref{lem:g1} that
\[%begin{equation}\label{eq:z2n}
\begin{cases}
-z''-\frac1r z'=e^{z}\text{ in }(0,\infty),\\
z(0)=0=z'(0).
\end{cases}
\]%end{equation}
Integrating the equation with the initial conditions, we get $z=z_0$. (For the detail, see the final part of the proof of Lemma 4.3 in \cite{GN}.) We note that  the monotonicity of $f$ and Lemma \ref{lem:g1} show that for any $r\ge0 $
\[
\begin{split}
 \frac{(1+o(1))e^{z_0(r)}r^2}{2}%\frac{f\left(\mu_n+\frac{z_{0,n}(r)}{g'(\mu_n)}\right)}{f(\mu_n)}r^2
\le \int_0^r \frac{h(\ga_{0,n}r)}{h(0)}\frac{f\left(\mu_n+\frac{z_{0,n}(s)}{g'(\mu_n)}\right)}{f(\mu_n)}sds\le \frac{(1+o(1)r^2}{2}\end{split}
\]
which with the equation implies 
\begin{equation}\label{eq:xxx}
\frac{(1+o(1))e^{z_0(r)}}{2}\le -\frac{z_{0,n}'(r)}{r}\le \frac{1+o(1)}2
\end{equation}
where $o(1)\to0$ as $n\to \infty$ locally  uniformly in $[0,\infty)$.  This gives that $z_{0,n}'(r)/r\to z_0'(r)/r$ locally uniformly for all $r\in[0,\infty)$. In fact, if the conclusion fails, noting the local uniform convergence of $z_{0,n}'$ proved above, we find a sequence $(r_n)\subset [0,\infty)$ such that $r_n\to0$ and $z_{0,n}'(r_n)/r_n- z_0'(r_n)/r_n\not \to 0$ as $n\to \infty$. This clearly contradicts \eqref{eq:xxx}. Hence, the equation shows the convergence of $z_{0,n}$ actually holds in $C^2_{\text{loc}}([0,\infty))$.  Moreover, for all $R>0$, we have  that
\[
g'(\mu_n)\int_0^{R\ga_{0,n}}\la_n hf(u_n)rdr=\int_0^R\frac{h(\ga_{0,n}r)}{h(0)}\frac{f(u_n(\ga_{0,n}r))}{f(\mu_n)}rdr\to \int_0^Re^{z_0}rdr
\]  
as $n\to \infty$ by Lemmas \ref{lem:g1}. This proves the final conclusion. We finish the proof.
\end{proof}
We slightly extend the first concentration interval and give the pointwise estimate on the outside of it.
\begin{lemma}\label{lem:c2} Define the sequences $(\ga_{0,n})$ and $(z_{0,n})$ as in the previous lemma.  Then, up to a subsequence, there exists a sequence $(\rho_{0,n})\subset (0,1)$ of values such that $u_n(\rho_{0,n})/\mu_n\to1$, $\rho_{0,n}/\ga_{0,n}\to \infty$, $g(\mu_n)^{-1}\log{(\rho_{0,n}/\ga_{0,n})}\to 0$,
$
\|z_{0,n}-z_0\|_{C^2([0,\rho_{0,n}/\ga_{0,n}])}\to0,
$
and 
\[
g'(\mu_n)\int_0^{\rho_{0,n}}\la_nhf(u_n)rdr\to 4
\]
as $n\to \infty$. Moreover, we have that
\[
z_{0,n}(r)\le -(4+o(1))\log{r}
\]
for all $r\in[\rho_{0,n}/\ga_{0,n},1/\ga_{0,n}]$ and all $n\in \mathbb{N}$ where $o(1)\to 0$ as $n\to \infty$ uniformly for $r$ in the interval. 
\end{lemma}
\begin{proof} From Lemma \ref{lem:c1}, there exists a sequence $(R_n)$ of positive values such that $R_n\to \infty$, $g(\mu_n)^{-1}\log{R_n}\to0$, $(\mu_ng'(\mu_n))^{-1}\log{R_n}\to0$, 
\[
\|z_{0,n}-z_0\|_{C^2([0,R_n])}\to0,\ \ \ \sup_{r\in[0,R_n]}|rz_{0,n}'(r)-rz_0'(r)|\to0,
\]
and 
\[
g'(\mu_n)\int_0^{R_n\ga_{0,n}}\la_nhf(u_n)rdr\to 4
\]
as $n\to \infty$. In particular, we have
\[
\frac{u_n(\ga_{0,n}R_n)}{\mu_n}=1+\frac{z_0(R_n)+o(1)}{\mu_ng'(\mu_n)}%=1+\frac{-4(1+o(1))\log{R_n}+o(1)}{\mu_n}
\to1
\]
and thus, $\rho_{0,n}\to0$ by Lemma \ref{lem:gl} as $n\to \infty$. Moreover, recalling \eqref{eq:z0n}, we have that $rz_{0,n}'(r)$ is nonincreasing and thus, deduce for all $r\in[R_n,1/\ga_{0,n}]$ that
\[
rz_{0,n}'(r)\le R_nz_{0,n}'(R_n)=R_n z_0'(R_n)+o(1)=-4+o(1)
\]
as $n\to \infty$. After integration, we see that
\[
z_{0,n}(r)\le z_0(R_n)-(4+o(1))\log{\frac r{R_n}}\le -(4+o(1))\log{r}
\]
as $n\to \infty$. Consequently, putting $\rho_{0,n}=R_n \ga_{0,n}$ for all $n\in \mathbb{N}$, we complete all the assertions. We finish the proof. 
\end{proof}
%\subsection{Proof of Theorems \ref{th:0}}
%In this subsection,
Now, let us prove Theorem \ref{th:0}. 
\begin{proof}[Proof of Theorem \ref{th:0}]  The proof follows from Lemmas \ref{lem:c1} and \ref{lem:c2} except for the last equality. For the proof, noting the monotonicity of $g'$, we estimate
\[
\begin{split}
g'(\mu_n)\int_0^{\rho_{0,n}}\la_n hf(u_n)rdr&\ge \int_0^{\rho_{0,n}}\la_n hf'(u_n)rdr\\
&=\int_0^{\rho_{0,n}/\ga_{0,n}}\frac{h(\ga_{0,n}r)}{h(0)}\frac{f'(u_n(\ga_{0,n}r))}{f'(\mu_n)}rdr%\to \int_0^Re^{z_0}rdr
\end{split}
\]  
for all large $n\in \mathbb{N}$. Then by Lemmas \ref{lem:c1}, \ref{lem:c2}, \ref{lem:g1}, and the Fatou lemma, we get the desired formula. We complete the proof.  
\end{proof}
\subsection{Center of the first concentration}\label{sub:cfc}
To describe all the concentration parts in a unified way with the sequence $(z_k)_{k\in \mathbb{N}}$ of the limit profiles, we slightly modify the results in the previous subsection. We first look for the ``center" of the first concentration. Recall the sequences $(\phi_n)$ and $(\psi_n)$ of functions  in Subsection \ref{sub:D}. 
\begin{lemma}\label{lem:c5} There exists a sequence $(r_{1,n})\subset (0,1)$ of values such that $\phi_n(r_{1,n})$ attains a local maximum value  of $\phi_n$, in particular  $\phi_n'(r_{1,n})=0$,  for all $n\in\mathbb{N}$ and $\phi_n(r_{1,n})\to 2$, $\psi_n(r_{1,n})\to2$, $r_{1,n}/\ga_{0,n}\to 2\sqrt{2}$, 
\[
u_n(r_{1,n})=\mu_n+O\left(\frac1{g(\mu_n)}\right)\frac{g(\mu_n)}{g'(\mu_n)},
\]
$g(u_n(r_{1,n}))/g(\mu_n)\to1$, and $g'(u_n(r_{1,n}))/g'(\mu_n)\to1$  as $n\to \infty$ up to a subsequence. 
\end{lemma}
\begin{proof} Set any number $R>0$. Then we have from Lemmas \ref{lem:c1}  and \ref{lem:g1} that 
\[
\begin{split}
\phi_n(\ga_{0,n}R)&=R^2\frac{h(\ga_{0,n}R)}{h(0)}\frac{f'(u_n(\ga_{0,n}R))}{f'(\mu_{n})}=(1+o(1))\frac{64 R^2}{(8+R^2)^2}
\end{split}
\]
as $n\to \infty$  where $o(1)\to0$ as $n\to \infty$ uniformly for all $R$ in every compact subset of $[0,\infty)$. Note that the function $L(x)=64x^2/(8+x^2)^2$ for $x>0$ has a unique maximum point $R_*=2\sqrt{2}$ with $L(R_*)=2$. Hence there exists a sequence $(r_{1,n})\subset (0,1)$ of values such that $\phi_n(r_{1,n})$ gives a local maximum value of $\phi_n$, in particular $\phi_n'(r_{1,n})=0$, for every large $n\in \mathbb{N}$ and   $\phi_n(r_{1,n})\to2 $ and $r_{1,n}/\ga_{0,n}\to 2\sqrt{2}$ as $n\to \infty$. Then, we also get by Lemma \ref{lem:c1} that 
\[
u_n(r_{1,n})=\mu_n+\frac{z_0(2\sqrt{2})+o(1)}{g'(\mu_n)}=\mu_n+O\left(\frac1{g(\mu_n)}\right)\frac{g(\mu_n)}{g'(\mu_n)}%=1+o(1)%\frac{z_0(2\sqrt{2})+o(1)}{\mu_ng'(\mu_n)}
\]
as $n\to \infty$. In particular, \eqref{g2} shows that $u_n(r_{1,n})/\mu_n\to1$ as $n\to \infty$. Therefore, it follows  that $g(u_n(r_{1,n}))/g(\mu_n)\to1$ and $g'(u_n(r_{1,n}))/g'(\mu_n)\to1$ from Lemma \ref{lem:g3} and \ref{lem:g4} and $r_{1,n}\to0$ by  Lemma \ref{lem:gl}  as $n\to \infty$.  In addition, Lemmas \ref{lem:c1} and \ref{lem:g1} and the Fatou lemma yield
\[
\begin{split}
\liminf_{n\to \infty}\psi_n(r_{1,n})%&= \liminf_{n\to \infty}\left\{(1+o(1))g'(\mu_n)\int_0^{r_{1,n}}\la_n hf(u_n)rdr\right\}\\
&= \liminf_{n\to \infty}\left\{(1+o(1))\int_0^{r_{1,n}/\ga_{0,n}}\frac{h(\ga_{0,n}r)}{h(0)}\frac{f(u_n(\ga_{0,n}r))}{f(\mu_n)}rdr\right\}\\
&\ge \int_0^{2\sqrt{2}}e^{z_0}rdr>0.
\end{split}
\]
Consequently, Lemma \ref{lem:D1} gives that $\psi_n(r_{1,n})\to2$ as $n\to \infty$.  We complete the proof. 
\end{proof}
In the following, we let $(r_{1,n})$ be the sequence  in the previous lemma and put $\mu_{1,n}=u_n(r_{1,n})$ for all $n\in \mathbb{N}$. We re-scale the first concentration part and deduce the limit profile  around the points $(r_{1,n})$ as follows.   
\begin{lemma}\label{lem:c6} Define sequences $(\ga_{1,n})$ of values and $(z_{1,n})$ of functions so that for every large $n\in\mathbb{N}$,
\[
\la_n h(r_{1,n})f'(\mu_{1,n})\ga_{1,n}^2=1
\]
and 
\[
z_{1,n}(r)=g'(\mu_{1,n})(u_n(\ga_{1,n}r)-\mu_{1,n})
\]
for all $r\in[0,1/\ga_{1,n}]$. Then we have that $\ga_{1,n}\to0$  and $z_{1,n}\to z_1$ in $C^2_{\text{loc}}([0,\infty))$ as $n\to \infty$ up to a subsequence where $z_1$ is the function defined by \eqref{def:zk} with $k=1$. Moreover, we get
\[
\begin{split}\lim_{R\to \infty}\lim_{n\to \infty}&g'(\mu_{1,n})\int_{0}^{R\ga_{1,n}}\la_n f(u_n)rdr\\
&=4=\lim_{R\to \infty}\lim_{n\to \infty}\int_{0}^{R\ga_{1,n}}\la_n g'(u_n)f(u_n)rdr.
\end{split}
\]
Finally, we obtain $\phi_n(\ga_{1,n} r)/( r^2e^{z_1(r)})\to1$ as $n\to \infty$ uniformly for all $r$ in every compact subset of $[0,\infty)$. 
\end{lemma}
\begin{proof} Since $(r_{1,n}/\ga_{1,n})^2=\phi_n(r_{1,n})\to 2$, we have that $\ga_{1,n}\to 0$ as $n\to \infty$. Moreover, from the definition and \eqref{pn}, we get that  
\begin{equation}\label{z1n}
\begin{cases}
\displaystyle-z_{1,n}''-\frac1r z_{1,n}'=\frac{h(\ga_{1,n}\cdot)}{h(r_{1,n})}\frac{f(u_n(\ga_{1,n} \cdot))}{f(\mu_{1,n})} \text{ in }(0,1/\ga_{1,n}),\\
z_{1,n}(r_{1,n}/\ga_{1,n})=0,\ (r_{1,n}/\ga_{1,n})z_{1,n}'(r_{1,n}/\ga_{1,n})=-2+o(1),
\end{cases}
\end{equation}
since $\psi_n(r_{1,n})=2+o(1)$ as $n\to \infty$ by Lemma \ref{lem:c5} and that
\[
z_{1,n}(r)=\frac{g'(\mu_{1,n})}{g'(\mu_n)}\left(z_{0,n}(\ga_{1,n}r/\ga_{0,n})-z_{0,n}(r_{1,n}/\ga_{0,n})\right)
\]
for all $r\in[0,1/\ga_{1,n}]$ and all large $n\in \mathbb{N}$. Hence from Lemmas \ref{lem:c1} and \ref{lem:c5}, we derive that $z_{1,n}\to z_1$ in $ C^2_{\text{loc}}([0,\infty))$ as $n\to \infty$. Moreover, for any value $R>0$, we get from Lemma  \ref{lem:g1} that
\[
\begin{split}
\lim_{n\to \infty}g'(\mu_{1,n})\int_{0}^{R\ga_{1,n}}\la_n hf(u_n)rdr&=\lim_{n\to \infty}\int_{0 }^{R}\frac{h(\ga_{1,n}r)}{h(r_{1,n})}\frac{f(u_n(\ga_{1,n}r))}{f(\mu_{1,n})}rdr\\
&=\int_0^Re^{z_1}rdr
\end{split}
\]
and similarly, 
\[
\begin{split}
\lim_{n\to \infty}\int_{0}^{R\ga_{1,n}}\la_n hf'(u_n)rdr=\int_0^Re^{z_1}rdr.
\end{split}
\]
Finally, choose any compact subset $K\in [0,\infty)$. Then we have from Lemma \ref{lem:g1} again that
\[
\begin{split}
\sup_{r\in K}\left|\frac{\phi_n(\ga_{1,n} r)}{ r^2e^{z_1(r)}}-1\right|&=\sup_{r\in K}\left|\frac{\frac{h(\ga_{1,n}r)}{h(r_{1,n})}\frac{f(u_n(\ga_{1,n} r))}{f(\mu_{1,n})}}{e^{z_1(r)}}-1\right|\\
&\to0
\end{split}
\]
as $n\to \infty$. We complete the proof. 
\end{proof}
%\begin{remark}\label{rmk:a1} As in the proof, since $r_{1,n}/\ga_{0,n}\to 2\sqrt{2}$ and $r_{1,n}/\ga_{1,n}\to \sqrt{2}$, we have that $\ga_{1,n}/\ga_{0,n}\to 2$ as $n\to \infty$.
%\end{remark}
We then extend the concentration region as before. 
\begin{lemma}\label{lem:c7} Assume and define as in the previous lemma. Then, there exists a sequence $(\rho_{1,n})\subset (0,1)$ of values such that $\rho_{1,n}\to0$, $\rho_{1,n}/\ga_{1,n}\to\infty$, $g(\mu_{1,n})^{-1}\log{(\rho_{1,n}/\ga_{1,n})}\to0$, $\|z_{1,n}-z\|_{C^2([0,\rho_{1,n}/\ga_{1,.n}])}\to 0,$ 
%\[
%\|z_{1,n}-z\|_{C^1_{\text{loc}}([0,\rho_{1,n}/\ga_{1,.n}])}\to 0, \ \ \ 
\[
\sup_{r\in[0,\rho_{1,n}/\ga_{1,n}]}|rz_{1,n}'(r)-rz_1'(r)|\to0,
\]
\[
 \sup_{r\in[0,\rho_{1,n}/\ga_{1,n}]}\left|\frac{h(\ga_{1,n}r)}{h(r_{1,n})}\frac{f(u_n(\ga_{1,n}r))}{f(\mu_{1,n})}-e^{z_1(r)}\right|\to0,
\]
\[
 \sup_{r\in[0,\rho_{1,n}/\ga_{1,n}]}\left|\frac{h(\ga_{1,n}r)}{h(r_{1,n})}\frac{f'(u_n(\ga_{1,n}r))}{f'(\mu_{1,n})}-e^{z_1(r)}\right|\to0,
\]
\[
\sup_{r\in[0,\rho_{1,n}/\ga_{1,n}]}\left|\frac{\phi_n(\ga_{1,n} r)}{ r^2e^{z_1(r)}}-1\right|\to0
\]
as $n\to \infty$ and
\[%begin{equation}\label{eq:en1}
\lim_{n\to \infty}g'(\mu_{1,n})\int_0^{\rho_{1,n}}\la_nhf(u_n)rdr=4=\lim_{n\to \infty}\int_0^{\rho_{1,n}}\la_nhf'(u_n)rdr.
\]%\end{equation}
Moreover, we have that
\[
u_n(\rho_{1,n})=\mu_{1,n}+o(1)\cdot \frac{g(\mu_{1,n})}{g'(\mu_{1,n})}=\mu_{n}+o(1)\cdot \frac{g(\mu_{n})}{g'(\mu_{n})},
\]
and  $\phi_n(\rho_{1,n})\to0$ as $n\to \infty$. Finally, we get
\begin{equation}\label{eq:pw1}
z_{1,n}(r)\le -(4+o(1))\log{r}
\end{equation}
for all $r\in[\rho_{1,n}/\ga_{1,n},1/\ga_{1,n}]$ and all $n\in \mathbb{N}$ where $o(1)\to 0$ as $n\to \infty$ uniformly for all $r$ in the interval. 
\end{lemma}
\begin{proof} From Lemmas \ref{lem:c6} and \ref{lem:g1}, there exists a sequence $(R_n)$ of positive values such that $R_n\to \infty$, $g(\mu_{1,n})^{-1}\log{R_n}\to0$, $\|z_{1,n}-z_1\|_{C^2([0,R_n])}\to 0$,
\[
\sup_{r\in[0,R_n]}|rz_{1,n}'(r)-rz_1'(r)|\to0,
\]
\[
 \sup_{r\in [0,R_n]}\left|\frac{h(\ga_{1,n}r)}{h(r_{1,n})}\frac{f(u_n(\ga_{1,n}r))}{f(\mu_{1,n})}-e^{z_1(r)}\right|\to0,
\]
\[
 \sup_{r\in [0,R_n]}\left|\frac{h(\ga_{1,n}r)}{h(r_{1,n})}\frac{f'(u_n(\ga_{1,n}r))}{f'(\mu_{1,n})}-e^{z_1(r)}\right|\to0,
\]
\[
\sup_{r\in[0,R_n]}\left|\frac{\phi_n(\ga_{1,n} r)}{ r^2e^{z_1(r)}}-1\right|\to0,
\]
and 
\[
g'(\mu_{1,n})\int_0^{R_n\ga_{1,n}}\la_n hf(u_n)rdr\to 4,\ \ \ \int_0^{R_n\ga_{1,n}}\la_n hf'(u_n)rdr\to 4,
\]
as $n\to \infty$. Then, we get  
\[
\begin{split}
\phi_n(\ga_{1,n}R_n)&=R_n^2\frac{h(\ga_{1,n}R_n)}{h(r_{k,n})}\frac{f'(u_n(\ga_{1,n}R_n))}{f'(\mu_{1,n})}=(1+o(1))\frac{16 R_n^2}{(2+R_n^2)^2}=o(1)
\end{split}
\]
and
\[
\begin{split}
u_n(\ga_{1,n}R_n)&=\mu_{1,n}+\frac{z_{1}(R_n)+o(1)}{g'(\mu_{1,n})}=\mu_{1,n}+o\left(1\right)\cdot\frac{g(\mu_{1,n})}{g'(\mu_{1,n})}\\
&=\mu_n+o\left(1\right)\cdot\frac{g(\mu_{n})}{g'(\mu_{n})}
\end{split}
\]
by Lemmas \ref{lem:c5}, \ref{lem:g3}, and \ref{lem:g4} as $n\to \infty$. Notice that Lemma \ref{lem:gl} yields $\ga_{1,n}R_n\to0$ as $n\to \infty$. Moreover, noting the monotonicity of $rz_{1,n}'(r)$ by \eqref{z1n}, we obtain for any $r\in[R_n,1/\ga_{1,n}]$,
\[
rz_{1,n}'(r)\le R_nz_1'(R_n)+o(1)=-4+o(1)
\] 
as $n\to \infty$. Integrating this inequality over $[R_n,r]$, we deduce
\[
z_{1,n}(r)\le z_1(R_n)-(4+o(1))\log{\frac{r}{R_n}}+o(1)\le -(4+o(1))\log{r}
\]
as $n\to \infty$. Hence, it suffices to put  $\rho_{1,n}=R_n\ga_{1,n}$ for all $n\in \mathbb{N}$. We finish the proof. 
\end{proof}
\section{Subsequent concentration}\label{sec:ic}
This section is devoted to the detection of the subsequent concentration which gives the core of our argument in this paper. The proof of our main theorems, that is, the proof of the presence of an infinite sequence of concentrating parts  in the supercritical case will be done by induction. Hence, the main step is, by assuming the presence of the first $k$ concentrating parts, to show the appearance of the $(k+1)$-th one. We shall demonstrate this step here. Throughout this section, we always assume (H1) with $q\in [1,2)$ and $\{(\la_n,\mu_n,u_n)\}$ is any sequence of solutions of \eqref{pn} such that $\mu_n\to \infty$ without further comments.  Recall the sequences $(a_k)$, $(\delta_k)$, $(\eta_k)$, and $(\tilde{\eta}_k)$ of numbers defined in Subsection \ref{sub:bf}.
\subsection{Detection of singular bubbles}
In the following, we fix $k\in \mathbb{N}$ and assume the presence of the first $k$ concentrating  parts  as follows.  %In this subsection, we always set  $q\le2$. If $q<2$,
%If $q=2$, we set $\delta_2=\eta_2=0$. We fix any $k\in \mathbb{N}$ if $q<2$ and $k=1$ if $q=2$ and assume the following.  %Choose the sequence $(\mu_{1,n})$ in Lemma.
\begin{enumerate}
\item[(A$_k$)] There exists a sequence $(r_{k,n})\subset (0,1)$ of values such that, putting $\mu_{k,n}=u_n(r_{k,n})$ for all $n\in \mathbb{N}$, we have that 
\[
\la_n r_{k,n}^2h(r_{k,n})f'(\mu_{k,n})\to \frac{a_k^2}{2},
\]
\[
\begin{cases}
\displaystyle\frac{\mu_{k,n}}{\mu_{n}}\to \delta_k&\text{ if }q\ge1,\\
\displaystyle\mu_{k,n}=\mu_n-\left(\log{\frac1{\eta_k}}+o(1)\right)\frac{g(\mu_{n})}{g'(\mu_{n})}&\text{ if }q=1,
\end{cases}
\]
as $n\to \infty$ and if we put sequences $(\ga_{k,n})$ of values and $(z_{k,n})$ of functions so that for each large $n\in \mathbb{N}$,
\[
\la_n h(r_{k,n})f'(\mu_{k,n})\ga_{k,n}^2=1
\]
and 
\[
z_{k,n}(r)=g'(\mu_{k,n})(u_n(\ga_{k,n}r)-\mu_{k,n})
\]
for all $r\in[0,1/\ga_{k,n}]$, then there exists a sequence $(\rho_{k,n})\subset(0,1)$ of numbers such that $\rho_{k,n}/\ga_{k,n}\to \infty$,  $g(\mu_{k,n})^{-1}\log{(\rho_{k,n}/\ga_{k,n})}\to0$, 
\[
\begin{cases}
\displaystyle\frac{u_n(\rho_{k,n})}{\mu_{n}}\to \delta_k&\text{ if }q\ge1,\\
\displaystyle u_n(\rho_{k,n})=\mu_{n}-\left(\log{\frac1{\eta_k}}+o(1)\right)\frac{g(\mu_{n})}{g'(\mu_{n})}&\text{ if }q=1,
\end{cases}
\]
%$u_n(\rho_{k,n})/\mu_{k,n}\to\delta_k$ if $q>1$ and $(\mu_{1,n}-u_n(\rho_{k,n}))g'(\mu_{1,n})/g(\mu_{1,n})\to \log{(1/\eta_k)}$ if $q=1$, 
\begin{equation}\label{as:en}
g'(\mu_{n})\int_0^{\rho_{k,n}}\la_n hf(u_n)rdr\to
%\begin{cases}\displaystyle \sum_{i=1}^n \frac{2a_i}{\delta_i^{p-1}}&\text{ if }q>1,\vspace{0.2cm}\\
\displaystyle \sum_{i=1}^k \frac{2a_i}{\tilde{\eta}_i},%&\text{ if }q=1,
%\end{cases}
\end{equation}
$\phi_n(\rho_{k,n})\to0$ as $n\to \infty$ and  
\begin{equation}\label{as:zk}
z_{k,n}(r)\le -(2+a_k+o(1))\log{r}
\end{equation}
for all $r\in[\rho_{k,n}/\ga_{k,n},1/\ga_{k,n}]$ and all $n\in \mathbb{N}$ where $o(1)\to0$ as $n\to \infty$ uniformly for all $r$ in the interval. 
\end{enumerate}
\begin{remark}\label{rmk:41} By the choice of $(r_{k,n})$ and $(\rho_{k,n})$ above, we notice that $u_n(\rho_{k,n})/\mu_{k,n}\to1$ for all $q\in[1.2)$ and
\[
\begin{split}
u_n(\rho_{k,n})%&=\mu_{k,n}+(\mu_n-\mu_{k,n})+\left(\log{\frac1{\eta_k}}+o(1)\right)\frac{g(\mu_n)}{g'(\mu_n)}\\
&
=\mu_{k,n}+o(1)\cdot \frac{g(\mu_{k,n})}{g'(\mu_{k,n})}
\end{split}
\] 
by Lemma \ref{lem:g4} if $q=1$ as $n\to \infty$.
\end{remark}
Notice that the argument in the previous subsection ensures that (A$_k$) is satisfied for $k=1$. From now on, we shall show the appearance of  the $(k+1)$-th concentrating part. We begin with  the next basic estimate.
\begin{lemma}\label{lem:c8} For any sequence $(r_n)\subset [\rho_{k,n},1)$ of values, we get that
\begin{equation}\label{eq:es1}
\frac{\log{\frac{r_n}{\ga_{k,n}}}}{g(\mu_{k,n})}\ge (1+o(1))\frac{g'(\mu_{k,n})}{g(\mu_{k,n})}\frac{\mu_{k,n}-u_n(r_n)}{g'(\mu_{k,n})\int_0^{r_n}\la_n f(u_n)rdr}
\end{equation}
and that if $\liminf_{n\to \infty}u_n(r_n)>t_0$, where $t_0$ is the number in our assumption (H1), then 
\begin{equation}\label{eq:es2}
\frac{\log{\frac{r_n}{\ga_{k,n}}}}{g(\mu_{k,n})}\le \frac12\left(1-\frac{g(u_n(r_n))}{g(\mu_{k,n})}+o(1)\right)
\end{equation}
as $n\to \infty$. 
\end{lemma}
\begin{proof} From \eqref{id2}, we get
\[
\begin{split} 
\mu_{k,n}-u_n(r_n)&\le \log{\frac{r_n}{r_{k,n}}} \int_0^{r_n}\la_n hf(u_n)rdr\\
&\le(1+o(1))\log{\frac{r_n}{\ga_{k,n}}}\int_0^{r_n}\la_n hf(u_n)rdr
\end{split}
\]
as $n\to \infty$ since $r_{k,n}/\ga_{k,n}\to a_k/\sqrt{2}$ and $r_n/\ga_{k,n}\to \infty$ as $n\to \infty$ by (A$_k$). This proves the former estimate. On the other hand, from \eqref{id1} and the monotonicity of $f$ noted in Lemma \ref{lem:monog} with our assumption, we find a value $c>0$ such that 
\begin{equation}\label{eq:c81}
\begin{split}
\mu_n-u_n(r_n)&=r_n^2\la_nh(r_n)f(u_n(r_n))\int_0^{1}\frac{h(r_n r)}{h(r_n)}\frac{f(u_n(r_nr))}{f(u_n(r_n))}r\log{\frac{1}{r}}dr\\
&\ge c\left(\frac{r_n}{\ga_{k,n}}\right)^2\frac{f(u_n(r_n))}{f'(\mu_{k,n})}
\end{split}
\end{equation}
for all large $n\in \mathbb{N}$. Here,  from  \eqref{as:zk}, we see
\[
\mu_{k,n}-u_n(r_n)\ge \frac{-z_{k,n}(\rho_{k,n}/\ga_{k,n})}{g'(\mu_{k,n})}\ge \frac{(2+a_k+o(1))\log{(\rho_{k,n}/\ga_{k,n}})}{g'(\mu_{k,n})}
\]
as $n\to \infty$. Hence, it follows from \eqref{eq:c81} with Lemma \ref{lem:monog} and (A$_k$) that
\[
\log{\frac{r_n}{\ga_{k,n}}}\le \frac{g(\mu_{k,n})}{2}\left(1-\frac{g(u_n(r_n))}{g(\mu_{k,n})}+o(1)\right)
\]
as $n\to \infty$. This proves the latter formula. We finish the proof. 
\end{proof}
Next, we roughly extend the region with no additional bubble.
\begin{lemma}\label{lem:c9} Choose any number $\eta\in(0,1)$ so that 
\[
a_k>\frac{(1-\eta)(2+a_k)^2}{4q\eta}\text{\ \  and \ \ } 
\begin{cases}
q \eta -p(1-\eta^{1/p})>0&\text{ if $q>1$},\\
\eta>\displaystyle\log{\frac1{\eta}} &\text{ if $q=1$}.
\end{cases}
\] 
%so that $q \eta -p(1-\eta^{1/p})>0$ if $q>1$ and $\eta/\log{(1/\eta)}>1$ if $q=1$.
Let $(r_n)\subset (\rho_{k,n},1)$ be any sequence of values such that 
\[
u_n(r_n)=
\begin{cases}
\eta^{1/p} \mu_{k,n}&\text{ if $q>1$},\\
\displaystyle\mu_{k,n}-\log{\frac1{\eta}} \frac{g(\mu_{k,n})}{g'(\mu_{k,n})}&\text{ if $q=1$} 
\end{cases}
\]
for all $n\in \mathbb{N}$ where this choice of $(r_n)$ is possible in view of Remark \ref{rmk:41}. Then we have that
\[
g'(\mu_{k,n})\int_{\rho_{k,n}}^{r_n}\la_nhf(u_n)rdr\to 0
\] 
as $n\to \infty$.
\end{lemma}
\begin{proof} From our choice of $(r_n)$ and Lemmas \ref{lem:g3} and \ref{lem:g4}, %there exists a constant $\delta\in(0,1]$ such that
we have that 
\begin{equation}\label{eq:s0}
\lim_{n\to \infty}\frac{g(u_n(r_n))}{g(\mu_{k,n})}= \eta.
\end{equation} %Moreover, we also get $\eta=\delta^p$ if $q>1$ and $\delta=1$ if $q=1$. 
Then, from (H0), there exist a constant $C>0$ and a sequence $(\theta_n) \subset (0,1)$ of values such that 
\[
\begin{split}
g'(\mu_{k,n})\int_{\rho_{k,n}}^{r_n}\la_nhf(u_n)rdr&=\int_{\rho_{k,n}/\ga_{k,n}}^{r_n/\ga_{k,n}}\frac{h(\ga_{k,n}r)}{h(r_{k,n})}e^{g\left(\mu_{k,n}+\frac{z_{k,n}(r)}{g'(\mu_{k,n})}\right)-g(\mu_{k,n})}rdr\\
%&\le C\int_{\rho_{k,n}/\ga_{k,n}}^{r_n/\ga_{k,n}}e^{g\left(\mu_{k,n}+\frac{z_{k,n}}{g'(\mu_{k,n})}\right)-g(\mu_{k,n})}rdr\\
&\le C\int_{\rho_{k,n}/\ga_{k,n}}^{r_n/\ga_{k,n}}e^{z_{k,n}(r)+\frac{g''\left(\mu_{k,n}+\frac{\theta_n z_{k,n}(r)}{g'(\mu_{k,n})}\right)}{2g'(\mu_{k,n})^2}z_{k,n}(r)^2}rdr\\
%&=C\int_{\rho_{k,n}/\ga_{k,n}}^{r_n/\ga_{k,n}}e^{z_{k,n}(r)+\frac{1}{2q+o(1)}\frac1{g\left(\mu_{k,n}+\frac{\theta_n z_{k,n}(r)}{g'(\mu_{k,n})}\right)}\left(\frac{g'\left(\mu_{k,n}+\frac{\theta_n z_{k,n}(r)}{g'(\mu_{k,n})}\right)z_{k,n}(r)}{g'(\mu_{k,n})}\right)^2}rdr\\
&\le C\int_{\rho_{k,n}/\ga_{k,n}}^{r_n/\ga_{k,n}}e^{z_{k,n}(r)+\frac{1+o(1)}{2q\eta g(\mu_{k,n})}z_{k,n}(r)^2}rdr
\end{split}
\]
for all large $n\in \mathbb{N}$ where for the last inequality, we noted 
\[
\inf_{r\in[\rho_{k,n}/\ga_{k,n},r_n/\ga_{k,n}]}\left(\mu_{k,n}+\frac{\theta_n z_{k,n}(r)}{g'(\mu_{k,n})}\right)\ge u_n(r_n)\to \infty
\]
as $n\to \infty$ and then, used \eqref{g1}, the monotonnicity of $g'$, and that of $g$ with \eqref{eq:s0}.  Here our choice of $(r_n)$ and $\eta$ and \eqref{g2} imply that there exists a constant $\e>0$ such that
\[
\begin{split}
z_{k,n}(r)&\ge g'(\mu_{k,n})(u_n(r_n)-\mu_{k,n})=
\begin{cases}
-(p(1-\eta^{1/p})+o(1))g(\mu_{k,n})&\text{ if }q>1,\\
-\log{(1/\eta)} g(\mu_{k,n})&\text{ if }q=1,\end{cases}\\
&> -(1-\e)q\eta g(\mu_{k,n}) 
\end{split}
\]
for all $r\in[ \rho_{k,n}/\ga_{k,n},r_n/\ga_{k,n}]$ and all large $n\in \mathbb{N}$. Hence from \eqref{as:zk}, there exists a sequence $(\e_n)$ of values such that $\e_n\to0$ as $n\to \infty$ and 
\[
\begin{split}
g'(\mu_{k,n})\int_{\rho_{k,n}}^{r_n}\la_nh&f(u_n)rdr\\&\le C\int_{\rho_{k,n}/\ga_{k,n}}^{r_n/\ga_{k,n}}e^{-(2+a_k+\e_n)\log{r}+\frac{1+o(1)}{2q\eta g(\mu_{k,n})}((2+a_k+\e_n)\log{r})^2}rdr\\
&=\frac{C}{2+a_k+\e_n}\int_{s_n}^{S_n}e^{-\frac{a_k+\e_n}{2+a_k+\e_n}s+\frac{1+o(1)}{2q\eta g(\mu_{k,n})}s^2}ds
\end{split}
\]
where we changed the variable with  $s=(2+a_k+\e_n)\log{r}$ and put $S_n=(2+a_k+\e_n)\log{(r_n/\ga_{k,n})}$ and $s_n=(2+a_k+\e_n)\log{(\rho_{k,n}/\ga_{k,n})}$ for all $n\in \mathbb{N}$. Then, setting  
\[
K_n=\sqrt{2q\eta g(\mu_{k,n})(1+o(1))^{-1}},
\]
we get 
\begin{equation}\label{eq:s1}
\begin{split}
g'(\mu_{k,n})&\int_{\rho_{k,n}}^{r_n}\la_nhf(u_n)rdr\\
&\le \frac{C}{2+a_k+\e_n}e^{-\left(\frac{(a_k+\e_n)K_n}{2(2+a_k+\e_n)}\right)^2}\int_{s_n}^{S_n}e^{\left(\frac{s}{K_n}-\frac{(a_k+\e_n)K_n}{2(2+a_k+\e_n)}\right)^2}ds\\
&=\frac{CK_n}{2+a_k+\e_n}e ^{-\left(\frac{(a_k+\e_n)K_n}{2(2+a_k+\e_n)}\right)^2}\int_{t_n}^{T_n}e^{t^2}dt
\end{split}
\end{equation}
where we again changed the variable with 
\[
t=\frac{s}{K_n}-\frac{(a_k+\e_n)K_n}{2(2+a_k+\e_n)}
\]
and put
\[
T_n=\frac{(2+a_k+\e_n)\log{\frac{r_n}{\ga_{k,n}}}}{K_n}-\frac{(a_k+\e_n)K_n}{2(2+a_k+\e_n)}
\]
and
\[
t_n=\frac{(2+a_k+\e_n)\log{\frac{\rho_{k,n}}{\ga_{k,n}}}}{K_n}-\frac{(a_k+\e_n)K_n}{2(2+a_k+\e_n)}
\] 
for all $n\in \mathbb{N}$. Notice that $T_n=O(\sqrt{g(\mu_{k,n})})$ and $t_n=O(\sqrt{g(\mu_{k,n})})$ as $n\to \infty$ by \eqref{eq:es2} with \eqref{eq:s0} and (A$_k$). Then we calculate
\[
K_ne ^{-\left(\frac{(a_k+\e_n)K_n}{2(2+a_k+\e_n)}\right)^2}\int_{[t_n,T_n]\cap\{|t|\le (a_k+\e_n)K_n/(4(2+a_k+\e_n))\}}e^{t^2}dt\to0
\]
and 
\[
\begin{split}
&K_ne ^{-\left(\frac{(a_k+\e_n)K_n}{2(2+a_k+\e_n)}\right)^2}\int_{[t_n,T_n]\cap\{|t|>(a_k+\e_n)K_n/(4(2+a_k+\e_n))\}}e^{t^2}dt\\
&\ \ \ \ =O\left(e ^{-\left(\frac{(a_k+\e_n)K_n}{2(2+a_k+\e_n)}\right)^2}\int_{t_n}^{T_n}|t|e^{t^2}dt\right)\\
&\ \ \ \ =O\left(e ^{-\left(\frac{(a_k+\e_n)K_n}{2(2+a_k+\e_n)}\right)^2+T_n^2}\right)+O\left(e ^{-\left(\frac{(a_k+\e_n)K_n}{2(2+a_k+\e_n)}\right)^2+t_n^2}\right)+o(1)\\
&\ \ \ \ \to0
\end{split}
\]
as $n\to \infty$ where noting our choice of $\eta$, we estimated
\[
\begin{split}
&e ^{-\left(\frac{(a_k+\e_n)K_n}{2(2+a_k+\e_n)}\right)^2+T_n^2}\\
&\ \ \ \ \ \ \ =\exp{\left[-\log{\frac{r_n}{\ga_{k,n}}}\left\{a_k+\e_n-\left(\frac{2+a_k+\e_n}{K_n}\right)^2\log{\frac{r_n}{\ga_{k,n}}}\right\}\right]}\to0
\end{split}
\]
from \eqref{eq:es2} with \eqref{eq:s0} and  similarly with  (A$_k$), 
\[
e ^{-\left(\frac{(a_k+\e_n)K_n}{2(2+a_k)}\right)^2+t_n^2}\to0
\]
as $n\to \infty$. Using these formulas for \eqref{eq:s1}, we get the desired conclusion. We complete the proof. 
\end{proof}
We next find the maximal interval with no additional bubble.
\begin{lemma}\label{lem:c10} There exists a sequence $(\sigma_n)\subset (\rho_{k,n},1)$ of values such that 
\[
u_n(\sigma_n)=
\begin{cases}
\displaystyle\left(\frac{\delta_{k+1}}{\delta_k}+o(1)\right) \mu_{k,n}&\text{ if $q>1$},\\
\displaystyle\mu_{k,n}-\left(\log{\frac{\eta_{k}}{\eta_{k+1}}}+o(1)\right) \frac{g(\mu_{k,n})}{g'(\mu_{k,n})}&\text{ if $q=1$}, 
\end{cases}
\]
%$u_n(\sigma_n)/\mu_{k,n}\to \delta_{k+1}/\delta_k$ if $q>1$ and $(\mu_{k,n}-u_n(\sigma_n))g'(\mu_{k,n})/g(\mu_{k,n})\to\log{(\eta_k/\eta_{k+1})})$,
$\phi_n(\sigma_n)\to0$, and  
\[
g'(\mu_{k,n})\int_{\rho_{k,n}}^{\sigma_n}\la_n hf(u_n)rdr\to 0
\]  
as $n\to \infty$.
\end{lemma}
\begin{proof} First, choose any value $\eta_0\in(0,1)$ so that
\[
a_k>\frac{(1-\eta_0)(2+a_k)^2}{4q \eta_0},\ \ \ \eta_0>\frac{\eta_{k+1}}{\eta_k},
\]
and
\[
\begin{cases}
q \eta_0 -p(1-\eta_0^{1/p})>0&\text{ if $q>1$},\\
\eta_0>\displaystyle\log{\frac1{\eta_0}}&\text{ if $q=1$},
\end{cases}
\]
and take any sequence $(\bar{r}_n)\subset (\rho_{k,n},1)$  such that
\[
u_n(\bar{r}_n)=
\begin{cases}
\eta_0^{1/p} \mu_{k,n}&\text{ if $q>1$},\\
\displaystyle\mu_{k,n}-\log{\frac1{\eta_0}} \frac{g(\mu_{k,n})}{g'(\mu_{k,n})}&\text{ if $q=1$}. 
\end{cases}
\]
Then, from the previous lemma, we get that
\begin{equation}\label{eq:c10-1}
\lim_{n\to \infty}g'(\mu_{k,n})\int_{\rho_{k,n}}^{\bar{r}_n}\la_nhf(u_n)rdr=0.
\end{equation}
Next, we select any value $\eta\in(\eta_{k+1}/\eta_k,\eta_0)$ and sequence $(\bar{s}_n)\subset (\bar{r}_n,1)$ so that 
\[
u_n(\bar{s}_n)=
\begin{cases}
\eta^{1/p} \mu_{k,n}&\text{ if $q>1$},\\
\displaystyle\mu_{k,n}-\log{\frac1{\eta}} \frac{g(\mu_{k,n})}{g'(\mu_{k,n})}&\text{ if $q=1$}, 
\end{cases}
\]
for all $n\in \mathbb{N}$. Then for any sequence $(s_n)\subset [\bar{r}_n,\bar{s}_n]$, putting $S_n=s_n/\ga_{k,n}$ for all $n\in \mathbb{N}$, we have by the monotonicity of $g$ and $g'$ and \eqref{as:zk} that
\begin{equation}\label{eq:s2}
\begin{split}
&\phi_n(s_n)\\
&\le \frac{h(s_n)}{h(r_{k,n})}\exp{\left[g\left(\mu_{k,n}+\frac{z_{k,n}(S_n)}{g'(\mu_{k,n})}\right)-g(\mu_{k,n})+2\log{S_n}\right]}\\
&\le %\exp{\left[g(\mu_{k,n})\left\{\frac{g\left(\mu_{k,n}-(2+a_k+o(1))\frac{g(\mu_{k,n})}{g'(\mu_{k,n})}\frac{\log{S_n}}{g(\mu_{k,n})}\right)}{g(\mu_{k,n})}-1+2\frac{\log{S_n}}{g(\mu_{k,n})}\right\}\right]}.
\frac{h(s_n)}{h(r_{k,n})}\exp{\left[g(\mu_{k,n})\left\{\frac{g\left(\mu_{k,n}-(2+a_k+o(1))\frac{\log{S_n}}{g'(\mu_{k,n})}\right)}{g(\mu_{k,n})}-1+2\frac{\log{S_n}}{g(\mu_{k,n})}\right\}\right]}
\end{split}
\end{equation}
for all large $n\in \mathbb{N}$. From \eqref{eq:es1} and \eqref{eq:es2} and  Lemmas \ref{lem:g3} and \ref{lem:g4}, we see
\begin{equation}\label{eq:xy}
x_n:=\frac{(\mu_{k,n}-u_n(\bar{r}_n))\frac{g'(\mu_{k,n})}{g(\mu_{k,n})}}{g'(\mu_{k,n}) \displaystyle\int_0^{\bar{r}_n}\la_n f(u_n)rdr}\le \frac{\log{S_n}}{g(\mu_{k,n})}\le \frac{1-\eta +o(1)}{2}=:y_n
\end{equation}
for all $n\in \mathbb{N}$. Furthermore, from Lemmas \ref{lem:g3} and \ref{lem:g4}, \eqref{g2} if $q>1$, \eqref{as:en} with \eqref{eq:c10-1}, and Lemma \ref{lem:b3}, we compute 
\[
\begin{split}
x_n &=\frac{(\mu_{k,n}-u_n(\bar{r}_n))\frac{g'(\mu_{k,n})}{g(\mu_{k,n})}}{(\tilde{\eta}_k+o(1))g'(\mu_n)\int_0^{\bar{r}_n}\la_n f(u_n)rdr}
%\\&
=
\begin{cases}
\displaystyle\frac{p(1-\eta_0^{1/p})}{2+a_k}+o(1)&\text{ if }q>1\vspace{0.2cm},\\
\displaystyle\frac{\log{(1/\eta_0)}}{2+a_k}+o(1)&\text{ if }q=1,
\end{cases}
\end{split}
\]
as $n\to \infty$. Now,  for each $n\in \mathbb{N}$, set a function
\[
\xi_n(x)=\frac{g\left(\mu_{k,n}-(2+a_k+o(1))\frac{g(\mu_{k,n})}{g'(\mu_{k,n})}x\right)}{g(\mu_{k,n})}-1+2x
\]
for all $x\in[x_n,y_n]$. Then, we calculate  that 
\[
\xi_n'(x)=-(2+a_k+o(1))\frac{g'\left(\mu_{k,n}-(2+a_k+o(1))\frac{g(\mu_{k,n})}{g'(\mu_{k,n})}x\right)}{g'(\mu_{k,n})}+2
\]
and, by \eqref{g2} with \eqref{b1} if $q>1$ and \eqref{g2} if $q=1$, 
\[
\mu_{k,n}-(2+a_k+o(1))\frac{g(\mu_{k,n})}{g'(\mu_{k,n})}x\ge \mu_{k,n}\left(1-(2+a_k+o(1))\frac{g(\mu_{k,n})}{\mu_{k,n}g'(\mu_{k,n})}y_n\right)>t_0
\]
for all $x\in[x_n,y_n]$ and all large $n\in \mathbb{N}$ where $t_0>0$ is the number in our assumption (H1). In particular,  $\xi_n'(x)$ is strictly increasing for all $x\in [x_n,y_n]$ if $n\in \mathbb{N}$ is sufficiently large. It follows that 
\begin{equation}\label{eq:s3}
\max_{x\in[x_n,y_n]}\xi_n(x)= \max\{\xi_n(x_n),\xi_n(y_n)\}
\end{equation}
for all large $n\in \mathbb{N}$. We first compute $\xi_n(x_n)$. If $q>1$, using \eqref{g2} and Lemma \ref{lem:g3}, we have
\[
\begin{split}
\xi\left(x_n\right)=%\frac{g\left(\mu_{k,n}-(4+o(1))\frac{g(\mu_{k,n})}{g'(\mu_{k,n})}\frac{g'(\mu_{k,n})(\mu_n-u_n(\bar{r}_n))}{(4+o(1))g(\mu_{k,n})}\right)}{g(\mu_{k,n})}-1+(2+o(1))\frac{g'(\mu_{k,n})(\mu_{k,n}-u_n(\bar{r}_n))}{(4+o(1))g(\mu_{k,n})}.
\frac{2p}{2+a_k}(1-\eta_0^{1/p})-1+\eta_0+o(1)
\end{split}
\]
as $n\to \infty$. If $q=1$,  noting Lemma \ref{lem:g4}, we obtain
\[
\xi_n(x_n)%=-\left\{1-\exp{\left[-(2+a_k)\frac{\log{(1/\eta_0)}}{\eta_0\sum_{i=1}^k (2a_i\eta_i^{-1})}\right]}-2\frac{\log{(1/\eta_0)}}{\eta_0\sum_{i=1}^k (2a_i\eta_i^{-1})}\right\}+o(1)
=\frac{2}{2+a_k}\log{\frac1{\eta_0}}-1+\eta_0+o(1)
\]
as $n\to \infty$. Hence,  since $\eta_0\in(\eta_{k+1}/\eta_k,1)$, noting \eqref{b1} and \eqref{b3}, we get $\lim_{n\to \infty}\xi_n(x_n)<0$. We next calculate $\xi_n(y_n)$. Using  \eqref{g2} and Lemmas \ref{lem:g3} and \ref{lem:g4}, we deduce  that
\[
\begin{split}
\xi_n\left(y_n\right)&
%&= g\left(\mu_{k,n}-\frac{4+o(1)}{g'(\mu_{k,n})}\frac{g(\mu_{k,n})}{2}\left(1-\eta +o(1)\right)\right)-g(\mu_{k,n})+(2+o(1))\frac{g(\mu_{k,n})}{2}\left(1-\eta +o(1)\right).\\
=\frac{g\left(\mu_{k,n}-\frac{(2+a_k)(1-\eta)+o(1)}{2}\frac{g(\mu_{k,n})}{g'(\mu_{k,n})}\right)}{g(\mu_{k,n})}-\eta+o(1) \\
&=
\begin{cases}-\left\{\eta-\left(1-\frac{2+a_k}{2p}(1-\eta)\right)^p\right\}+o(1)&\text{ if }q>1,\vspace{0.2cm}\\
-\left(\eta-e^{-\frac{2+a_k}{2}(1-\eta)}\right)+o(1)&\text{ if }q=1
\end{cases}
\end{split}
\]
as $n\to \infty$. Therefore, again as  $\eta\in (\eta_{k+1}/\eta_k,1)$, using \eqref{b1}, and \eqref{b3}, we conclude $\lim_{n\to \infty}\xi_n(y_n)<0$.  Consequently, recalling \eqref{eq:s3}, we find a number $\e_0>0$ such that 
\[
\sup_{x\in[x_n,y_n]}\xi_n(x)\le -\e_0
\]
 for all large $n\in \mathbb{N}$. Using this  for \eqref{eq:s2}, we obtain that 
\[
\sup_{s\in[\bar{r}_n,\bar{s}_n]}\phi_n(s)\le e^{-(\e_0+o(1)) g(\mu_{k,n})}
\] 
for all large $n\in \mathbb{N}$. This and Lemma \ref{lem:monog} prove that
\[
\begin{split}
g'(\mu_{k,n})\int_{\bar{r}_n}^{\bar{s}_n}\la_n hf(u_n)rdr&= g'(\mu_{k,n})\int_{\bar{r}_n}^{\bar{s}_n}\frac{\phi_n(r)}{rg'(u_n(r))}dr\\
&\le e^{-(\e_0+o(1)) g(\mu_{k,n})}\left(\log{\frac{\bar{s}_n}{\ga_{k,n}}}-\log{\frac{\bar{r}_n}{\ga_{k,n}}}\right)\\
&\to0
\end{split}
\]
as $n\to \infty$ by \eqref{eq:xy}. Finally, notice that we can choose the value $\eta$ and the corresponding sequence $(\bar{s}_n)$ so that $\eta$ is arbitrarily close to $\eta_{k+1}/\eta_k$. Noting this fact together with \eqref{eq:c10-1},  we get the desired conclusion. We finish the proof. 
\end{proof}
The next lemma ensures that a certain energy appears on the outside of the interval $(0,\sigma_n)$.
\begin{lemma}\label{lem:c11} Assume that $(\sigma_n)$ is the sequence obtained in the previous lemma. Let  $\eta\in(0,\eta_{k+1}/\eta_k)$ and $(t_n)\subset (\sigma_n,1)$ be any value and  sequence of numbers such that  
\[
u_n(t_n)=
\begin{cases}
\displaystyle\eta^{1/p} \mu_{k,n}&\text{ if $q>1$},\\
\displaystyle\mu_{k,n}-\log{\frac{1}{\eta}} \frac{g(\mu_{k,n})}{g'(\mu_{k,n})}&\text{ if $q=1$}, 
\end{cases}
\]
for all $n\in \mathbb{N}$. Then we have that 
\[
\liminf_{n\to \infty}g'(\mu_n)\int_0^{t_n}\la_n hf(u_n)rdr\ge
\begin{cases} \displaystyle\frac{2p}{\delta_k^{p-1}}\cdot\frac{\frac{\delta_{k+1}}{\delta_k}-\eta^{1/p}}{\left(\frac{\delta_{k+1}}{\delta_k}\right)^p-\eta}&\text{ if }q>1\vspace{0.3cm},\\
\displaystyle\frac2{\eta_k}\cdot\frac{\log{\frac1{\eta}}-\log{\frac{\eta_{k}}{\eta_{k+1}}}}{\frac{\eta_{k+1}}{\eta_{k}}-\eta}&\text{ if }q=1.
\end{cases}
\]
Moreover, there exists a sequence $(\tilde{t}_n)\subset (\sigma_n,1)$ such that 
\[
u_n(\tilde{t}_n)=
\begin{cases}
\displaystyle \left(\frac{\delta_{k+1}}{\delta_{k}}+o(1)\right) \mu_{k,n}&\text{ if $q>1$},\\
\displaystyle\mu_{k,n}-\left( \log{\frac{\eta_k}{\eta_{k+1}}}+o(1)\right) \frac{g(\mu_{k,n})}{g'(\mu_{k,n})}&\text{ if $q=1$}, 
\end{cases}
\]
 and 
\[
\liminf_{n\to \infty}g'(\mu_n)\int_0^{\tilde{t}_n}\la_n hf(u_n)rdr
> \displaystyle\sum_{i=1}^{k} \frac{2a_i}{\delta_i^{p-1}}.
%> 
%\begin{cases}
%\displaystyle\sum_{i=1}^k \frac{2a_k}{\delta_i^{p-1}}&\text{ if }q>1\\
%\displaystyle \sum_{i=1}^k \frac{2a_i}{\tilde{\eta}_i}. %&\text{ if }q=1.
%\end{cases}
\]
\end{lemma}
%\begin{remark} The former estimate for $q=1$ with \eqref{b3} implies that the left-hand side diverges to infinity as $\eta\to0^+$.  
%\end{remark}
\begin{proof} From \eqref{id2}, \eqref{eq:es1}, and \eqref{eq:es2}, we get
\[
\begin{split}
u_n&(\sigma_n)-u_n(t_n)\\
&\le \left(\log{\frac{t_n}{\ga_{k,n}}} -\log{\frac{\sigma_n}{\ga_{k,n}}}\right)\int_0^{t_n}\la_nh f(u_n)rdr\\
&\le \frac{g(\mu_{k,n})}2\left(1-\eta-\frac{(2+o(1))(\mu_{k,n}-u_n(\sigma_n))}{g(\mu_{k,n})\int_0^{\sigma_n}\la_n f(u_n)rdr}+o(1)\right)\int_0^{t_n}\la_n hf(u_n)rdr
\end{split}
\] 
by Lemmas \ref{lem:g3} and \ref{lem:g4}. Hence, noting \eqref{g2}, Lemmas \ref{lem:g3} and \ref{lem:g4}, and \eqref{as:en} with Lemma \ref{lem:c10}, we compute
\[
\begin{split}
g'(\mu_n)\int_0^{t_n}&\la_n h f(u_n)rdr\\
&\ge \frac{g'(\mu_n)}{g'(\mu_{k,n})}\frac{2g'(\mu_{k,n})(u_n(\sigma_n)-u_n(t_n))}{g(\mu_{k,n})\left(1-\eta-\frac{(2+o(1))g'(\mu_{k,n})(\mu_{k,n}-u_n(\sigma_n))}{g(\mu_{k,n})g'(\mu_{k,n})\int_0^{\sigma_n}\la_n f(u_n)rdr}+o(1)\right)}\\
&=
\begin{cases}
\displaystyle\frac{\frac{2p}{\delta_k^{p-1}}\left(\frac{\delta_{k+1}}{\delta_k}-\eta^{1/p}\right)}{1-\eta-2p\left(1-\frac{\delta_{k+1}}{\delta_k}\right)/\left(\delta_k^{p-1}\sum_{i=1}^k \frac{2a_i}{\delta_i^{p-1}}\right)}+o(1)&\text{ if }q>1,\vspace{0.3cm}\\
\displaystyle\frac{\frac{2}{\eta_k}\log{\frac{\eta_{k+1}}{\eta_{k}\eta}}}{1-\eta-\left(2\log{\frac{\eta_{k}}{\eta_{k+1}}}\right)/\left(\eta_k \sum_{i=1}^k \frac{2a_i}{ \eta_i}\right)}+o(1)&\text{ if }q=1,
\end{cases}
\end{split}
\] 
as $n\to \infty$. Then Lemma \ref{lem:b3}, \eqref{b1}, and \eqref{b3} confirm the former conclusion. Let us ensure the latter one. For the case $q>1$, we calculate that 
\[
\lim_{\eta\to ((\delta_{k+1}/\delta_k)^{p})^-}\displaystyle\frac{2p}{\delta_k^{p-1}}\cdot\frac{\frac{\delta_{k+1}}{\delta_k}-\eta^{1/p}}{\left(\frac{\delta_{k+1}}{\delta_k}\right)^p-\eta}=\frac{2}{\delta_{k+1}^{p-1}}= \frac{a_{k+1}}{\delta_{k+1}^{p-1}}+\displaystyle\sum_{i=1}^{k} \frac{2a_i}{\delta_i^{p-1}}
\]
from \eqref{b2} with Lemma \ref{lem:b3}. For the case $q=1$, we similarly compute that 
\[
\lim_{\eta\to (\eta_{k+1}/\eta_k)^-}\displaystyle\frac2{\eta_k}\cdot\frac{\log{\frac1{\eta}}-\log{\frac{\eta_{k}}{\eta_{k+1}}}}{\frac{\eta_{k+1}}{\eta_{k}}-\eta}=\frac{2}{\eta_{k+1}}=\frac{a_{k+1}}{\eta_{k+1}}+\displaystyle \sum_{i=1}^k \frac{2a_i}{\eta_i}
\]  
from \eqref{b4} with Lemma \ref{lem:b3}. Therefore, since we can choose the value $\eta$ and the corresponding sequence $(t_n)$ so that $\eta$ is arbitrarily closed to $\eta_{k+1}/\eta_k$ in the former argument, we get the latter assertion. We finish the proof.  
\end{proof}
Now, we are arriving in the next concentrating region.  
\begin{lemma}\label{lem:c12} Let $(\sigma_n),(\tilde{t}_n)$ be the sequences in the previous lemma. Then there exists a sequence $(\tau_n)\subset (\sigma_n,\tilde{t}_n]$ of values such that $\phi_n(\tau_n)\to c$ and $\sup_{\sigma_n\le r\le  \tau_n}\psi_n(r)\to d$ for some values $c>0$ and $0<d<2$ up to a subsequence. Moreover, we have $\sigma_n/\tau_n\to0$ as $n\to \infty$. %Then there exists a value $\eta\in (0,1)$ and a sequence  $(t_n)\subset (s_n,1)$ such that if $q>1$, $\eta\in(0,\delta_2^p]$ and $u_n(t_n)/\mu_{1,n}\to \eta^{1/p}$ and if $q=1$, $u_n(t_n)=\mu_{1,n}-(\log{1/\eta}+o(1))g(\mu_{1.n})/g'(\mu_{1,n})$ as $n\to \infty$ and $\phi_n(t_n)\to c$ for some value $c>0$.
\end{lemma}
\begin{proof}  We first notice that Lemma \ref{lem:b3}, \eqref{b2}, and \eqref{b4}  imply that
\[
%\begin{cases}
%\displaystyle\delta_{k+1}^{p-1}\sum_{i=1}^k \frac{2a_i}{\delta_i^{p-1}}=2-a_{k+1}<2 &\text{ if }q>1,\vspace{0.2cm}\\
\displaystyle\tilde{\eta}_{k+1}\sum_{i=1}^k\frac{2a_i}{\tilde{\eta}_i}=2-a_{k+1}<2.%&\text{ if }q=1.
%\end{cases}
\]
Then using (A$_k$) and Lemmas \ref{lem:c10} and \ref{lem:c11}, we choose a sequence $(\tau_n)\subset (\sigma_n,\tilde{t}_n]$ and a value $\e>0$ so that 
\begin{equation}\label{eq:ppo}
\lim_{n\to \infty}g'(\mu_{n})\int_0^{\tau_n}\la_nhf(u_n)rdr=\displaystyle \sum_{i=1}^k \frac{2a_i}{\tilde{\eta}_i}+\e
\end{equation}
and
\[
%\begin{cases}
%\displaystyle\delta_{k+1}^{p-1}\left(\sum_{i=1}^k \frac{2a_i}{\delta_i^{p-1}}+\e\right)<2 &\text{ if }q>1,\vspace{0.2cm}\\
\displaystyle\tilde{\eta}_{k+1}\left(\sum_{i=1}^k\frac{2a_i}{\tilde{\eta}_i}+\e\right)<2.%&\text{ if }q=1. 
%\end{cases}
\]
It follows from Lemmas \ref{lem:monog}, \ref{lem:g3}, and \ref{lem:g4} that
\begin{equation}\label{eq:infps}
\begin{split}
\inf_{r\in[\sigma_n,\tau_n]}\psi_n(r)&\ge \frac{g'(u_n(\tilde{t}_n))}{g'(\mu_{n})}g'(\mu_{n})\int_0^{\sigma_n}\la_n hf(u_n)rdr\\
&=%\begin{cases}\displaystyle\delta_{k+1}^{p-1}\sum_{i=1}^k \frac{2a_i}{\delta_i^{p-1}} +o(1)&\text{ if }q>1,\\
\displaystyle\tilde{\eta}_{k+1}\sum_{i=1}^k\frac{2a_i}{\tilde{\eta}_i} +o(1)%&\text{ if }q=1,
%\end{cases}
\end{split}
\end{equation}
and similarly from our choice of $(\tau_n)$ that
\begin{equation}\label{eq:supps}
\begin{split}
\sup_{r\in[\sigma_n,\tau_n]}\psi_n(r)&\le %\frac{g'(u_n(\sigma_n))}{g'(\mu_{n})}g(\mu_{n})\int_0^{\tau_n}\la_n hf(u_n)rdr
%&=\begin{cases}\displaystyle\delta_{k+1}^{p-1}\left(\sum_{i=1}^k \frac{2a_i}{\delta_i^{p-1}}+\e\right)+o(1)\text{ if }q>1,\\
%=
\displaystyle\tilde{\eta}_{k+1}\left(\sum_{i=1}^k\frac{2a_i}{\tilde{\eta}_i}+\e\right)+o(1)%\text{ if }q=1
%\end{cases}
\end{split}
\end{equation}
as $n\to \infty$. Hence after selecting a subsequence if necessary, we find a value $d\in(0,2)$ such that $\sup_{r\in[\sigma_n,\tau_n]}\psi_n(r)\to d$ as $n\to \infty$. Next using Lemma \ref{lem:monog}, we have that 
\[
\frac12\phi_n(\tau_n)\le \psi_n(\tau_n)\le 2
\]
for all large $n\in \mathbb{N}$. Hence there exists a value $c\ge 0$ such that $\phi_n(\tau_n)\to c$ as $n\to \infty$ up to a subsequence. We claim $c\not=0$. We prove this by contradiction. Assume $c=0$. Lemma \ref{lem:D1} with the fact that $d<2$ yields that $\phi_n'(r)>0$ for all $r\in[\sigma_n,\tau_n]$ and all large $n\in \mathbb{N}$. This implies that  $\sup_{r\in[\sigma_n,\tau_n]}\phi_n(r)\to0$ as $n\to \infty$. Then Lemma \ref{lem:D2} shows that  
\[
g'(\mu_{n})\int_{\sigma_n}^{\tau_n}\la_n hf(u_n)rdr\to0
\]  
as $n\to \infty$. This contradicts \eqref{eq:ppo}. Finally, from Lemmas \ref{lem:c10} and \ref{lem:monog}, we have that
\[
o(1)=\frac{\phi_n(\sigma_n)}{\phi_n(\tau_n)}\ge \frac{h(\sigma_n)}{h(\tau_n)}\left(\frac{\sigma_n}{\tau_n}\right)^2
\]
for all large $n\in \mathbb{N}$. This proves the final assertion.  We finish the proof.
\end{proof}
\begin{remark}\label{rmk:k1} Since $(\tau_n)\subset (\sigma_n,\tilde{t}_n]$, we have that 
\[
u_n(\tau_n)=
\begin{cases}
\displaystyle\left(\frac{\delta_{k+1}}{\delta_k}+o(1)\right) \mu_{k,n}&\text{ if $q>1$},\\
\displaystyle\mu_{k,n}-\left(\log{\frac{\eta_{k}}{\eta_{k+1}}}+o(1)\right) \frac{g(\mu_{k,n})}{g'(\mu_{k,n})}&\text{ if $q=1$}, 
\end{cases}
\]
as $ n\to \infty$. 
\end{remark}
We then detect the $(k+1)$-th concentrating part. 
\begin{lemma}\label{lem:c13} There exists a sequence $(r_{k+1,n})\subset (\sigma_n,1)$ of values  and a value $a\in(0,2)$ such that $\phi_n(r_{k+1,n})$ attains a local maximum value of $\phi_n$, especially  $\phi_n'(r_{k+1,n})=0$, for every $n\in \mathbb{N}$ and $\phi_n(r_{k+1,n})\to a^2/2$, $\psi_n(r_{k+1,n})\to2$,  and
\[
u_n(r_{k+1,n})=
\begin{cases}
\displaystyle\left(\frac{\delta_{k+1}}{\delta_k}+o(1)\right) \mu_{k,n}&\text{ if $q>1$},\\
\displaystyle\mu_{k,n}-\left(\log{\frac{\eta_{k}}{\eta_{k+1}}}+o(1)\right) \frac{g(\mu_{k,n})}{g'(\mu_{k,n})}&\text{ if $q=1$}, 
\end{cases}
\]
%$u_n(r_{k+1,n})/\mu_{k+1,n}\to \delta_{k+1}$ if $q>1$ and $(\mu_{k,n}-u_n(r_{k+1,n}))g'(\mu_{k,n})/g(\mu_{k,n})\to \log{\eta_k/\eta_{k+1}}$
 as $n\to \infty$, and if we put $\mu_{k+1,n}=u_n(r_{k+1,n})$ for all $n\in \mathbb{N}$ and set sequences $(\ga_{k+1,n})$ of positive values and $(z_{k+1,n})$ of functions so that for every large $n\in \mathbb{N}$,
\[
\la_nh(r_{k+1,n})f'(\mu_{k+1,n})\ga_{k+1,n}^2=1
\]  
and 
\[
z_{k+1,n}(r)=g'(\mu_{k+1,n})(u_n(\ga_{k+1,n}r)-\mu_{k+1,n})
\]
for all $r\in[0,1/\ga_{k+1,n}]$, then we have that $\ga_{k+1,n}\to 0$ and  $\|z_{k+1,n}-z_{k+1}\|_{C^2_{\text{loc}}((0,\infty))}\to 0$ as $n\to \infty$ where $z_{k+1}$ is a function in $(0,\infty)$ such that 
\[
z_{k+1}(r)=\log{\frac{2a^2b}{r^{2-a}(1+b r^{a})^2}}
\]
for all $r>0$ with $b=(\sqrt{2}/a)^a$ which satisfies
\[
\begin{cases}
\displaystyle-z_{k+1}''-\frac1r z_{k+1}'=e^{z_{k+1}} \text{ in }(0,\infty),\\
z_{k+1}(a/\sqrt{2})=0,\ (a/\sqrt{2})z_{k+1}'(a/\sqrt{2})=-2,
\end{cases}
\] 
 and 
\[
\int_0^{\infty}e^{z_{k+1}}rdr=2a.
\]
Moreover, we have that
\[
\lim_{(\e,R)\to (0,\infty)}\lim_{n\to \infty}g'(\mu_{k+1,n})\int_{\e\ga_{k+1,n}}^{R\ga_{k+1,n}}\la_n h f(u_n)rdr=2 a
\]
and 
\[
\lim_{(\e,R)\to (0,\infty)}\lim_{n\to \infty}\int_{\e\ga_{k+1,n}}^{R\ga_{k+1,n}}\la_n f'(u_n)rdr=2 a.
\]
Finally, we obtain that $\phi_n(\ga_{k+1,n}r)/( r^2 e^{z_{k+1}(r)})\to1$ as $n\to \infty$ uniformly for all $r$ in every compact subset of $(0,\infty)$. 
\end{lemma}
\begin{remark} We will see later in  Lemma \ref{lem:c15} below that $a=a_{k+1}$ and thus, $z_{k+1}$ in the present lemma  coincides with the one defined by \eqref{def:zk}.
\end{remark}
\begin{proof} Let $(\tau_n)$, $c>0$, and $d\in(0,2)$ be the sequence and values obtained in the previous lemma. For all $n\in \mathbb{N}$, set $m_n=u_n(\tau_n)$ and
\[
w_n(r)=g'(m_n)(u_n(\tau_n r)-m_n)
\] 
for all $r\in [0,1/\tau_n]$. Then it satisfies 
\begin{equation}\label{wn}
\begin{cases}
\displaystyle-w_n''-\frac1r w_n'=\phi_n(\tau_n)\frac{h(\tau_n r)}{h(\tau_n)}\frac{f(u_n(\tau_n r))}{f(m_n)} \text{ in }(0,1/\tau_n),\\
w_n(1)=0,\ w_n'(1)=-\psi_n(\tau_n),
\end{cases}
\end{equation}
for all $n\in \mathbb{N}$. Notice that $1/\tau_n\to \infty$ as $n\to \infty$ by Lemma \ref{lem:gl}. We shall give some local uniform estimates for $(w_n)$. To this end, we first choose any $r\in(0,1]$. Then using \eqref{id0}, we have that
\begin{equation}\label{es:wn1}
-r w_n'(r)=g'(m_n)\int_0^{\tau_n r}\la_nh f(u_n)rdr\le d+o(1)
\end{equation}
as $n\to \infty$. Moreover, from Lemma \ref{lem:c12}, we get that $\rho_{k,n}/\tau_n\le \sigma_n/\tau_n=o(1)$ as $n\to \infty$. It follows from Lemmas \ref{lem:g3} and \ref{lem:g4} and \eqref{as:en} that  
\[
-r w_n'(r) = \frac{g'(m_n)}{g'(\mu_{n})} g'(\mu_{n})\int_0^{\tau_n r}\la_n h f(u_n)rdr
\ge
%\begin{cases}
%\delta_{k+1}^{p-1} \sum_{i=1}^n \frac{2a_i}{\delta_i^{p-1}}+o(1)&\text{ if }q>1\\
\tilde{\eta}_{k+1} \sum_{i=1}^k \frac{2a_i}{\tilde{\eta}_i}+o(1)%&\text{ if }q=1
%\end{cases}
\]
as $n\to \infty$. Integrating the above two inequalities over $[r,1]$, we get
\begin{equation}\label{es:wn2}
\left(\tilde{\eta}_{k+1} \sum_{i=1}^n \frac{2a_i}{\tilde{\eta}_i}+o(1)\right)\log{\frac1r}\le  w_n(r) \le (d+o(1))\log{\frac{1}{r}}
%\begin{cases}
%\left(\delta_{k+1}^{p-1} \sum_{i=1}^n \frac{2a_i}{\delta_i^{p-1}}+o(1)\right)\frac1r&\text{ if }q>1,\\
%&\text{ if }q=1,
%\end{cases}
\end{equation}
as $n\to \infty$ where $o(1)\to0$ as $n\to \infty$ uniformly for each compact subset of $(0,1]$.  Next we set $r>1$. Then, integrating the equation in \eqref{wn} over $[1,r]$ and noting Lemma \ref{lem:monog}, we get a constant $C>0$ which is independent of the choice of $r>1$ such that 
\begin{equation}\label{es:wn3}
-w_n'(r)\le Cr
%\end{equation}
%for all large $n\in \mathbb{N}$. Integrating again gives that 
%\begin{equation}\label{es:wn4}
\text{ and }-w_n(r)\le \frac{C}{2}r^2%\frac{c+o(1)}{2}\left(\frac{r^2-1}{2}-\log{r}\right)+(d+o(1))\log{r}
\end{equation}
by integration again for all large $n\in \mathbb{N}$. Consequently, \eqref{es:wn1}, \eqref{es:wn2}, and \eqref{es:wn3} prove that $(w_n)$ is bounded in $C^1_{\text{loc}}((0,\infty))$. Then the Ascoli-Arzel\`{a} theorem, the equation in \eqref{wn}, and Lemma \ref{lem:g1} ensure that there exists a smooth function $w$ such that $w_n\to w$ in $C^2_{\text{loc}}((0,\infty))$ and  
\[
\begin{cases}
\displaystyle -w''-\frac1r w'=c e^w \text{ in }(0,\infty),\\
w(1)=0,\ w'(1)=-d.
\end{cases}
\]
Hence setting functions $z(r)=w(r/\sqrt{c})$ for all $r\in(0,\infty)$ and $z_n(r)=w_n(r/\sqrt{\phi_n(\tau_n)})$ for all $r\in [0,\sqrt{\phi_n(\tau_n)}/\tau_n]$ and all large $n\in \mathbb{N}$, we get $z_n\to z$ in $C^2_{\text{loc}}((0,\infty))$ as $n\to \infty$ and $z$ satisfies 
\begin{equation}\label{z}
\begin{cases}
\displaystyle -z''-\frac1r z'= e^z \text{ in }(0,\infty),\\
z(\sqrt{c})=0,\ \sqrt{c}z'(\sqrt{c})=-d.
\end{cases}
\end{equation}
Integrating the equation in \eqref{z} gives that 
\[
z(r)=\log{\frac{2a^2\tilde{b}}{r^{2-a}(1+\tilde{b} r^{a})^2}}
\]
for all $r>0$ with some constants $a,\tilde{b}>0$. (See (17) in the proof of Proposition 3.1 in \cite{GGP}.) Notice that the first inequality in \eqref{es:wn2} yields that $a<2$. 

Now, as in the argument in Subsection \ref{sub:cfc}, we look for the center of the concentration and  modify  the scaled function appropriately.  To this end, take any value $R>0$ and set $r_n=R \tau_n/\sqrt{\phi_n(\tau_n)}$ for all $n\in \mathbb{N}$. It follows from Lemma  \ref{lem:g1} that
\[ 
\phi_n(r_n)=R^2 \frac{h(r_n)}{h(\tau_n)}\frac{f'(u_n(r_n))}{f'(u_n(\tau_n))}=(1+o(1))\frac{2a^2\tilde{b}R^a}{(1+\tilde{b} R^{a})^2}
\]
as $n\to \infty$. Notice that the function $R\mapsto 2a^2\tilde{b}R^{a}/(1+\tilde{b}R^{a})^2$ on $(0,\infty)$  has a unique maximum point $R^*=(1/\tilde{b})^{1/a}$ with the maximum value $a^2/2$. Hence we find a sequence $(r_{k+1,n})\subset (\sigma_n,1)$ of values such that $\phi_n(r_{k+1,n})$ attains a local maximum value  of $\phi_n$, in particular $\phi_n'(r_{k+1,n})=0$, for every large $n\in \mathbb{N}$ and $\phi_n(r_{k+1,n})\to a^2/2$ and $r_{k+1,n}/\tau_n\to R^*/\sqrt{c}$ as $n\to \infty$. Furthermore, we see 
\begin{equation}\label{eq:y1}
u_n(r_{k+1,n})=m_n+\frac{z(R^*)+o(1)}{g'(m_n)}
\end{equation}
as $n\to \infty$. Then since $\sigma_n/\tau_n=o(1)$, we have that
\[
\begin{split}
\psi_n(r_{k+1,n})&\ge \frac{g'(u_n(r_{k+1,n}))}{g'(u_n(\tau_n))}\frac{g'(u_n(\tau_n))}{g'(\mu_{k,n})}\frac{g'(\mu_{k,n})}{g'(\mu_n)}g'(\mu_n)\int_0^{\sigma_n}\la_n hf(u_n)rdr\\
&=\tilde{\eta}_{k+1}\sum_{i=1}^k \frac{2a_i}{\tilde{\eta}_i}+o(1)
\end{split}
\]
as $n\to \infty$ by Lemma \ref{lem:g2} with \eqref{eq:y1}, Lemmas \ref{lem:g3} and \ref{lem:g4}, and  \eqref{as:en} with Lemma \ref{lem:c10}.  Consequently, Lemma \ref{lem:D1} ensures that $\psi_n(r_{k+1,n})\to2$ as $n\to \infty$.

 Then we define the modified scaled function. Putting $\mu_{k+1,n}=u_n(r_{k+1,n})$ for all $n\in \mathbb{N}$, we set sequences $(\ga_{k+1,n})$ of values and $(z_{k+1,n})$ of functions so that for each large $n\in \mathbb{N}$,
\[
\la_n h(r_{k+1,n})f'(\mu_{k+1,n})\ga_{k+1,n}^2=1
\]
and 
\[
z_{k+1,n}(r)=g'(\mu_{k+1,n})(u_n(\ga_{k+1,n}r)-\mu_{k+1,n})
\]
for all $r\in[0,1/\ga_{k+1,n}]$. Then since $(r_{k+1,n}/\ga_{k+1,n})^2=a^2/2+o(1)$, we have $\ga_{k+1,n}\to0$ as $n\to \infty$ by Lemma \ref{lem:gl}. Moreover, note that 
\begin{equation}\label{zkn}
\begin{cases}
\displaystyle-z_{k+1,n}''-\frac1r z_{k+1,n}'=\frac{h(\ga_{k+1,n}\cdot)}{h(r_{k+1,n})}\frac{f(u_n(\ga_{k+1,n} \cdot))}{f(\mu_{k+1,n})} \text{ in }(0,1/\ga_{k+1,n}),\\
z_{k+1,n}\left(\frac{r_{k+1,n}}{\ga_{k+1,n}}\right)=0,\ \left(\frac{r_{k+1,n}}{\ga_{k+1,n}}\right)z_{k+1,n}'\left(\frac{r_{k+1,n}}{\ga_{k+1,n}}\right)=-2+o(1),
\end{cases}
\end{equation}
and
\[
z_{k+1,n}(r)=\frac{g'(\mu_{k+1,n})}{g'(m_n)}\left\{z_n\left(\frac{\sqrt{\phi_n(\tau_n)}}{\tau_n}\ga_{k+1,n}r\right)-z_n\left(\frac{\sqrt{\phi_n(\tau_n)}}{\tau_n}r_{k+1,n}\right)\right\}
\]
%and
%\[
%\mu_{k+1,n}=u_n\left(\frac{\tau_n}{\sqrt{\phi_n(\tau_n)}}(R^*+o(1))\right)=m_n+\frac{z(R^*)+o(1)}{g'(m_n)}
%\]
for all $n\in \mathbb{N}$. Then, using \eqref{eq:y1} with Lemma \ref{lem:g2} and putting
\[
z_{k+1}(r):=z\left(\frac{\sqrt{2}}{a}\left(\frac1{\tilde{b}}\right)^{1/a}r\right)-z\left(\left(\frac1{\tilde{b}}\right)^{1/a}\right)
\]
for all $r\in(0,\infty)$, we derive that $z_{k+1,n}\to z_{k+1}$ in $C^2_{\text{loc}}((0,\infty))$ as $n\to \infty$,  
\[
z_{k+1}(r)=\log{\frac{2a^2b}{r^{2-a}(1+b r^{a})^2}}
\]
 with $b=(\sqrt{2}/a)^a$, and by Lemma \ref{lem:g1}, 
\begin{equation}\label{zk1}
\begin{cases}
\displaystyle-z_{k+1}''-\frac1r z_{k+1}'=e^{z_{k+1}} \text{ in }(0,\infty),\\
z_{k+1}(a/\sqrt{2})=0,\ \left(a/\sqrt{2}\right)z_{k+1}'(a/\sqrt{2})=-2.
\end{cases}
\end{equation}
Now, noting \eqref{eq:y1} with Remark \ref{rmk:k1}, we get $\mu_{k+1,n}/\mu_{k,n}\to \delta_{k+1,n}/\delta_k$ if $q>1$. If $q=1$, from \eqref{eq:y1} with Remark \ref{rmk:k1} and Lemma \ref{lem:g4}, we have
\[
\begin{split}
\mu_{k+1,n}&=\mu_{k,n}-\left(\log{\frac{\eta_k}{\eta_{k+1}}}-\frac{z(R^*)+o(1)}{g(\mu_{k,n})}\frac{g'(\mu_{k,n})}{g'(m_n)}+o(1)\right)\frac{g(\mu_{k,n})}{g'(\mu_{k,n})}\\
&=\mu_{k,n}-\left(\log{\frac{\eta_k}{\eta_{k+1}}}+o(1)\right)\frac{g(\mu_{k,n})}{g'(\mu_{k,n})}
\end{split}
\]
as $n\to \infty$. 
 
Let us complete the final assertions.  Using Lemma \ref{lem:g1} again, for any values $0<\e<R$, we get
\[
\begin{split}
\lim_{n\to \infty}g'(\mu_{k+1,n})\int_{\e \ga_{k+1,n}}^{R \ga_{k+1,n}}\la_nh f(u_n)rdr&=\lim_{n\to \infty}\int_{\e }^R \frac{h(\ga_{k+1,n}r)}{h(r_{k+1,n})}\frac{f(u_n(\ga_{k+1,n}r))}{f(\mu_{k+1,n})}rdr\\
&=\int_\e^Re^{z_{k+1}}rdr
\end{split}
\]
and analogously,
\[
\begin{split}
\lim_{n\to \infty}\int_{\e \ga_{k+1,n}}^{R \ga_{k+1,n}}\la_nh f'(u_n)rdr=\int_\e^Re^{z_{k+1}}rdr.
\end{split}
\] 
Finally, choose any compact subset $K\subset (0,\infty)$. Then, we deduce form Lemma \ref{lem:g1} that
\[
\begin{split}
\sup_{r\in K}\left|\frac{\phi_n(\ga_{k+1,n}r)}{r^2e^{z_{k+1}(r)}}-1\right|=\sup_{r\in K}\left|\frac{\frac{h(\ga_{k+1,n}r)}{h(r_{k+1,n})}\frac{f(u_n(\ga_{k+1,n}r))}{f(\mu_{k+1,n})}}{e^{z_{k+1}(r)}}-1\right|\to 0
\end{split}
\]
as $n\to \infty$.  We finish the proof. 
\end{proof}
We slightly extend the present concentration region and connect it  with the previous one.
\begin{lemma}\label{lem:c14} There exist sequences $(\bar{\rho}_{k+1,n}),(\rho_{k+1,n})\subset (0,1)$ of values such that  $\bar{\rho}_{k+1,n}/\ga_{k+1,n}\to0$,  $\rho_{k+1,n}/\ga_{k+1,n}\to \infty$, $g(\mu_{k+1,n})^{-1}\log{(\bar{\rho}_{k+1,n}/\ga_{k+1,n})}\to0$,
 $g(\mu_{k+1,n})^{-1}\log{(\rho_{k+1,n}/\ga_{k+1,n})}\to0$,
\[
\|z_{k+1,n}-z_{k+1}\|_{C^2([\bar{\rho}_{k+1,n}/\ga_{k+1,n},\rho_{k+1,n}/\ga_{k+1,n}])}\to0,
\]
\[
\sup_{r\in[\bar{\rho}_{k+1,n}/\ga_{k+1,n},\rho_{k+1,n}/\ga_{k+1,n}]}|rz_{k+1,n}'(r)-rz_{k+1}'(r)|\to0,
\]
\[
\sup_{r\in [\bar{\rho}_{k+1,n}/\ga_{k+1,n}, \rho_{k+1,n}/\ga_{k+1,n}]}\left|\frac{h(\ga_{k+1,n}r)}{h(r_{k+1,n})}\frac{f(u_n(\ga_{k+1,n}r))}{f(\mu_{k+1,n})}-e^{z_{k+1}(r)}\right|\to0
\]
\[
\sup_{r\in [\bar{\rho}_{k+1,n}/\ga_{k+1,n}, \rho_{k+1,n}/\ga_{k+1,n}]}\left|\frac{h(\ga_{k+1,n}r)}{h(r_{k+1,n})}\frac{f'(u_n(\ga_{k+1,n}r))}{f'(\mu_{k+1,n})}-e^{z_{k+1}(r)}\right|\to0
\]
\[
\sup_{r\in [\bar{\rho}_{k+1,n}/\ga_{k+1,n}, \rho_{k+1,n}/\ga_{k+1,n}]}\left|\frac{\phi_n(\ga_{k+1,n}r)}{r^2e^{z_{k+1}(r)}}-1\right|\to0,
\]
\[
g'(\mu_{k+1,n})\int_{\bar{\rho}_{k+1,n}}^{\rho_{k+1,n}}\la_n hf(u_n)rdr\to2 a,
\]
and
\[
\int_{\bar{\rho}_{k+1,n}}^{\rho_{k+1,n}}\la_n hf'(u_n)rdr\to 2 a
\]
%$u_n(\bar{\rho}_{k+1,n})/\mu_{1,n}\to \delta_{k+1}$, $u_n(\rho_{k+1,n})/\mu_{1,n}\to \delta_{k+1}$ and $(\mu_{1,n}-u_n(\bar{\rho}_{k+1,n}))g'(\mu_{1,n})/g(\mu_{1,n})\to\log{(1/\eta_*)}$ and $(\mu_{1,n}-u_n(\rho_{k+1,n}))g'(\mu_{1,n})/g(\mu_{1,n})\to\log{(1/\eta_*)}$
 as $n\to \infty$ up to a subsequence.   Moreover,  for the sequences $(r_n)=(\rho_{k+1,n}),(\bar{\rho}_{k+1,n})$, we have that 
\[
u_n(r_n)=
\begin{cases}
\displaystyle\left(\frac{\delta_{k+1}}{\delta_k}+o(1)\right) \mu_{k,n}&\text{ if $q>1$},\\
\displaystyle\mu_{k,n}-\left(\log{\frac{\eta_{k}}{\eta_{k+1}}}+o(1)\right) \frac{g(\mu_{k,n})}{g'(\mu_{k,n})}&\text{ if $q=1$}, 
\end{cases}
\]
$r_n\to 0$, and $\phi_n(r_n)\to0$ as $n\to \infty$. Finally, we obtain that%$\phi_n(\bar{\rho}_{k,n})\to0$, $\phi_n(\rho_{k,n})\to0$, 
\[
g'(\mu_{n})\int_{\rho_{k,n}}^{\bar{\rho}_{k+1,n}}\la_n hf(u_n)rdr\to 0
\]
as $n\to \infty$ up to a subsequence.
\end{lemma}
\begin{proof} From Lemma \ref{lem:c13} with Lemma \ref{lem:g1}, we get sequences $(\e_n),(R_n)$ of values such that $\e_n\to0$, $R_n\to \infty$,  $g(\mu_{k+1,n})^{-1}\log{\e_n}\to 0$, $g(\mu_{k+1,n})^{-1}\log{R_n}\to 0$, $\|z_{k+1,n}-z_{k+1}\|_{C^2([\e_n,R_n])}\to0$, $\sup_{r\in[\e_n,R_n]}|rz_{k+1,n}'(r)-rz_{k+1}'(r)|\to0$, %$\sup_{[\e_n,R_n]}\left|h(\ga_{k,n}r)/h(r_{k,n})-1\right|\to0,$ $\sup_{\e_n/R_n}\left|g(u_n(\ga_{k,n}r))/g(\mu_{k,n})-1\right|\to0,$ $\sup_{[\e_n,R_n]}\left|f(u_n(\ga_{k,n}r))/f(\mu_{k,n})-e^{z_k(r)}\right|\to0$, 
\[
\sup_{r\in [\e_n,R_n]}\left|\frac{h(\ga_{k+1,n}r)}{h(r_{k+1,n})}\frac{f(u_n(\ga_{k+1,n}r)}{f(u_n(r_{k+1,n}))}-e^{z_{k+1}(r)}\right|\to0,
\]
\[
\sup_{r\in [\e_n,R_n]}\left|\frac{h(\ga_{k+1,n}r)}{h(r_{k+1,n})}\frac{f'(u_n(\ga_{k+1,n}r))}{f'(u_n(r_{k+1,n}))}-e^{z_{k+1}(r)}\right|\to0,
\]
\[
\sup_{r\in [\e_n,R_n]}\left|\frac{\phi_n(\ga_{k+1,n}r)}{r^2e^{z_{k+1}(r)}}-1\right|\to0,
\]
\[
g'(\mu_{k+1,n})\int_{\e_n \ga_{k+1,n}}^{R_n\ga_{k+1,n}}\la_n hf(u_n)rdr\to 2a,
\]
and
\[
\int_{\e_n \ga_{k+1,n}}^{R_n \ga_{k+1,n}}\la_n hg'(u_n)f(u_n)rdr\to 2a,
\]
as $n\to \infty$. Then setting $\bar{\rho}_{k+1,n}=\e_n \ga_{k+1,n}$ and $\rho_{k+1,n}=R_n \ga_{k+1,n}$ for all $n\in \mathbb{N}$, we obtain the first conclusions. 

Next we check the second ones. We see
\[
u_n(\bar{\rho}_{k+1,n})=\mu_{k+1,n}+\frac{z_{k+1}(\e_n)+o(1)}{g(\mu_{k+1,n})}\frac{g(\mu_{k+1,n})}{g'(\mu_{k+1,n})}
\] 
and
\[
u_n(\rho_{k+1,n})=\mu_{k+1,n}+\frac{z_{k+1}(R_n)+o(1)}{g(\mu_{k+1,n})}\frac{g(\mu_{k+1,n})}{g'(\mu_{k+1,n})}
\] 
as $n\to \infty$. Hence for the sequences $(r_n)=(\rho_{k+1,n}),(\bar{\rho}_{k+1,n})$, we deduce that  if $q>1$, then $u_n(r_n)/\mu_{k,n}\to \delta_{k+1}/\delta_k$ as $n\to \infty$ by our choice of $(\e_n)$ and $(R_n)$ and \eqref{g2} and if $q=1$, then
\[
\begin{split}
u_n(r_n)&=\mu_{k+1,n}+o(1)\cdot\frac{g(\mu_{k+1,n})}{g'(\mu_{k+1,n})}\\
&=\mu_{k,n}-\left(\log{\frac{\eta_k}{\eta_{k+1}}}+o(1)\right)\frac{g(\mu_{k,n})}{g'(\mu_{k,n})}
\end{split}
\]
by Lemma \ref{lem:g4}. It follows from Lemma \ref{lem:gl} that $\rho_{k+1,n}\to0$ as $n\to \infty$. Moreover, from our choice of $(\e_n)$ and $(R_n)$, we see that
\[
\begin{split}
\phi_n(\bar{\rho}_{k,n})%&=\e_n^2\frac{h(\ga_{k+1,n}\e_n)}{h(r_{k+1,n})}\frac{f'(u_n(\ga_{k+1,n}\e_n))}{f'(\mu_{k+1,n})}
=(1+o(1))\frac{2a^2b \e_n^a}{(1+b\e_n^a)^2}\to0
\end{split}
\]
and similarly, $\phi_n(\rho_{k,n})\to0$ as $n\to \infty$.  This completes the second assertions. 

Finally, let $(\sigma_n)$ and $(\tau_n)$ be the sequences obtained in Lemmas \ref{lem:c10} and \ref{lem:c12} respectively. Then, if $\bar{\rho}_{k+1,n}\le \sigma_n$ for all large $n\in \mathbb{N}$, we clearly get from Lemma \ref{lem:c10} with Lemmas  \ref{lem:g3} and \ref{lem:g4} that 
\[
\lim_{n\to \infty}g'(\mu_n)\int_{\rho_{k,n}}^{\bar{\rho}_{k+1,n}}\la_n hf(u_n)rdr= 0.
\]
In the other case, we get $\sigma_n< \bar{\rho}_{k+1,n}<\tau_n$ for all $n\in \mathbb{N}$ up to a subsequence since $\bar{\rho}_{k+1,n}/r_{k+1,n}=\e_n \ga_{k+1,n}/r_{k+1,n}\to0$ and $r_{k+1,n}/\tau_n\to c_0$ as $n\to \infty$ for some value $c_0\not=0$  by our choice of $(r_{k+1,n})$ in the proof of Lemma \ref{lem:c13}. Recall that there exist constants $0<c_1<c_2<2$ such that $c_1\le\psi_n(r)\le c_2$ for all $r\in[\sigma_n,\tau_n]$ and all large $n\in \mathbb{N}$ by \eqref{eq:infps} and \eqref{eq:supps}. Then, we obtain from Lemma \ref{lem:D1} that, as $n\to \infty$, $\sup_{r\in[\sigma_n,\bar{\rho}_{k+1,n}]}\phi_n(r)=\phi_n(\bar{\rho}_{k+1,n})\to0$. Consequently, Lemma \ref{lem:D2} shows that   
\[
\lim_{n\to \infty}g'(\mu_{n})\int_{\sigma_n}^{\bar{\rho}_{k+1,n}}\la_n hf(u_n)rdr= 0.
\]
This with Lemma \ref{lem:c10} completes the final conclusion. We finish the proof.  
\end{proof}
Let us determine the number $a$.
\begin{lemma}\label{lem:c15} Let $a$ be the number obtained in Lemma \ref{lem:c13}. Then we have that $a=a_{k+1}$.
\end{lemma}
\begin{proof}  Using Lemma \ref{lem:c13} and noting \eqref{as:en} with Lemma \ref{lem:c14} and Lemmas \ref{lem:g3} and \ref{lem:g4}, we calculate
\[
\begin{split}
2&=\lim_{n\to \infty}\psi_n(r_{k+1,n})\\
%&=\lim_{n\to \infty}\left(
%\frac{g'(\mu_{k+1,n})}{g'(\mu_{n})}g'(\mu_{n})
%g'(\mu_{k+1,n})\int_0^{\sigma_n}\la_n hf(u_n)rdr+g'(\mu_{k+1,n})\int_{\sigma_n}^{r_{k+1,n}}\la_n hf(u_n)rdr\right)\\
&=\lim_{n\to \infty}\Bigg(
%\frac{g'(\mu_{k+1,n})}{g'(\mu_{n})}g'(\mu_{n})
\frac{g'(\mu_{k+1,n})}{g'(\mu_{k,n})}\frac{g'(\mu_{k,n})}{g'(\mu_{n})}g'(\mu_n)\int_0^{\rho_{k,n}}\la_n hf(u_n)rdr+o(1)\\%+g'(\mu_{k+1,n})\int_{\rho_{k,n}}^{\bar{\rho}_{k+1,n}}\la_n hf(u_n)rdr\\
&\ \ \ \ \ \ \  \ \ \ \ \ \ \ \ \ \ \ \ \ \ \ \ \ \ \ \ \ \ +\int_{\bar{\rho}_{k+1,n}/\ga_{k+1,n}}^{r_{k+1,n}/\ga_{k+1,n}} \frac{h(\ga_{k+1,n}r)}{h(r_{k+1,n})}\frac{f(u_n(\ga_{k+1,n}r))}{f(\mu_{k+1,n})}rdr\Bigg)\\
&=%\begin{cases}
%\delta_{k+1}^{p-1} \sum_{i=1}^k \frac{2a_i}{\delta_i^{p-1}}+a\text{ if }q>1,\\
\tilde{\eta}_{k+1} \sum_{i=1}^k \frac{2a_i}{\tilde{\eta}_i}+a%\text{ if }q=1.
%\end{cases}
\\
&=2-a_{k+1}+a
\end{split}
\]
by Lemma \ref{lem:b3}, \eqref{b2}, and \eqref{b4}. This completes the proof. 
\end{proof}
Lastly, we deduce the pointwise estimate for the next step. 
\begin{lemma}\label{lem:c16} We obtain that
\[
z_{k+1,n}(r)\le -(2+a_{k+1}+o(1))\log{r}
\]
for all $r\in [\rho_{k+1,n}/\ga_{k+1,n},1/\ga_{k+1,n}]$ and all large $n\in \mathbb{N}$ where $o(1)\to0$ as $n\to \infty$ uniformly for all $r$ in the interval.
\end{lemma}
\begin{proof} Set $R_n=\rho_{k+1,n}/\ga_{k+1,n}$ for all large $n\in \mathbb{N}$. Noting the fact that  $rz_{k+1,n}'(r)$ is nonincreasing  on $[R_n,1/\ga_{k+1,n}]$ by \eqref{zkn} and also using Lemma \ref{lem:c14}, we have  any $r\in[R_n,1/\ga_{k+1,n}]$  that 
\[
\begin{split}
rz_{k+1,n}'(r)&\le  R_nz_{k+1}'(R_n)+o(1)=-(2+a_{k+1}+o(1))
\end{split}
\] 
as $n\to \infty$ . Integrating over $[R_n,r]$, we obtain
\[
\begin{split}
z_{k+1,n}(r)&\le z_{k+1}(R_n)-(2+a_{k+1}+o(1))\log{\frac r{R_n}}+o(1)\\
&\le -(2+a_{k+1}+o(1))\log{r}
\end{split}
\]
as $n\to \infty$. We finish the proof. 
\end{proof}
\section{Infinite concentration estimates}\label{sec:prf}
In this final section, we collect all the results in the previous sections and complete the infinite concentration estimates  in Theorems \ref{th:sup1} and \ref{th:sup2}. 
\subsection{Proofs of Theorems \ref{th:sup1} and \ref{th:sup2}}
We first show Theorem \ref{th:sup1}. 
\begin{proof}[Proof of Theorem \ref{th:sup1}] We first assume that for any number $k\in \mathbb{N}$, the assumption (A$_k$) holds true. Then for the sequences $(r_n)=(r_{k+1,n}),(\rho_{k+1,n})$, and $(\bar{\rho}_{k+1,n})$ in Lemmas \ref{lem:c13} and \ref{lem:c14}, we have  that if $q>1$, then $u_n(r_n)/\mu_n%=\left(\frac{\delta_{k+1}}{\delta_k}+o(1)\right)\mu_{k,n}
\to \delta_{k+1}$ and if $q=1$, then from Lemma \ref{lem:g4},
\[
\begin{split}
u_n(r_n)%&=\mu_{k,n}-\left(\log{\frac{\eta_k}{\eta_{k+1}}}+o(1)\right)\frac{g(\mu_{k,n})}{g'(\mu_{,k,n})}\\
&=\mu_n-\left(\log{\frac{1}{\eta_{k+1}}}+o(1)\right)\frac{g(\mu_{n})}{g'(\mu_{,n})}
\end{split}
\]
as $n\to \infty$. Moreover, by \eqref{as:en} and Lemma \ref{lem:c14} with Lemmas \ref{lem:c15}, \ref{lem:g3}, and \ref{lem:g4}, we have that
\[
\lim_{n\to \infty}g'(\mu_n)\int_0^{\rho_{k+1,n}}\la_n hf(u_n)rdr=\sum_{i=1}^{k+1}\frac{2a_i}{\tilde{\eta}_i}
\]
for all $q\ge1$. 

Let us start main argument of the proof. We choose the desired sequences by induction.  We first take the sequences $(r_{1,n}),(\ga_{1,n}),(\rho_{1,n})\subset(0,1)$ of values and $(z_{1,n})$ of functions from Lemmas \ref{lem:c5}, \ref{lem:c6}, and \ref{lem:c7}. Moreover, we define $\bar{\rho}_{1,n}=0$ for all $n\in \mathbb{N}$. Then, these sequences complete all the assertions of the theorem,  except for \eqref{eq:sup02}, for $k=1$. Next, we observe that they verify all the conditions in  (A$_1$) in Section \ref{sec:ic}. Then we can take sequences $(r_{2,n}),(\ga_{2,n}),(\bar{\rho}_{2,n}),(\rho_{2,n})\subset(0,1)$ of values and $(z_{2,n})$ of functions as in Lemmas \ref{lem:c13} and \ref{lem:c14}. Consequently, we complete \eqref{eq:sup02} for $k=1$. Moreover, noting also Lemma \ref{lem:c15} and the remark above, we confirm that all the conclusions of the theorem, except for \eqref{eq:sup02}, are satisfied  for $k=2$. In addition, using also Lemma \ref{lem:c16}, we check that they ensure  all the conditions in (A$_2$). Now we assume that this procedure is repeated $l-1$ times for some $l\in \mathbb{N}$ with $l\ge2$. That is, we suppose that we find   sequences $(r_{k,n}),(\ga_{k,n}),(\bar{\rho}_{k,n}),(\rho_{k,n})\subset(0,1)$ of values and $(z_{k,n})$ of functions for all $k=1,\cdots,l$ which satisfy all the assertions in the theorem for all $k=1,\cdots,l$ except for \eqref{eq:sup02} with $k=l$ and all the conditions  in (A$_k$) for all $k=1,\cdots,l$. Then thanks to (A$_{l}$), we can choose sequences $(r_{l+1,n}),(\ga_{l+1,n}),(\bar{\rho}_{l+1,n}),(\rho_{l+1,n})\subset(0,1)$ of values and $(z_{l+1,n})$ of functions as in Lemmas \ref{lem:c13} and \ref{lem:c14}. Then similarly to the previous discussion, we confirm \eqref{eq:sup02} for $k=l$. Moreover, we check all the assertions in the theorem except for \eqref{eq:sup02} for $k=l+1$ and all the conditions in (A$_{l+1}$). In this way, we can repeat the argument  arbitrary times and obtain the desired families of the sequences which satisfy all the assertions of the theorem.  This finishes the proof.   
\end{proof}
\begin{remark}\label{rmk:osc} We remark that from the proof above with  Lemmas \ref{lem:c7} and \ref{lem:c14}, we have that for all $k\in \mathbb{N}$,
\[
 \lim_{n\to \infty}\frac{\log{\frac{\rho_{k,n}}{\ga_{k,n}}}}{g(\mu_{k,n})}=0,\ \ \ \lim_{n\to \infty}\frac{\log{\frac{\bar{\rho}_{k,n}}{\ga_{k,n}}}}{g(\mu_{k,n})}=0,
\]
where we assumed $k\not=1$ in the latter formula, and 
\[
\lim_{n\to \infty}\sup_{r\in [\bar{\rho}_{k,n}/\ga_{k,n}, \rho_{k,n}/\ga_{k,n}]}\left|\frac{\phi_n(\ga_{k,n} r)}{r^2e^{z_k(r)}}-1\right|=0.
\]
We will use these fact in our oscillation analysis in \cite{N3}. 
\end{remark}
For the proof of Theorem \ref{th:sup2}, we give some preliminary lemmas. In the following, we assume that for all $k\in \mathbb{N}$, $(r_{k,n}),(\ga_{k,n}),(\bar{\rho}_{k,n}),(\rho_{k,n})\subset(0,1)$,  and $(z_{k,n})$ are the sequences in Theorem \ref{th:sup1} and put $\mu_{k,n}=u_n(r_{k,n})$ for all $n\in \mathbb{N}$ as before. Note that, as discussed in the proof of the theorem above, we may suppose that they satisfy all the assertions of (A$_k$) for all $k\in \mathbb{N}$. We first prove the next lemma. 
\begin{lemma}\label{lem:c18} For any numbers $k,l\in \mathbb{N}$, we have that
%\[
%\frac{\log{\frac{1}{r_{k,n}}}}{g(\mu_{k,n})}= \frac{1}{2}\left(1-\frac{\log{\frac1{\la_n}}}{g(\mu_{k,n})}\right)+o(1)
%\]
%and 
\[
\log{\frac{r_{l,n}}{r_{k,n}}}=\frac{g(\mu_{k,n})}2\left(1-\frac{\eta_l}{\eta_k}+o(1)\right)
%\frac{g(\mu_{k,n})}2 \left(1-\frac{g(\mu_{l,n})}{g(\mu_{k,n})}+o(1)\right)
\]
as $n\to \infty$.
\end{lemma}
\begin{proof} For $i=k,l$, from \eqref{eq:sup0} and \eqref{eq:lgg}, we get
\[%begin{equation}\label{eq:ff1}
2\log{\frac1{r_{i,n}}}= g(\mu_{i,n})(1+o(1))-\log{\frac1{\la_n}}
\]%end{equation}
as $n\to \infty$. %This proves the former formula. 
Combining this formula for $i=k$ with that for $i=l$, we have 
\[
\log{\frac{r_{l,n}}{r_{k,n}}}=\frac{g(\mu_{k,n})}2\left(1-\frac{g(\mu_{l,n})}{g(\mu_{k,n})}+o(1)\right)
%\frac{g(\mu_{k,n})}2\left(1-\frac{\eta_l}{\eta_k}+o(1)\right)
\]
as $n\to \infty$. Then  Lemmas \ref{lem:g3} and \ref{lem:g4} prove the desired formula. This completes the proof. 
\end{proof}
Then, we get the next estimate for the scaling parameter. 
\begin{lemma}\label{lem:c18} For any $k\in \mathbb{N}$, we have that
\[
\liminf_{n\to \infty}\frac{\log{\frac1{\ga_{k,n}}}}{g(\mu_{k,n})}\ge\frac12.
\]
\end{lemma}
\begin{proof} 
Assume $k\in \mathbb{N}$. Fix any value $\e\in(0,1)$. Noting Lemma \ref{lem:b2}, choose a number $l\in \mathbb{N}$ so that $\eta_{l}/\eta_k<\e$. Then from the previous lemma and Lemma \ref{eq:sup0}, we get that 
\[
\frac{\log{\frac1{\ga_{k,n}}}}{g(\mu_{k,n})}\ge \frac{\log{\frac{r_{l,n}}{r_{k,n}}}+\log{\frac{r_{k,n}}{\ga_{k,n}}}}{g(\mu_{k,n})}\ge \frac{1}{2}\left(1-\e\right)
\]
for all large $n\in \mathbb{N}$. Since $\e\in (0,1)$ is arbitrary, we prove 
\[%begin{equation}\label{eq:x1}
\liminf_{n\to \infty} \frac{\log{\frac1{\ga_{k,n}}}}{g(\mu_{k,n})}\ge\frac12.
\]%end{equation}
This completes the proof. 
\end{proof}
The next lemma will give an improved pointwise estimate. 
\begin{lemma} Fix any $k\in \mathbb{N}$. If $q>1$, choose any value $\e\in(0,1)$ and if $q=1$, select any number $M>0$. Then, there exists a sequence $(r_n)\subset (\rho_{k,n},1)$ such that 
\[
\liminf_{n\to \infty}g'(\mu_{k,n})\int_0^{r_n}\la_n hf(u_n)rdr > \begin{cases}
2p&\text{ if }q>1,\\
M&\text{ if }q=1,
\end{cases}
\]
and 
\[
z_{k,n}(r_n/\ga_{k,n})\le
\begin{cases}  -2p(1-\e)\log{(r_n/\ga_{k,n})}&\text{ if }q>1,\\
-M\log{(r_n/\ga_{k,n})}&\text{ if }q=1,
\end{cases}
\]
for all large $n\in \mathbb{N}$. 
\end{lemma}
\begin{proof} Take any value $\eta\in(0,\eta_{k+1}/\eta_k)$, which will be chosen to be smaller if necessary, and a sequence $(t_n)\subset (\rho_{k,n},1)$ as in Lemma \ref{lem:c11}.  In particular, we have that 
\[
u_n(t_n)=
\begin{cases}
\displaystyle\eta^{1/p} \mu_{k,n}&\text{ if $q>1$},\\
\displaystyle\mu_{k,n}-\log{\frac{1}{\eta}} \frac{g(\mu_{k,n})}{g'(\mu_{k,n})}&\text{ if $q=1$}, 
\end{cases}
\]
for all $n\in \mathbb{N}$ and, from Lemmas \ref{lem:g3} and \ref{lem:g4}, 
\[
\liminf_{n\to \infty}g'(\mu_{k,n})\int_0^{t_n}\la_n hf(u_n)rdr\ge
\begin{cases} \displaystyle \frac{2p\left(\frac{\delta_{k+1}}{\delta_k}-\eta^{1/p}\right)}{\left(\frac{\delta_{k+1}}{\delta_k}\right)^p-\eta}&\text{ if }q>1\vspace{0.3cm},\\
\displaystyle \frac{2\left(\log{\frac1{\eta}}-\log{\frac{\eta_{k}}{\eta_{k+1}}}\right)}{\frac{\eta_{k+1}}{\eta_{k}}-\eta}&\text{ if }q=1.
\end{cases}
\]
Note that choosing $\eta>0$ smaller if necessary, we have that the right-hand side is strictly greater than $2p$ if $q>1$ and $M$ if $q=1$.  This proves the former assertion. Next we calculate from \eqref{g2} if $q>1$ and \eqref{eq:es2} with Lemmas \ref{lem:g3} and \ref{lem:g4} that
\[
\begin{split}
z_{k,n}(t_n/\ga_{k,n})&=-g(\mu_{k,n})(\mu_{k,n}-u_n(t_n))\frac{g'(\mu_{k,n})}{g(\mu_{k,n})}\\
%&\le -\frac{2(\mu_{k,n}-u_n(t_n))\frac{g'(\mu_{k,n})}{g(\mu_{k,n})}}{1-\eta+o(1)}\log{\frac{t_n}{\ga_{k,n}}}\\
&\le 
\begin{cases}
-\displaystyle2p\left(\frac{1-\eta^{1/p}}{1-\eta}+o(1)\right)\log{\frac{t_n}{\ga_{k,n}}}&\text{ if }q>1,\vspace{0.2cm}\\
-\displaystyle2\left(\frac{\log{\frac1{\eta}}}{1-\eta}+o(1)\right)\log{\frac{t_n}{\ga_{k,n}}}&\text{ if }q=1
\end{cases}\\
&\le 
\begin{cases}
-\displaystyle 2p(1-\e)\log{\frac{t_n}{\ga_{k,n}}}&\text{ if }q>1,\\
-\displaystyle M\log{\frac{t_n}{\ga_{k,n}}}&\text{ if }q=1,
\end{cases}
\end{split}
\] 
for all large $n\in \mathbb{N}$ by taking $\eta>0$ smaller if necessary.   This completes the proof. 
\end{proof}
Then, we prove an improved pointwise estimate. 
\begin{lemma} Fix any $k\in \mathbb{N}$. If $q>1$, take any value $\e\in(0,1)$ and if $q=1$, choose any constant $M>0$. Then,   there exists a sequence $(r_n)\subset (\rho_{k,n},1)$ such that 
\[
z_{k,n}(r)\le
\begin{cases}
-2p(1-\e)\log{r}&\text{ if }q>1,\\
-M \log{r}&\text{ if }q=1
\end{cases}
\]
for all $r\in[r_n/\ga_{k,n},1/\ga_{k,n}]$ and all large $n\in \mathbb{N}$. 
\end{lemma}
\begin{proof} For any values $\e$ and $M$ chosen as in the assumption, take the sequence $(r_n)\subset (\rho_{k,n},1)$ as in the previous lemma. Then recalling the monotonicity of $rz_{k,n}'(r)$ by \eqref{zkn}, we have that
\[
rz_{k,n}'(r)\le (r_n/\ga_{k,n})z_{k,n}'(r_n/\ga_{k,n})\le 
\begin{cases} -2p&\text{ if }q>1\vspace{0.3cm},\\
\displaystyle -M&\text{ if }q=1,
\end{cases}
\]  
for all $r\in[r_n/\ga_{k,n},1/\ga_{k,n}]$ and  all large $n\in \mathbb{N}$. Integrating over $[r_n/\ga_{k,n},r]$, we get
\[
\begin{split}
z_{k,n}(r)&\le  \begin{cases}\displaystyle z_{k,n}(r_n/\ga_{k,n})-(2p-\e)\log{\frac{r}{r_n/\ga_{k,n}}}&\text{ if }q>1\vspace{0.3cm},\\
\displaystyle z_{k,n}(r_n/\ga_{k,n})-M\log{\frac{r}{r_n/\ga_{k,n}}}&\text{ if }q=1,
\end{cases}\\
&\le \begin{cases}
-2p(1-\e)\log{r}&\text{ if }q>1,\\
-M \log{r}&\text{ if }q=1,
\end{cases}
\end{split}
\]
for all large $n\in \mathbb{N}$ by the previous lemma.  We finish the proof.
\end{proof}
This proves the next lemma. 
\begin{lemma}\label{lem:c19} For any $k\in \mathbb{N}$, we have that
\[
\lim_{n\to \infty}\frac{\log{\frac1{\ga_{k,n}}}}{\mu_{k,n}g'(\mu_{k,n})}=\frac1{2p}
\]
where regarded $1/\infty=0$. Moreover, additionally assuming (H2) if $q=1$, we get 
\[
\lim_{n\to \infty}\frac{\log{\frac1{\ga_{k,n}}}}{g(\mu_{k,n})}=\frac12
\]
for all $q\in [1,2)$. 
\end{lemma}
\begin{proof} Take any value $\e\in(0,1)$ if $q>1$ and $M>0$ if $q=1$. From the previous lemma, we get
\[
-\mu_{k,n}g'(\mu_{k,n})=z_{k,n}(1/\ga_{k,n})\le 
\begin{cases}
-2p(1-\e)\log{(1/\ga_{k,n})}&\text{ if }q>1,\\
-M \log{(1/\ga_{k,n})}&\text{ if }q=1,
\end{cases}
\]
for all large $n\in \mathbb{N}$. It follows that 
\[
\frac{\log{\frac1{\ga_{k,n}}}}{\mu_{k,n}g'(\mu_{k,n})}\le 
\begin{cases}
\displaystyle\frac1{2p(1-\e)}&\text{ if }q>1,\vspace{0.2cm}\\
\displaystyle\frac1M&\text{ if }q=1,
\end{cases}
\]
for all large $n\in \mathbb{N}$. Since $\e\in(0,1)$ and $M>0$ are arbitrary, we prove that 
\[
\limsup_{n\to \infty}\frac{\log{\frac1{\ga_{k,n}}}}{\mu_{k,n}g'(\mu_{k,n})}\le\frac1{2p}
\]
for all $q\ge1$ where $1/\infty=0$. Hence, we prove the former assertion by Lemma \ref{lem:c18} with \eqref{g2}. Clearly this is equivalent to the latter one if $q>1$ by \eqref{g2}. If $q=1$, from \eqref{eq:lgg} and Lemma \ref{lem:kap}, we obtain
\[
\limsup_{n\to \infty}\frac{\log{\frac{1}{\ga_{k,n}}}}{g(\mu_{k,n})}= \frac12\limsup_{n\to \infty}\frac{g(\mu_{k,n})(1+o(1))+\log{\la_n}}{g(\mu_{k,n})}\le \frac12.
\]  
Hence Lemma \ref{lem:c18} proves the desired formula. We finish the proof. 
\end{proof}
Lastly, we complete the final theorem. 
\begin{proof}[Proof of Theorem \ref{th:sup2}] Choose a sequence $(r_n)$ as in the theorem. Then for any $k\in \mathbb{N}$, we have that $u_n(\rho_{k,n})>u_n(r_n)$, which implies $\rho_{k,n}<r_n$, for all large $n\in \mathbb{N}$. It follows from Theorem \ref{th:sup1} that
\[
\int_0^{r_n}\la_n hg'(u_n)f(u_n)rdr\ge\int_0^{\rho_{k,n}}\la_n hg'(u_n)f(u_n)rdr=\sum_{i=1}^k (2a_i)+o(1)
\] 
as $n\to \infty$. Since $k$ is arbitrary, noting Lemma \ref{lem:b2}, we prove \eqref{eq:sup1}. Here, we may choose the sequence $(r_n)$ above so that $u_n(r_n)\to \infty$ as $n\to \infty$. Then \eqref{lem:gl} implies $r_n\to 0$ and thus, from \eqref{id2}, we have that for all $r\in (0,1)$, 
\[
g'(\mu_n) u_n(r)\ge g'(\mu_n)\int_0^r \la_n hf(u_n)rdr\log{\frac1r}\to \infty 
\]
as $n\to \infty$. This proves \eqref{eq:sup4}. 
Assume in addition (H2) if $q=1$. Then from the definition of $(\ga_{k,n})$ and the latter conclusion of Lemma \ref{lem:c19} with  \eqref{eq:lgg}, we have that
\[
\lim_{n\to \infty}\frac{\log{\frac1{\la_n}}}{g(\mu_{k,n})}=0.
\]
This with Lemmas \eqref{lem:g3} and  \eqref{lem:g4} proves  \eqref{eq:sup2}. Finally, from \eqref{eq:sup0}, \eqref{eq:sup2}, \eqref{eq:lgg}, and  Lemmas \ref{lem:g3} and \ref{lem:g4}, we derive \eqref{eq:sup3}.  We complete the proof. 
\end{proof}
\appendix
\section{Blow-up formulas for the subcritical case}\label{sec:sub}
In this appendix, we give proofs of the basic blow up estimates \eqref{eq:sub1}, \eqref{eq:sub3}, and \eqref{eq:sub2} for the subcritical case, that is, we prove the next theorem. 
\begin{theorem}\label{apth:sub} Suppose as in Theorem \ref{th:0} and $q>2$. Then we have that
\begin{equation}\label{aeq:sub1}
\lim_{n\to \infty}g'(\mu_n)\int_0^1\la_n hf(u_n)rdr=4=\lim_{n\to \infty}\int_0^1\la_n hf'(u_n)rdr,
\end{equation}
\begin{equation}\label{aeq:sub3}
\lim_{n\to \infty}g'(\mu_n)u_n(r)=4\log{\frac1r}\text{ in }C^2_{\text{loc}}((0,1]),
\end{equation}
and 
\begin{equation}\label{aeq:sub2}
\lim_{n\to \infty}\frac{\log{\frac1{\la_n}}}{g(\mu_n)}=\frac{2-p}{2}.
\end{equation}
\end{theorem}
%\begin{theorem}\label{th:sub2}
%\end{theorem}
For the proof, we give  some preliminary lemmas. In the following, we always assume (H1) without further comments. We begin with the proof of a property of our nonlinearity. 
\begin{lemma}\label{lem:g30} Assume $q>1$. Then, for any $\e>0$, there exists a value $t_\e>0$ such that
\[
\left(\frac{t}{s}\right)^{p-\e}\le\frac{g(t)}{g(s)}\le \left(\frac{t}{s}\right)^{p+\e}
\]
for all $t>s\ge t_\e$.
\end{lemma}
\begin{proof}
For any $t>s>0$, we have that 
\[
\log{\frac{g(t)}{g(s)}}=\int_{s}^{t}\frac{\tau g'(\tau)}{g(\tau)}\frac1\tau d\tau%\le \left(p+\frac{\e}{\log{\frac1{\delta}}}\right)\log{\frac1{\delta}} %
=p\int_{s}^{t}\frac1\tau d\tau+\int_{s}^{t}\left(\frac{\tau g'(\tau)}{g(\tau)}-p\right)\frac1\tau d\tau.
\]
Then from \eqref{g2}, for any $\e>0$, there exists a value $t_\e>0$ such that   
\[
(p-\e)\log{\frac{t}{s}}\le\log{\frac{g(t)}{g(s)}}\le (p+\e)\log{\frac{t}{s}}
\]
for all $t>s\ge t_\e$. This proves the desired estimate. We finish the proof.  
\end{proof}
Let $(\ga_{0,n})$ be the sequence in Theorem \ref{th:0}. We give the next asymptotic estimate. 
\begin{lemma}\label{lem:c3} Suppose  $q\ge 1$. Then we have that
\[
\limsup_{n\to \infty}\frac{\log{\frac1{\ga_{0,n}}}}{g(\mu_n)}\le \frac{p}{4}.
\] 
\end{lemma}
\begin{proof} From Lemma \ref{lem:c2}, we get
\[
-g'(\mu_n)\mu_n=z_{0,n}\left(\frac1{\ga_{0,n}}\right)\le -(4+o(1))\log{\frac1{\ga_{0,n}}}
\]
as $n\to \infty$. This with \eqref{g2} proves the desired formula. We complete the proof. 
\end{proof}
We next give the limit energy.  
\begin{lemma}\label{lem:c4} Assume $q>2$. Then we get that
\[
\lim_{n\to \infty}g'(\mu_n)\int_0^1\la_nhf(u_n)rdr= 4.
\]
\end{lemma}
\begin{proof} From Lemma \ref{lem:c3}, the definition of $(\ga_{0,n})$, \eqref{g2}, and \eqref{eq:lgg}, we deduce that 
\[
\la_n\le e^{-(2-p+o(1))g(\mu_n)/2}
\]
as $n\to \infty$. Noting Lemma \ref{lem:g3}, we choose a sequence $(r_n)\subset (\rho_{0,n},1)$ so that $g(u_n(r_n))=(2-p)g(\mu_n)/4+o(1)$ as  $n\to \infty$. It follows from Lemma \ref{lem:monog}  that
\[
g'(\mu_n)\int_{r_n}^1\la_n hf(u_n)rdr\le\max_{0\le r\le1}h(r) e^{-(2-p+o(1))g(\mu_n)/2} (e^{g(u_n(r_n))}+O(1))\to0
\] 
as $n\to \infty$. Hence in view of Theorem \ref{th:0}, it suffices to show 
\[
g'(\mu_n)\int_{\rho_{0,n}}^{r_n}\la_n hf(u_n)rdr\to0
\]
as $n\to \infty$. Then choose a value $\e>0$ so that $p-\e>1$. Using Lemma \ref{lem:g30}, \eqref{g2}, and Lemma \ref{lem:c2}, we find a constant $C>0$ such that   
\[
\begin{split}
g'(\mu_{n})\int_{\rho_{0,n}}^{r_n}\la_n hf(u_n)rdr& \le C\int_{\rho_{0,n}/\ga_{0,n}}^{r_n/\ga_{0,n}}e^{g(\mu_n)\left\{\frac{g\left(\mu_n+\frac{z_{0,n}(r)}{g'(\mu_n)}\right)}{g(\mu_n)}-1\right\}}rdr\\
&\le  C\int_{\rho_{0,n}/\ga_{0,n}}^{r_n/\ga_{0,n}}e^{g(\mu_n)\left\{\left(1+\frac{z_{0,n}(r)}{\mu_ng'(\mu_n)}\right)^{p-\e}-1\right\}}rdr\\
&\le  %C\int_{\rho_{0,n}/\ga_{0,n}}^{r_n/\ga_{0,n}}e^{\frac{z_{0,n}(r)}{p+o(1)}}rdr\le
 C\int_{\rho_{0,n}/\ga_{0,n}}^{r_n/\ga_{0,n}}e^{-\frac{4+o(1)}{p}\log{r}}rdr\\
&\to0
\end{split}
\]
as $n\to \infty$ since $p<2$ and $\rho_{0,n}/\ga_{0,n}\to\infty$ as $n\to \infty$. This finishes the proof. 
\end{proof}
Consequently, we get the exact asymptotic formulas for the sequence $(\ga_{0,n})$.
\begin{lemma}\label{lem:c41} If $q>2$, then we have that
\[
\lim_{n\to \infty}\frac{\log{\frac1{\ga_{0,n}}}}{g(\mu_n)}= \frac p4.
\]
\end{lemma}
\begin{proof} From Lemma \ref{lem:c3}, it suffices to show that
\[
\liminf_{n\to \infty}\frac{\log{\frac1{\ga_{0,n}}}}{g(\mu_n)}\ge \frac p4.
\]
Then, using \eqref{id1} and Lemma \ref{lem:c4}, we get
\[
\begin{split}
g'(\mu_n)\mu_n&=  g'(\mu_n)\int_0^1 \la_n hf(u_n)rdr\log{\frac1{\ga_{0,n}}}+g'(\mu_n)\int_0^1 \la_nhf(u_n)r\log{\frac{\ga_{0,n}}{r}}dr\\
&\le (4+o(1))\log{\frac1{\ga_{0,n}}}
\end{split}
\]
as $n\to \infty$ where we  estimated by taking large $R>1$ and using Lemma \ref{lem:c1} with Lemma \ref{lem:g1},
\[
\begin{split}
g'(\mu_n)\int_0^1 &\la_nhf(u_n)r\log{\frac{\ga_{0,n}}{r}}dr\\
&= \int_0^{1/\ga_{0,n}}\frac{h(\ga_{0,n}r)}{h(0)}\frac{f(u_n(\ga_{0,n}r))}{f(\mu_n)}r\log{\frac1r}dr\\
&=\int_0^{R}e^{z_0}r\log{\frac1r}dr+\int_R^{1/\ga_{0,n}}\frac{h(\ga_{0,n}r)}{h(0)}\frac{f(u_n(\ga_{0,n}r))}{f(\mu_n)}r\log{\frac1r}dr+o(1)\\
&<0
\end{split}
\]
for all large $n\in \mathbb{N}$. Hence noting \eqref{g2}, we prove the desired estimate. This completes the proof. 
\end{proof}
Finally we complete the proof.  
\begin{proof}[Proof of Theorem \ref{apth:sub}] First from Lemma \ref{lem:c41}, the definition of $(\ga_{0,n})$, and \eqref{eq:lgg}, we prove \eqref{aeq:sub2}. In particular, we get $\la_n\to0$ as $n\to \infty$. Then we choose $(r_n)\subset (0,1)$ so that $u_n(r_n)=t_0$ for all $n\in \mathbb{N}$ where $t_0$ is the number in our conditions on $f$. Noting the monotonicity of $g'$ noted in  Lemma \ref{lem:monog}, we see 
\[
\begin{split}
\int_0^{1} \la_n hf'(u_n)rdr&=\int_0^{r_n} \la_n hf'(u_n)rdr+o(1)\le g'(\mu_n)\int_0^1 \la_n hf(u_n)rdr+o(1)
\end{split}
\]
as $n\to \infty$. Hence Lemma \ref{lem:c4} gives
\[
\limsup_{n\to \infty}\int_0^{1} \la_n hf'(u_n)rdr\le4.
\]
Then recalling the final formula in Theorem \ref{th:0}, we prove \eqref{aeq:sub1}.  Finally, fix any $r_0\in(0,1)$. Then for all $r\in[r_0,1]$, using  \eqref{id2}, we get
\[
g'(\mu_n)\int_0^r\la_nf(u_n)sds\log{\frac1{r}}\le g'(\mu_n)u_n(r)\le g'(\mu_n)\int_0^1\la_nf(u_n)sds\log{\frac1{r}}
\]
for all $n\in \mathbb{N}$. Since  $\rho_{0,n}<r_0$ for all large $n\in \mathbb{N}$, we have from Theorem \ref{th:0} and \eqref{aeq:sub1} that
\[
g'(\mu_n)u_n(r)=(4+o(1))\log{\frac1r}
\]  
where $o(1)\to 0$ as $n\to \infty$ uniformly for all $r\in[r_0,1]$. Moreover, from \eqref{id0}, we similarly get
\[
g'(\mu_n)u_n'(r)=-(4+o(1))\frac1r
\]
as $n\to \infty$ with $o(1)$ as above. Consequently, from \eqref{pn},  the previous conclusions, and \eqref{aeq:sub2} with Lemma \ref{lem:gl} and \eqref{eq:lgg}, we obtain
\[
\begin{split}
g'(\mu_n)u_n''(r)&=-g'(\mu_n)u_n'(r)\frac1r-\la_ng'(\mu_n)h(r)f(u_n(r))\\
&=(4+o(1))\frac1{r^2}+o(1)
\end{split}
\]
as $n\to \infty$ as $n\to \infty$ with $o(1)$ as above. This completes the proof. 
\end{proof}
\subsection*{Acknowledgement} This work is supported by JSPS KAKENHI Grant Numbers 21K13813. 

\end{document}